\documentclass[reqno]{preprint}

\usepackage[utf8x]{inputenc}
\usepackage[english]{babel}
\usepackage{microtype}

\DeclareMathAlphabet{\mathbf}{OT1}{ptm}{bx}{n}
\DeclareMathAlphabet{\mathrm}{OT1}{ptm}{m}{n}
\DeclareMathAlphabet{\mathbb}{U}{jkpsyb}{m}{n}

\usepackage[marginparwidth=1in]{geometry}
\usepackage[full]{textcomp}
\usepackage[osf]{newtxtext}
\usepackage{cabin}
\usepackage{enumerate}
\usepackage{enumitem}
\usepackage{bbm}
\usepackage[dvipsnames]{xcolor}

\usepackage{orcidlink}

\usepackage[varqu,varl]{inconsolata}
\usepackage[cal=boondoxo]{mathalfa}

\usepackage{hyperref}
\usepackage{breakurl}
\usepackage{mathrsfs}
\usepackage{booktabs}
\usepackage{caption}
\usepackage{stmaryrd}

\usepackage{tikz}
\usepackage{amsmath} 
\usepackage{mhequ} 
\usepackage{mhsymb} 
\usepackage{mhenvs} 
\usepackage{microtype}
\usepackage{amssymb} 

\usepackage{centernot}

\makeatletter

\DeclareRobustCommand{\TitleEquation}[2]{\texorpdfstring{\StrLeft{\f@series}{1}[\@firstchar]$\if%
b\@firstchar\boldsymbol{#1}\else#1\fi$}{#2}}

\DeclareRobustCommand{\cev}[1]{%
  {\mathpalette\do@cev{#1}}%
}
\newcommand{\do@cev}[2]{%
  \vbox{\offinterlineskip
    \sbox\z@{$\m@th#1 x$}%
    \ialign{##\cr
      \hidewidth\reflectbox{$\m@th#1\vec{}\mkern7mu$}\hidewidth\cr
      \noalign{\kern-\ht\z@}
      $\m@th#1#2$\cr
    }%
  }%
}

\makeatother

\def\dash{\leavevmode\unskip\kern0.18em--\penalty\exhyphenpenalty\kern0.18em}
\def\slash{\leavevmode\unskip\kern0.15em/\penalty\exhyphenpenalty\kern0.15em}

\newcommand{\1}{\mathbbm{1}}

\renewcommand{\R}{\mathbf{R}}

\renewcommand{\Z}{\mathbf{Z}}
\renewcommand{\E}{\mathbb{E}}
\renewcommand{\P}{\mathbb{P}}

\renewcommand{\le}{\leqslant}
\renewcommand{\leq}{\leqslant}

\renewcommand{\geq}{\geqslant}

\renewcommand{\a}{\mathbf{a}}
\renewcommand{\b}{\mathbf{b}}

\newcommand{\q}{\mathbf{q}}
\renewcommand{\r}{\mathbf{r}}

\colorlet{darkblue}{blue!90!black}
\colorlet{darkred}{red!90!black}
\colorlet{dr}{red!90!black}
\let\comm\margincomment

\def\r{\mathrm{r}}

\def\one{\mathord{\mathbf{1}}}

\def\ts{t_{\star}^{\kappa}}
\def\tsd{t_{\star}^{\kappa, \delta}}
\def\u{\mathrm{u}}
\def\sotimes{\mathbin{\otimes_{s}}}
\newcommand{\ep}{\varepsilon}

\newcommand{\RR}{\mathbf{R}}
\newcommand{\NN}{\mathbf{N}}
\newcommand{\ZZ}{\mathbf{Z}}

\newcommand{\TT}{\mathbf{T}}

\newcommand{\EE}{\mathbb{E}}

\newcommand{\PP}{\mathbb{P}}

\newcommand{\mL}{\mathcal{L}}

\newcommand{\mC}{\mathcal{C}}

\newcommand{\mI}{\mathcal{I}}

\newcommand{\mF}{\mathcal{F}}
\newcommand{\mD}{\mathcal{D}}
\newcommand{\mX}{\mathcal{X}}

\newcommand{\mS}{\mathcal{S}}
\newcommand{\mB}{\mathcal{B}}
\newcommand{\mA}{\mathcal{A}}
\newcommand{\mH}{\mathcal{H}}

\newcommand{\mZ}{\mathcal{Z}}

\newcommand{\mf}[1]{\mathfrak{#1}}

\newcommand{\dee}{\, \mathrm{d}}

\newcommand{\ds}{\dee s}
\newcommand{\dr}{\dee r}

\newcommand{\Div}{\operatorname{div}}

\newcommand{\leqc}{\lesssim}

\renewcommand{\l}{\ell}

\newcommand{\para}{\varolessthan}
\newcommand{\rpara}{\varogreaterthan}
\newcommand{\reso}{\varodot}

\newcommand{\ve}{\varepsilon}
\newcommand{\vr}{\varrho}
\newcommand{\vt}{\vartheta}

\newcommand*{\ud}{\mathrm{\,d}}
\renewcommand{\div}{\operatorname{div}}

\newcommand{\pareq}{\mathrel{\rlap{\kern0.07em\tikz[x=0.1em,y=0.1em,baseline=0.04em] \draw[line width=0.3pt] (0,1.93) -- (4.5,0);}\varolessthan}}
\newcommand{\rpareq}{\mathrel{\rlap{\kern0.23em\tikz[x=0.1em,y=0.1em,baseline=0.04em] \draw[line width=0.3pt] (0,0) -- (4.5,1.93);}\varogreaterthan}}

\begin{document}

\title{Lower bounds on the top Lyapunov exponent for linear PDEs driven by the 2D stochastic Navier--Stokes equations}
\author{Martin Hairer$^{1,2}$\orcidlink{0000-0002-2141-6561}, Sam Punshon-Smith$^3$ \orcidlink{0000-0003-1827-220X}, Tommaso Rosati$^4$ \orcidlink{0000-0001-5255-6519} and Jaeyun Yi$^1$ \orcidlink{0000-0002-2940-9307}}

\institute{EPFL, Switzerland, \email{martin.hairer@epfl.ch, jaeyun.yi@epfl.ch}
\and Imperial College London, UK, \email{m.hairer@imperial.ac.uk}  
\and Tulane University, USA, \email{spunshonsmith@tulane.edu}
\and University of Warwick, UK, \email{t.rosati@warwick.ac.uk} }

\maketitle

\begin{abstract}
We consider the top Lyapunov exponent associated to the advection-diffusion
and linearised Navier--Stokes equations on the two-dimensional
torus. The velocity field is given by the stochastic Navier--Stokes equations driven by a
non-degenerate white-in-time noise with a power-law correlation structure.  We show
that the top Lyapunov exponent is bounded from below by a negative power of the
diffusion parameter. This partially answers a conjecture of Doering and
Miles and provides a first lower bound on the Batchelor scale in terms of the
diffusivity. The proof relies on a robust analysis of the projective process
associated to the linear equation, through its spectral median dynamics.
We introduce a probabilistic argument to show that high-frequency
states for the projective process are unstable under stochastic perturbations,
leading to a Lyapunov drift condition and quantitative-in-diffusivity estimates.
\vspace{1em}

\noindent{\it Mathematics Subject Classification 2020:} 60H15, 35Q35, 37H15, 37L30

\noindent {\it Keywords:} Fluid mechanics, dynamics, Lyapunov exponents, multiplicative ergodic theorem
\end{abstract}

\setcounter{tocdepth}{2}

\tableofcontents

\section{Introduction}
This work concerns the study of top Lyapunov exponent associated to solutions $ \varrho$ of linear parabolic stochastic PDEs on the two dimensional torus $\TT^{2}$,
\begin{equ}[e:main]
\partial_t \varrho + L[u]\varrho - \kappa \Delta \varrho = 0 \;, \qquad
\varrho(0, x) = \varrho_{0} (x)\;,
\end{equ}
where $ \kappa > 0 $ is a viscosity parameter. The equation is driven by a
Markov process $u_t$ taking values in some space of regular velocity fields through a linear differential
operator $ L[u]  $. 
As we explain below, the methods that we introduce allow in principle for the treatment of a
wide class of driving stochastic processes $ u_{t}$ and of linear
operators $ L[u] $. However, because of their particular relevance, we restrict
our analysis to examples from
fluid mechanics, where $u$ is assumed to solve the incompressible
stochastic Navier--Stokes (SNS) equations on $\TT^2$:
\begin{equation}\label{e:sns}
\partial_t u + u \cdot \nabla u - \Delta u +\nabla p = \xi \;, \quad \div u
= 0 \;, \quad u(0, x) = u_{0}(x) \;,
\end{equation}
driven by some noise $ \xi $ that is white in time and sufficiently
non-degenerate, see \eqref{e:assu-noise} for the precise assumptions.
In this paper we focus primarily on two examples of equation \eqref{e:main}:
the advection diffusion equation and the linearisation of the Navier--Stokes
equations in vorticity form. These correspond to the following operators $L[u]$:
\minilab{e:form-L}
\begin{equs}
L[u]\varrho & = u \cdot \nabla \varrho \;, \qquad \qquad \qquad & \textbf{(Advection)} \;, \label{e:form-PSA}\\
L[u]\varrho & = u\cdot \nabla \varrho + \Delta u\cdot \nabla^{-1}\varrho \;,
\qquad \qquad \qquad & \textbf{(Linearised SNS)} \label{e:form-LNS} \;.
\end{equs}
For the advection diffusion equation we are particularly interested in the {\em small diffusivity }regime $ 0< \kappa \ll 1 $. For the linearised Navier--Stokes equation, we note that there is a viscosity coefficient mismatch between \eqref{e:main} and $\eqref{e:sns}$, so that $\eqref{e:main}$ is only the linearisation of the vorticity form of $\eqref{e:sns}$ with respect to its initial data (also called the first variation equation) when $\kappa = 1$. Regardless, even
if $ \kappa \in (0,1) $ we will refer to this as the linearised SNS
system.

\subsection{Lyapunov exponents and motivation}

The top Lyapunov exponent, $\lambda_1^\kappa$, describes the maximal asymptotic growth rate of the $L^2$ norm of the solution $ \varrho_t$ to \eqref{e:main}
\begin{equ}\label{e:def-lyap-intro}
\lambda_1^\kappa = \lim_{t\to\infty}\frac{1}{t} \sup_{\|\varrho_0\|_{L^2} = 1}\log\|\varrho_t\|_{L^2}\,.
\end{equ}
See section  \ref{sec:main} and Appendix~\ref{appendix:Lyapunov} for more details on the definition of $\lambda_1^\kappa$ and its properties.
Lyapunov exponents have been extensively studied in the context of fluid
mechanics, particularly for understanding the stability of solutions to the
Euler and Navier--Stokes equations
\cite{Chandrasekhar-Hydrodynamic-2013q,Drazin-Hydrodynamic-2004j,IUdovich-Linearization-1989c,Yaglom-Hydrodynamic-2012w},
and in the study of the dynamics of semilinear parabolic problems, again such as the Navier--Stokes equations, to provide upper bounds on the dimension of global attractors \cite{Constantin1985-as,Constantin1985-as,Foias-Navier-stokes-2001t,Temam1997-ft}. Furthermore, they are expected to play a crucial role in analysing the transition to turbulence \cite{Ruelle1971-ic}, with Ruelle and Takens \cite{Ruelle1971-ic,MacKay1991-qb} proposing dynamical chaos, characterised by positive Lyapunov exponents, as a key mechanism. 

Our main result is a lower bound on $\lambda_1^\kappa$ that is \emph{finite} and \emph{quantitative} in the diffusivity coefficient $ \kappa $. Specifically, Theorem~\ref{thm:main} states that for any $ \ve > 0 $ there exists a constant $ C_{\ve} $ depending only on the intensity of the noise, such that for all $\kappa \in (0,1)$
\begin{equ}[e:main-result]
\lambda^{\kappa}_{1} \geqslant -C_{\ve} \kappa^{-3 - \ve} \;.
\end{equ}
Additionally, in Corollary \ref{cor:stationary}, we establish a Furstenberg--Khasminskii-type formula for $\lambda_1^\kappa$
\begin{equ}\label{e:fk-intro}
\lambda_1^\kappa \geq - \int  \left(\kappa\|\nabla\pi\|^2_{L^2} +
\langle\pi,L[u]\pi\rangle_{L^2}\right) \dee \nu(u,\pi) \;,
\end{equ}
where $\nu$ is a stationary probability measure for the ``projective dynamics'' $(u_t,\pi_t)$, $\pi_t = \varrho_t/\|\varrho_t\|$. Formulas of this type are well-known in the context of finite-dimensional systems, see for example \cite{Baxendale1992-ms,Arnold1995-rs,Bedrossian2022-ho}. However, in infinite dimensions, the existence of such a measure is not guaranteed, and the proof of \eqref{e:fk-intro} is a non-trivial extension of the classical theory.

Establishing lower bounds on Lyapunov exponents for systems of the type \eqref{e:main} is crucial for understanding several fundamental open problems in fluid mechanics. For example, in the study of stochastic Navier--Stokes equations, it is expected that there is a positive Lyapunov exponent when the viscosity is taken small enough.  This is supported by rigorous results for Galerkin truncations of the SNS equations \cite{MR4733339} as well as in numerical simulations \cite{Bohr_Jensen_Paladin_Vulpiani_1998,Keefe1992-uw,Fouxon2021-uw,Mohan2017-uw}. Moreover, the positivity of Lyapunov exponents is essential for understanding the efficiency of magnetic field generation in the fast dynamo problem \cite{childress2008stretch}, and the existence of multiple invariant measures in stochastic systems almost surely preserving a submanifold \cite{MR4244269}. These phenomena are all linked to the dynamics of systems with a structure similar to \eqref{e:form-LNS} in the small viscosity regime. The objective of this work is to introduce a new and robust method to obtain quantitative lower bounds on Lyapunov exponents for parabolic PDEs.

\subsection{The Batchelor scale conjecture for advection-diffusion}

A key motivation for our work is the Batchelor scale conjecture, which predicts a ``typical length scale'' of order $\sqrt{\kappa}$ in advection-diffusion systems \cite{Batchelor_1959}.  Above this scale, advection dominates the dynamics, while below it, diffusion takes over. Ultimately mixing should be halted at the Batchelor scale preventing super-exponential decay.  An extensive description of this problem is given by Miles and Doering \cite{miles2018diffusion}, with numerical experiments that validate the Batchelor scale prediction.

The Batchelor scale conjecture can be reformulated in terms of the Lyapunov exponent as follows: the quantity
\begin{equ}\label{e:filamentation}
\ell(\varrho) \eqdef \frac{\|\varrho\|_{L^2}}{\|\nabla \varrho\|_{L^2}} \;,
\end{equ}
defines a type of ``filamentation length scale'' for the solution $\varrho$ to \eqref{e:main}, which one would expect approach the Batchelor scale after long times. Therefore, by the Furstenberg--Khasminskii formula \eqref{e:fk-intro}, if the Batchelor scale were $\sqrt{\kappa}$, the top Lyapunov exponent is bounded from below by 
\begin{equ}\label{e:fk-intro-2}
\lambda_1^\kappa \geq -\int \kappa \ell^{-2}(\pi) \dee \nu(u,\pi) \sim -1 \;.
\end{equ}

Our main result, \eqref{e:main-result}, provides the first rigorous proof of a
finite lower bound for the top Lyapunov exponent in this setting. While the
bound is not yet optimal to fully confirm the Batchelor scale conjecture \dash our
current bound places the Batchelor scale somewhere between $\kappa^{2}$ and
$\kappa^{1/2}$, up to logarithmic corrections \dash it represents a significant step towards its resolution by establishing a quantitative dependence on the diffusivity.

The upper bound $ \lambda_{1}^{\kappa} \lesssim -1 $ is instead a form of {\em enhanced dissipation}, implying the the decay rate is bounded away from zero uniformly in $\kappa$. This {\em uniform in diffusivity mixing} was recently obtained in \cite{bedrossian2021almost} using chaos in the Lagrangian formulation of
advection-diffusion. This phenomenon is also reminiscent of the stochastic stability observed in Ruelle--Pollicott
resonances associated with stationary hyperbolic flows \cite{Dyatlov2015-tq}.
These resonances, which describe the decay of correlations in uniformly
hyperbolic maps, also exhibit a singular dependence on a perturbation parameter
analogous to the viscosity coefficient in our setting.

\subsection{Challenges and prior work}

A major challenge in analysing Lyapunov exponents for infinite-dimensional systems is the possibility of super-exponential decay ($ \lambda^\kappa_1= - \infty $). This phenomenon, where the system's energy\footnote{We will always use ``energy'' to denote the (square of the) $L^2$ norm of $\rho$, even if this doesn't have any physical interpretation as energy.} dissipates at an infinitely fast rate, has been observed in discrete-time systems involving, for example, random Blaschke products \cite{MR4275477} or pulsed diffusions of the cat map \cite{Feng2019-bi}.

In the context of the parabolic problem \eqref{e:main}, the possibility of super-exponential decay is related to the interplay of advection and diffusion. Advection by a smooth incompressible velocity fields tends to mix through filamentation, transferring energy to higher frequencies at exponential rates. When coupled with diffusion, this can lead to a cascade of energy to increasingly high frequencies where it is rapidly dissipated. 
However, such super-exponential decay is not expected for sufficiently
``generic'' velocity fields, as implied by the Batchelor scale conjecture
described above \cite{Batchelor_1959,miles2018diffusion}. 

For the advection diffusion equation, when the velocity field $u = u(x)$ is
smooth and stationary, a classical result on the spectral theory of elliptic
operators due to Agmon \cite[Theorem 16.4]{Agmon-Lectures-ac}
implies that the spectrum of the operator $ -\kappa \Delta + u\cdot\nabla
$ is non-empty, and hence $\rho_t$ decays at most exponentially in time, with
a rate that depends on the initial datum. However,
when $u$ is time dependent, the situation is more complicated.
Super-exponential decay can still be ruled out for scalar, positivity preserving equations,
with signed initial data, through a generalisation of the Krein--Rutman
theorem \cite{doi:10.1142/S0219493722500101, gu2024some}.
Apart from these settings, and the previous work \cite{hairer2023spectral} that
is described below, it is unknown whether a smooth velocity field can produce
super-exponential decay.

\subsection{High-frequency stochastic instability}

To overcome the challenge of potential super-exponential decay, we analyse the \emph{projective process} $ \pi_{t} = \varrho_{t} / \| \varrho_{t} \|_{L^{2}}$ associated to \eqref{e:main}, with values in the $ L^{2} $ unit sphere. In finite dimensions, the analysis of top Lyapunov exponent through ergodic properties of projective
processes is commonplace: see the monographs \cite{MR2894052, MR3289050,Arnold1995-rs}. However, in infinite dimensions, due to lack of compactness of the unit sphere, the projective process might become increasingly irregular at large times and the existence of invariant
measures on the $L^{2}$ unit sphere is not guaranteed. 

The key difficulty is that projective process follows dynamics that are, at least in the deterministic setting, unlike those of a
parabolic equation. For example, in absence of noise and with $ u $ time-independent, $ \pi_{t} $ can be trapped in
high-frequency states that correspond to eigenfunctions of the operator $\kappa \Delta + L[u]$. These states are expected to be saddle points that are unstable under sufficiently generic perturbations.

The present work demonstrates that this instability does appear, due to the randomness in the dynamics of $ u $, a property we refer to as
\emph{high-frequency stochastic instability}.
To quantify this instability, we follow the evolution of the spectral
median of the scalar $ \varrho $, denoted $ M(\varrho_{t}) $. This is the level at which the scalar $ \varrho_{t} $ has roughly half the energy at frequencies higher or lower than $ M (\varrho_{t}) $. This approach was first introduced in \cite{hairer2023spectral}, to analyse linear SPDEs driven by multiplicative white (in time) noise. However, in its original presentation, the method was not quantitative in the viscosity coefficient and could not handle equations driven
by a process that is not white in time, like $ (u_{t})_{t \geqslant 0}
$, or equations with perturbations of first-order differential operators like \eqref{e:form-L}.

We overcome these issues through a new approach to high-frequency stochastic instability, based on a small-time Gaussian chaos decomposition of the solution $ u $ to \eqref{e:sns}, under strong non-degeneracy assumptions on $ \xi$ (see Section~\ref{sec:hfsi}). Our lower bound \eqref{e:main-result} follows from Furstenberg--Khasminskii-type inequalities \eqref{e:fk-intro} (see Corollary~\ref{cor:stationary}).

\subsection{Extensions, limitations and future work}

While our current results provide valuable insights into the behaviour of Lyapunov exponents for parabolic SPDEs, there are several limitations and potential extensions that we aim to address in future work. While parts of our argument are quite general and other parts model-dependent, we believe that a very similar argument should cover other choices of the fluid process $u$ and linear operators $L[u]$, for instance the 3D hyper-viscous SNS, vector advection, or simply the stochastic Stokes equations with noise \eqref{e:assu-noise}; see, for example, see for example \cite[Section~1]{bedrossian2021almost}.

Our method currently relies on a notion of non-degeneracy for the stochastic Navier--Stokes equations, where the noise acts on all sufficiently large frequencies, allowing us to leverage high-high mode interactions to create low frequencies. This excludes certain cases, like Galerkin truncations or systems with only finitely many modes forced stochastically. We are exploring techniques to relax this assumption and extend our results to a broader class of noise structures. We also believe that the exponent $-3-\ep$ in our lower bound \eqref{e:main-result} is not optimal. Based on the Batchelor scale conjecture, we expect the optimal exponent to be $0$. This discrepancy arises from limitations in our use of Gaussian chaos expansions to prove high-frequency stochastic instability. We are investigating alternative approaches to improve the bound and potentially establish its optimality.

Our analysis focuses on the two-dimensional torus. A natural extension is to
generalise our results to higher-dimensional systems. This presents additional
challenges due to the increased complexity of the dynamics, but we believe that
our core techniques can be adapted to address this. While we establish a
Furstenberg--Khasminskii-type inequality, obtaining the corresponding formula
would require proving a spectral gap for the projective process. We conjecture
that such a spectral gap holds for $\kappa>0$ and plan to investigate this
using techniques like asymptotic coupling, see also Remark~\ref{rem:avrg-lyap}.

\subsection*{Acknowledgments}
TR gratefully acknowledges financial support through a Leverhulme Early Career
Fellowship 2024-543. SPS acknowledges support from the National Science Foundation NSF Award DMS-1803481 and a Sloan Fellowship.

\subsection*{Structure of the article}
In Section~\ref{sec:main} we introduce the setting and the main results of this
work, Theorems~\ref{thm:main} and~\ref{thm:lyap}. The strategy and the main
steps for their proof are presented in Section~\ref{sec:hfsi}. Our proof relies
on some probabilistic estimates, which are stated in
Proposition~\ref{prop:bds-new}, and which are proven in
Section~\ref{sec:gaussian-estimates} in the case of pure advection, and in
Section~\ref{sec:gauss-sns} in the case of linearised SNS. Finally, the last
two sections contain mostly technical results.
Section~\ref{sec:regularity} is devoted to estimates on the high-frequency
regularity of the projective process, and Section~\ref{sec:analytic_est}
contains estimates on the velocity field and related quantities. In Appendix~\ref{appendix:Lyapunov} we provide a brief introduction to Lyapunov exponents and their properties.

\subsection*{Notation} Set $ \NN = \{ 0, 1, 2, \dots \} $ and similarly  $
\NN_{*} = \{ 1, 2,3, \dots. \} $, $ \ZZ^{d}_{*} =
\ZZ^{d} \setminus \{ 0 \} $ and $\RR^{d}_{*} = \RR^{d}\setminus \{0\}$. Further let $
\TT^{2} = \RR^{2} / \left( 2 \pi \ZZ^{2} \right) $ be the two-dimensional torus
and $ L^{2}_{0} (\TT^{2}; \RR^{d}) $ the space of
square integrable, mean zero (in each component) functions $ \varphi \colon \TT^{2} \to
\RR^{d} $. Then define $ L^{2}_{*} (\TT^{2}; \RR^{d}) = L^{2}_{0}(\TT^{2};
\RR^{d}) \setminus \{ 0 \} $ and we will omit the target space $ \RR^{d} $ when
clear from context. We denote with $ \| \varphi \| = \|
\varphi \|_{L^{2}} $ the $ L^{2} $ norm of a function $ \varphi $. 
We write $ \mF ( \varphi) \colon
\ZZ^{2}_{*} \to \mathbf{C} $ for the Fourier transform of any
mean zero Schwartz distribution $ \varphi $, given by $$ \hat{\varphi}(k) = \mF
(\varphi) (k) = \frac{1}{(2 \pi)^2} \int_{\TT^{2}}
\varphi(x) e^{ -  \iota k \cdot x} \ud x \;.$$
For a set $ \mX $ and two functions $ \varphi, \psi \colon \mX \to \RR $ we
write $ \varphi \lesssim \psi $ if there exists a constant $ c > 0 $ such that $
\varphi (x) \leqslant c \psi (x) $ for all $ x \in \mX $. We will often use
the Bessel norms
\begin{equ}
\| \varphi \|_{H^{\alpha}}^{2} = \sum_{k \in \ZZ^{2}_{*}} | k
|^{2 \alpha} | \hat{\varphi}(k) |^{2} \;, \qquad \alpha\in \mathbf{R}\;.
\end{equ}
Moreover, we write $ \mC^{\alpha} $ for the Besov--H\"older space with parameter $ \alpha \in \RR
$, which is defined in terms of Paley blocks in \eqref{e:holder}.
Next, we will use high- and low-frequency projections of functions. Here we
use that $ e_{k}(x) = e^{ \iota k \cdot x} $ is an eigenfunction for the
negative Laplacian $ -\Delta $ with eigenvalue $  | k
|^{2} $. Then for every $ M > 0 $ we define respectively the high- and low-frequency projections
\begin{equ}
\mH_{M} \varphi (x) = \sum_{ | k | > M} \hat{\varphi} (k) e_{k} (x) \;, \qquad
\mL_{M} \varphi (x) = \sum_{ | k | \leqslant M} \hat{\varphi}(k) e_{k} (x) \;.
\end{equ}
Finally, we will work with stochastic processes defined on a filtered
probability space $ (\Omega, \mF, (\mF_{t})_{t \geqslant 0}, \PP) $,
where the filtration $(\mF_{t})_{t \geqslant 0} $ is generated by the noise $
\xi $ in \eqref{e:assu-noise}, namely it is formally given by $\mF_{t} = \sigma \left( \xi (x,s)  \; \colon \; s \leqslant
t, x \in \TT^{2} \right)$. In this setting we
denote with $ \EE_{t} $ (and similarly $ \PP_{t} $) the conditional expectation
$ \EE_{t} [\varphi] = \EE [\varphi | \mF_{t}] $. Additionally, as is common for stochastic processes, we will denote the time $t$ evaluation of a process $X$ by $X_t$, not to be confused with the notation for time derivative often used for partial differential equations.

\section{Setting and main results} \label{sec:main}

The goal of this paper is to study the Lyapunov exponent of the system
\eqref{e:main}--\eqref{e:sns}, under the  assumption \eqref{e:form-L} on
the linear evolution of the scalar $ \varrho $. 
To completely describe the system, we must introduce the structural assumptions
on the noise $ \xi $ that drives \eqref{e:sns}.
The latter is a centered Gaussian field
supported on a probability
space $ (\Omega, \mF, \PP) $ and chosen to be Gaussian, white
in time, and non-degenerate. Namely we
force all high frequency modes in such a way that the Fourier coefficients of
the solution $ u $ to \eqref{e:sns} have polynomial decay. More precisely,
we fix $ \xi $ a white-in-time noise with zero mean, with the following 
correlation structure (here $ \hat{\xi} $ denotes the spatial Fourier transform):
\begin{equ}[e:assu-noise]
\EE [ \hat{\xi}_{i} (k , s) \hat{\xi}_{j} (l, t)] = \sigma^2\delta (t-s) \frac{1}{| k
|^{\a}} 1_{\{ k + l = 0 \}} 1_{\{ i = j \}} \;, \qquad \a > 10 \;,
\end{equ}
for $ i, j \in \{ 1, 2 \}, k \in \ZZ^{2}_{*}, t,s \in [0, \infty) $ and 
parameters $ \a > 10 $ and $\sigma >0$. Note that while we allow for an
arbitrary choice of $ \sigma $, our estimates will not be uniform over it. 
The system \eqref{e:sns} is well-posed for example for all $u_0 \in
H^{\alpha}$, for any $ \alpha > 0 $.
See section \ref{sec:velocity-estimates}
below for some preliminary estimates on the stochastic Navier--Stokes system
\eqref{e:sns}. 

In particular, we note that under \eqref{e:assu-noise} the SNS equations
\eqref{e:sns} admit a unique invariant measure, to which the solution converges in total
variation distance \cite[Theorem 3.1]{MR1346374}.  We
write $ \mu $ for the invariant measure on the state space $ H^{\b} $ for $
\b = (\a-3)/2 $. The choice of $ \b $ is arbitrary, for later convenience, and
could be replaced by any other regularity parameter in $ (0, \a/2) $.

We are interested in the study of the Lyapunov exponent associated
 to the solution $ \varrho$. For any $ \varrho_{0} \in L^{2}_{*} $,  we define the Lyapunov
exponent to be the $ \PP \times \mu $-almost sure limit:
\begin{equ}[e:def-lyap]
\lambda^{\kappa} ( \varrho_{0}, u_{0}, \omega ) = \lim_{ t \to \infty} \frac{1}{t} \log{\|
\varrho_{t}(\omega) \| } \;.
\end{equ}
We expect the above limit to hold exist for all initial data $ u_{0} $ and not
just for $ \mu $-almost all. However working with the measure $ \PP \times \mu
$ on $ \Omega \times H^{\b} $ turns $ \varrho $ into a linear co-cycle with
respect to a measure-preserving transformation (given by the skew product $
\Theta_{t} (\omega,  u_{0}) = (\vt_{t} \omega, u_{t}(\omega;u_0)) $, where $
\vt_{t} $ is the time-shift of the noise and $u_t(\omega;u_0)$ the solution to \eqref{e:sns} with initial state $ u_{0} $.), so that we are in a classical setting and
can apply well-known results.
Indeed, the existence of $ \lambda^{\kappa}_1 (\varrho_{0}, u_{0}, \omega ) \in [- \infty,
\infty) $ is guaranteed by the multiplicative ergodic theorem (MET): see for example \cite[Theorem
A]{MR2674952}, and note that the log integrability condition required by the
theorem is checked in
Lemma~\ref{lem:integrability}. Moreover, the MET also implies the existence of the top
Lyapunov exponent $\lambda^\kappa_1\in [-\infty, \infty)$, which is deterministic and independent of the
initial conditions $ \varrho_{0}, u_0$, and satisfies that $\P \times \mu $-almost surely
$\lambda_1^{\kappa}\geqslant \lambda^{\kappa}(\varrho_{0}, u_{0} , \omega)$:
\begin{equ}[e:lambda-top]
\lambda^{\kappa}_1 = \lim_{ t \to \infty} \sup_{\|\varrho_{0} \| =1}\frac{1}{t} \log{\| \varrho_{t} \| } \;.
\end{equ}
Note that the top Lyapunov exponent is ``attainable'', in the sense that one
can find an initial condition with growth arbitrarily close to $
\lambda_{1} $. Namely, by Lemma~\ref{lem:exp-decay}, for any $ \ve > 0 $ and $ \PP \times \mu $-almost all $
(\omega, u_{0}) \in \Omega \times H^{\b} $ there exists a non-empty,
finite-dimensional, invariant subspace $ G_{\ve} (\omega, u_{0}) \subseteq L^{2}_{0}$, a conical neighbourhood $C_\ve$ of $G_{\ve}$ and a constant $ D_{\ve}(\omega, u_{0}) > 0  $ such that for any $
\varrho_{0} \in C_{\ve} $:
\begin{equ}[e:lb-lyap]
   \|\varrho_t\| \geq  D_\ep(\omega,u_0)\|\varrho_0\|
e^{(\lambda^{\kappa}_1-\ep) t}\;.
\end{equ}
Furthermore, by \cite{MR4642633}, the exponential bound \eqref{e:lb-lyap}
is not special to $ L^2$, and the $ L^{2} $ norm can for example be replaced by
an $ H^{\alpha} $ norm for any $ | \alpha | < \a/2 $. More generally,
under suitable log moment condition, the statement of
Lemma~\ref{lem:exp-decay} also holds on any invariant Banach space of
functions $B$ such that the evolution $\varrho_0\mapsto \varrho_t$ is
compact on $ B $ and such that either $B$ is continuously
and densely embedded in $L^2$, or $L^2$ is continuously and densely embedded in
$B$.

At this level of generality, the exponent $ \lambda^{\kappa}$ defined in $\eqref{e:def-lyap}$ may be random and depend on the initial
conditions $ \varrho_{0},u_0$. However, if there exists a unique stationary measure $\nu$ for the projective dynamics $(u_t, \pi_{t})$, with $ \pi_{t} = \varrho_{t} / \|
\varrho_{t} \| $, then {\em for every} $(\varrho_0,u_0) \in
L^2_*\times H^\b$, we have $\lambda^\kappa(\varrho_0,u_0,\omega) =
\lambda^\kappa_1$, $\P$ almost surely. See for example \cite[Theorem
III.1.2]{MR884892} and \cite{MR805125} for such results in finite dimensions,
and see also Remark~\ref{rem:avrg-lyap} below.
Proving uniqueness of invariant measures for the projective dynamics is a major
open problem. In \cite{hairer2023spectral} it has been possible to prove a
spectral gap when $ u \cdot \nabla $ is replaced by a multiplicative
white-in-time noise, via the Bismut--Elworthy--Li formula. In our setting,
obtaining a spectral gap would require a significant extension of asymptotic
strong Feller techniques, perhaps through an extension of the Lie bracket structure in
\cite{MR4733339} to infinite dimensions. 

In any case, uniqueness of the invariant measure is not the focus of the
present work, which concerns lower bounds (in expectation) on $ \lambda^{\kappa}$ that are uniform over all
parameters: in particular, our lower bound will be 
independent of both $ \varrho_{0}, u_{0} $, and with an explicit dependence on
the viscosity $ \kappa $. 
We highlight that obtaining a uniform lower bound on the Lyapunov exponent is
only possible when the initial condition is adapted independent of the future
noise $ \xi $. This is implied in the formulation of all our results, since $
\varrho_{0} $ can not depend on $ \omega $. Our main result is the following. 

\begin{theorem}\label{thm:main}
Let $ \xi $ satisfy \eqref{e:assu-noise} and $ \lambda^{\kappa}_{1} $ be defined
as in \eqref{e:lambda-top}.
Then, for every $ \q > 3$, there exists a $ C > 0 $ such that, uniformly over
$ \kappa \in (0, 1] $, we have 
$$\lambda^\kappa_1 \geqslant  - C
\kappa^{-\q}\;.$$
\end{theorem}
This result is implied by Corollary~\ref{cor:stationary} below.
The main ingredient in the proof of Theorem~\ref{thm:main}  is a bound on the projective process $
\pi_{t} = \varrho_{t} / \| \varrho_{t} \| $. This is a process with values in
the infinite-dimensional sphere
\begin{equ}
S = \{ \varphi \in L^{2}_{*} (\TT^{2})  \; \colon \; \| \varphi \| =1 \} \;,
\end{equ}
and our control on $ \pi_{t} $ is made precise in the following theorem. Note that by parabolic backwards
uniqueness \cite{Poon}, $ \varrho_{t} \neq 0 $ for all $ t \geqslant 0 $ and
all $ \varrho_{0} \in L^{2}_{*} $, so that the process $ \pi_{t} $ is defined for all times.

\begin{theorem}\label{thm:lyap} 
Consider the same setting as in Theorem~\ref{thm:main}.
For any $\q>4$, there exist constants $ C , \r > 1 $ and a time $ t_{0} > 0 $
such that uniformly over all $ \kappa \in (0, 1], \pi_{0} \in S $ and $ u_{0}
\in H^{\b}$ with $ \b = (\a -3)/2 $
\begin{equ}[e:mainLyap]
\EE\left[\| \pi_{t_{0}} \|^2_{H^1}\right] \leqslant C\kappa^{- \q} V_{\r}(u_0)
\;,
\end{equ}
with $ V_{\r} $ as defined in \eqref{e:lyap-sns}.
\end{theorem}
The previous theorems are directly connected to the existence of invariant
measures for the projective process, and to Furstenberg--Khasminskii formulas.
We highlight this connection in the following corollary.
\begin{corollary}\label{cor:stationary}
In the same setting as in Theorem~\ref{thm:main}, define the projective process $ \pi_{t} =
\varrho_{t}/ \| \varrho_{t} \| $ for all $ t \geqslant 0$. Then
the couple $
(u_{t}, \pi_{t})_{t \geqslant 0} $ is a Markov process with values in $
H^{\b} \times S $ and there exists at least one stationary measure $\nu$ for $(u_t, \pi_t) $ on $ H^{\b} \times S$. Moreover any stationary measure $\nu$ satisfies for any $\q > 3$
\begin{equ}\label{e:Furstenberg-Khasminskii-PSA}
\lambda_1^\kappa \geq - \int_{H^{\b} \times S} \kappa\|\nabla \pi\|^2 \ud
\nu(u,\pi) \geq -C\kappa^{-\q} \;,
\end{equ}
in the case in which $L[u]$ is given by \eqref{e:form-PSA}, and
\begin{equ}\label{e:Furstenberg-Khasminskii-LNS}
   \lambda_1^\kappa \geq - \int_{H^{\b} \times S} \kappa\|\nabla \pi\|^2
+\langle \pi, \Delta u\cdot \nabla^{-1}\pi\rangle \ud \nu(u,\pi) \geq
-C\kappa^{-\q}\;,
   \end{equ}
in the case in which $L[u]$ is given by \eqref{e:form-LNS}.
\end{corollary}

\begin{remark}\label{rem:avrg-lyap}
Rather than considering the top Lyapunov exponent, we can also study $
\lambda(\varrho_{0}, u_{0}, \omega)$ as defined in \eqref{e:def-lyap}. Because
for every $ \pi_{0} \in S, u_{0} \in H^{\b} $ Theorem~\ref{thm:lyap} implies
the existence of at least one invariant measure $ \nu $ for $ (u, \pi) $, we
deduce from Corollary~\ref{cor:stationary} the estimate
\begin{equ}
\EE [ \lambda (\varrho_{0}, u_{0})] \geqslant - C \kappa^{- \q} \;,
\end{equ}
with the same $ C > 0 $ as in Corollary~\ref{cor:stationary}. However, this result is not surprising,
since we expect something much stronger, namely all the
Lyapunov exponents $ \lambda (\varrho_{0}, u_{0}) $ to be identically equal to $ \lambda_{1}^{\kappa} $. This
would be an immediate consequence of unique ergodicity for $
(u, \pi) $. As we have already remarked, such unique ergodicity is expected
(one can show weak irreducibility for the Markov process $ (u, \pi) $), but it is technically
challenging to prove the (asymptotic) strong Feller property.
\end{remark}
Now we are ready to prove Corollary~\ref{cor:stationary}.

\begin{proof}
Note that $ \pi_{t} $ is defined for all times, for all $ u_{0} \in
H^{\b} $ and $ \varrho_{0} \in L^{2}_{*} $, as discussed above the statement of
Theorem~\ref{thm:lyap}.
That $ (u_{t}, \pi_{t})_{t \geqslant 0} $ is a Markov process follows from
the linearity of \eqref{e:main}. Furthermore, it follows from
Theorem~\ref{thm:lyap} and Lemma~\ref{lem:lyap-sns} that for any $ \r > 0 $ and
$ \beta \in (\b, \a/2-1) $
$$ \sup_{t_{0} \leqslant t < \infty}
\EE  \left[ \| \pi_{t} \|^{2}_{H^{1}} + V_{\r, \beta} (u_{t})\right]  \leqslant
C(\kappa, u_{0}) < \infty \;,$$
where $ V_{\r, \beta} $ is the Lyapunov functional for \eqref{e:sns} defined in
\eqref{e:lyap-sns}.
This immediately implies tightness for the process $ (u_{t}, \pi_{t})_{t
\geqslant 0} $, and therefore the existence of stationary ergodic measures via
Krylov--Bogolyubov, provided that $
(u_{t}, \pi_{t}) $ is Feller. To show this, we first note that the solution $u$ to \eqref{e:sns} 
viewed as an element of $C(\R_+,C^1)$
depends continuously on $u_0$ for almost every $\xi$. 
Furthermore, the 
solution to \eqref{e:main} depends continuously on $\varrho_0$ and $u$. 
It remains to show that $\varrho_t \neq 0$, so that 
the random flow $ (u_{t}, \pi_{t}) = \Phi_{t} (u_{0},
\pi_{0}) $ is almost surely continuous, thus implying the Feller property.
This however is an immediate consequence of parabolic backwards uniqueness,
see for example \cite{Poon}.

Next, we prove \eqref{e:Furstenberg-Khasminskii-PSA}: we will omit the proof of
\eqref{e:Furstenberg-Khasminskii-LNS}, which is similar. Note that by
Theorem~\ref{thm:lyap}
\[
(u_0,\pi_0)\mapsto \E_{u_0,\pi_0}\|\pi_{t_0}\|_{H^1}^2 
\]
is integrable with respect to $\nu$ since $V_\r$ is integrable with respect to
the $u$-marginal $\mu$. The integrability of $V_\r$ with respect to $\mu$ is a
standard consequence of the Lyapunov property of $V_\r$. Using the stationarity
of $\nu$ gives that for $\q>3$
\[
\int_{H^\b\times S} \kappa \|\pi\|_{H^1}^2\ud\nu(u,\pi) = \int_{H^\b\times
S} \kappa \E\|\pi_{t_0}\|_{H^1}^2\ud\nu(u_0,\pi_0) \leq C\kappa^{-\q}\;.
\]
Note that we can write
\[
\log\|\varrho_t\| = \int_0^t \frac{\langle \varrho_s,\kappa\Delta\varrho_s - u_s\cdot\nabla \varrho_s\rangle}{\|\varrho_s\|^2}\ds = -\int_0^t\kappa \|\nabla \pi_s\|^2\ud s\;,
\]
where in the last line, we used integration by parts and the divergence-free property of $u$. Now, suppose that $\nu$ is ergodic, then by the ergodic theorem then for $\nu$-almost all $(u_0,\pi_0)$
\begin{equs}
\lambda_1^\kappa  \geq \E\lambda(\pi_0,u_0) &= \lim_{t\to\infty}\frac{1}{t}\E\log\|\varrho_t\|
   =- \lim_{t\to \infty}\frac{1}{t}\int_{t_0}^t\kappa\E\|\pi_{s}\|_{H^1}^2\ud s\\
 &=
-\int_{H^\b\times S}\kappa\|\pi\|_{H^1}^2\ud\nu(u,\pi) \geq -C\kappa^{-\q}\;.
\end{equs}
If $\nu$ is not ergodic, then we can consider the ergodic decomposition of $\nu$ and the above argument applies to each ergodic component of $\nu$. Taking a convex combination of the resulting lower bounds gives the desired result.
\end{proof}
The fundamental step in the proof of Theorem~\ref{thm:lyap} is a study of
spectral quantiles, in the spirit of \cite{hairer2023spectral}. We show that
even if the energy of the
system is initially concentrated at high frequencies, say of order $ M \gg 1$, then
by a time of order one the majority of the energy will have reached low
frequencies of order one, independently of the value of $ M $. A consequence of this
phenomenon is that the
right-hand side in the estimate in Theorem~\ref{thm:lyap} does not depend on
the initial condition $ \pi_{0}$, and moreover the theorem implies
the existence of projective stationary measures.

Our approach to prove such phenomenon rests
on what we call \emph{high-frequency stochastic instability},
which is the idea that high-frequency states are unstable
under stochastic perturbations. 
In fact, the main challenge in analysing infinite-dimensional projective processes is
that the dynamics of such processes might get trapped in high-frequency states.
For example, if $ u = 0 $ and $ \varrho $ solves the heat equation, then every
eigenfunction of the Laplacian is a steady state for the projective process. In
particular a uniform lower bound on the instantaneous dissipation rate as in Theorem~\ref{thm:main}
is impossible, since one can always choose $ \varrho_{0} $ to be an
eigenfunction with an arbitrarily negative eigenvalue: the spectrum of the Laplacian
is unbounded from below.

Here \emph{probability} comes to help. The presence of noise makes any high-frequency state
unstable, because the fluctuations excite low frequencies and transfer to them some
small amount of energy from high frequencies. Then, because high frequencies
dissipate at a much faster rate than low ones,  the vast majority of the
energy quickly concentrates in low frequencies. Of
course, this picture is only partial since the threshold that divides ``high''
from ``low'' frequencies diverges as the viscosity vanishes (at order $ \kappa^{- \q} $ for
$ \q > 4 $). Therefore, whenever we speak of ``low'' frequencies, these may still be
very large for $ \kappa \ll 1 $. 

All the main steps in the proof
of Theorem~\ref{thm:lyap} are detailed in the next section, with technical
results deferred to later parts of the paper. 
In particular, high-frequency stochastic instability is the consequence of
the probabilistic energy transfer estimates in Proposition~\ref{prop:bds-new}
(which is proven in Section~\ref{sec:gaussian-estimates}), together with
some analytic results which rely on the diffusive effect of the
Laplacian: these are contained in Lemma~\ref{lem:t-1}, Lemma~\ref{lem:t-2} and
Proposition~\ref{prop:stopped} below.
The proof of Theorem~\ref{thm:lyap} itself is then provided at the end
of Section~\ref{sec:prof-main-rslt}.

\section{High-frequency stochastic instability} \label{sec:hfsi}

This section is devoted to the proof of Theorem~\ref{thm:lyap}. As we have
briefly explained above, the core of the proof lies in using probability to
show that high frequency states are unstable for the projective dynamics. The
first step is therefore to provide a mathematically useful definition of being
in a ``high frequency state''. Here we follow a similar approach as in
\cite{hairer2023spectral} through the study of spectral quantiles, and in
particular through the spectral median: a high frequency state
would correspond to a large value of such median.

To introduce spectral quantiles, we start with the high- and low-frequency projections
$ \mH_{M} $ and $ \mL_{M} $, which are respectively the projections on
frequencies higher or lower than $ M $:
\begin{equ}
\mH_{M} \varphi = \sum_{ | k | > M} \hat{\varphi} (k)
e_{k} (x) \;,
\qquad \mL_{M} \varphi = \sum_{ | k | \leqslant M}
\hat{\varphi} (k) e_{k}(x) \;.
\end{equ}
This definition extends naturally to vector-valued functions by applying the
projection to each component.

We introduce spectral
$ \beta$-quantiles (with some abuse of jargon, because in its common use a
quantile $ x $ is related to $ \beta $ by $ \beta = (1-x)/x $), for arbitrary $ \beta >0 $ by setting
\begin{equ}[e:beta-median]
M^{(\beta)} (\varrho) = \min \{ M \in \NN_{*}  \; \colon \; \| \mH_{M} \varrho \|
\leqslant \beta \| \mL_{M} \varrho \| \} \;.
\end{equ}
When $\beta = 1$, we write $M(\varrho) = M^{(1)}(\varrho)$ and call it
the \emph{spectral median}. In our setting, we are mainly interested in the
spectral median of the passive scalar $ \varrho $, so we set
\begin{equ}
M_{t} = M (\varrho_{t}) = M (\pi_{t}) \;,
\end{equ}
observing that the median is invariant under multiplication of $\varrho$ by (non-zero) scalars.

Our aim is to show that if the median is very large at a given time (which, 
since the process is time-homogeneous, can be chosen to be zero), then
it will rapidly decrease with high probability, because some energy is
transferred to lower frequencies. We will show that, with some positive probability, 
a small such energy transfer happens at the latest by the \emph{diffusive} timescale
\begin{equ}[e:t-kappa]
t_{\star}^{\kappa} (M) \eqdef  \lambda \kappa^{-1} M^{-2} \log{(M)} \;, \qquad
\lambda = \frac{\a}{2} + 5 \;.
\end{equ}
The choice of the parameter $ \lambda $ is somewhat arbitrary and used only in
Lemma~\ref{lem:t-1} below: it can be replaced by any constant strictly greater
than $ \a/2 +4 $.
The choice of this
timescale is related to our overall approach to show energy transfer to low
frequencies, which is a proof by contradiction.
Indeed, if the spectral median $ M(\varrho_{t}) $ were
to remain larger than $
M $ for all $ t \in [0, \ts (M)] $, then through dissipation by time $ \ts (M)
$ the total amount of energy would be
polynomially small, of order $ M^{- \lambda} \| \varrho_{0} \| $. 
Therefore, if the energy increase in low frequencies is
such that by time $ \ts (M) $ we obtain a lower bound that
is an inverse power of $ M $ on the total energy (and if $ \lambda $ is sufficiently
large), we would find a contradiction, implying that the median must
have decreased before time $ \ts (M) $: for the details of this argument see
the proof or Lemma~\ref{lem:t-1}.

In fact, we are able to show that an energy production at low frequencies, with a polynomial lower
bound, takes place by time $ \ts(M) $. This is the content of the following
proposition.

\begin{proposition}\label{prop:bds-new}
Consider the setting of Theorem~\ref{thm:main}.
For any $\beta \geqslant 1 $ and $ \q > 2 $, there exist $ \alpha, C_{1},
C_{2}>0, \gamma \in (0, \a/2) $ and $ \r > 1 $, such that uniformly over $
\kappa \in (0, 1] $ and $
  \varrho_{0} \in L^{2}_{*}$ such that $ \| \varrho_{0} \|=1 $, the following estimate holds:
\begin{equ}
\PP \left( \left\| \mL_{\sqrt{2}} \varrho_{\ts (M)} \right\| \geqslant C_{1} \kappa^{- \frac{1}{2}} M^{-
\frac{\a}{2} -4} \right) \geqslant \alpha  \;, \qquad \forall M \geqslant M^{(\beta)}(\varrho_{0}) \vee
\kappa^{- \q } \vee C_{2} (\| u_{0} \|_{\mC^{\gamma}}+1)^{\r} \;.
\end{equ}
\end{proposition}
A complete proof of the proposition is deferred to 
Section~\ref{sec:prf-prop}. We note that the projection onto a ball of
frequencies of radius $ \sqrt{2} $ (rather than for example a ball of radius
$ 1 $) is necessary for technical reasons to be able to cover the case of
linearised SNS, see Section~\ref{sec:gauss-sns}.
However, before we proceed, we detail in the next subsection the strategy for
the proof of the proposition, which is based on a short-time expansion of the
solution.

\subsection{Polynomial chaos expansions for low frequency energy production}
\label{sec:chaos-expansion}
Our approach to showing Proposition~\ref{prop:bds-new} builds on an expansion of
the solutions $ \varrho $ and $ u $  
to~\eqref{e:main}--\eqref{e:sns}. The idea is that for short times $ u
$ is well approximated by a Gaussian process, which influences linearly the
passive scalar.

\begin{remark}\label{rem:short-time}
The idea of a short-time expansion for the mild solution of the passive scalar
is sub-optimal to obtain uniform bounds in the viscosity $ \kappa \in
(0,1) $. When the viscosity is very small, the solution should resemble more a
transport equation than a perturbation of the heat equation. However, our
expansion is mathematically convenient, as it reduces the entire problem to
bounding a Gaussian stochastic integral. Developing a
better approach for low frequency energy production is left to future work.
\end{remark}
Let us explain more precisely our expansion starting with the Navier--Stokes
equations~\eqref{e:sns}. Since the solution $ u $ to \eqref{e:sns} is divergence free, it is
convenient to project the equation onto the space of divergence-free vector fields through the Leray projection $
\mathbf{P} $, so that \eqref{e:sns} can be rewritten as
\begin{equ}
\partial_{t} u + \mathbf{P} \div(u \otimes u) - \Delta u =
\mathbf{P} \xi\;, \qquad u(0, \cdot) = u_{0}(\cdot) \;,
\end{equ}
where we have used the definition of matrix-valued
divergence:
$\div ( A)_j (x) = \sum_{i =1}^{2} \partial_{i} A_{ij} (x)$.
We then decompose the solution into its linear part and a remainder. Given an initial velocity
field $u_0$, we define the Gaussian process $ X $, which depends on the initial
condition $ u_{0} $, by
\begin{equ}[e:def-X]
\partial_{t} X -  \Delta X = \mathbf{P} \xi \;, \qquad X(0, \cdot) =
u_{0}(\cdot) \;.
\end{equ}
Next, for any time-dependent vector fields $ w_{1}$ and $w_{2} $, define the
quadratic form
\begin{equ}[e:def-psi]
\Psi[w_{1}, w_{2}]_{t} =  - \int_{0}^{t} P_{t-s} \mathbf{P} \div(
w_{1} \sotimes w_{2}) \ud s \;,
\end{equ}
where $ P_{t} = \exp ( t \Delta) $ denotes the heat semigroup and $\sotimes$ 
denotes the symmetrised tensor product.
In this context, the solution $u$ to \eqref{e:sns} can be written as $ u = X + \psi $,
where $ \psi $ solves the fixed point problem
\begin{equ}
\psi = \Psi [X+ \psi, X+ \psi] \;.
\end{equ}
This allows to formally expand $ u $ in terms of $ X $ as
$ u = X + \Psi [X] +2 \Psi[X, \Psi[X]] + \Psi[\Psi[X]] + \dots$, where we have
used the shorthand notation $ \Psi[w] = \Psi[w, w] $.

We can proceed similarly for the solution $ \varrho $ to \eqref{e:main}, by
defining
\begin{equ}[e:def-Y]
 Y_t = P_t^{\kappa} \varrho_{0} \;,\qquad P_t^{\kappa} = \exp ( \kappa t \Delta)\;,
\end{equ}
and setting $ \varrho = Y + \varphi $, so that $\varphi$ solves the
fixed point problem $\varphi = \Phi [u, Y+ \varphi]$ with
\begin{equ}[e:def-phi]
\Phi[u, w]_{t} = - \int_{0}^{t} P^{\kappa}_{t-s}  \left[ L[u_{s}] w_{s} \right]
\ud s \;.
\end{equ}
This similarly leads to an expansion for $ \varrho $ of the form
\begin{equ}
\varrho = \Phi[X, Y] + \Phi[\Psi[X], Y] + \Phi[X, \Phi[X, Y]] + \dots \;.
\end{equ}
This should be interpreted as a \emph{short-time} expansion, meaning
roughly that higher terms in the expansion behave like higher powers of
$ t $ as $ t \to 0 $. Establishing precise estimates on such decay rates is
not trivial since we need to simultaneously keep track on the dependence on
$M$,
 and this is the objective of Section~\ref{sec:gaussian-estimates}. At
the heart of our approach lies the idea that the dominating term in the
expansion of $ \varrho $ in terms of energy production at low frequencies, for
small times, is the first term $ \Phi [X,Y] $,
which is Gaussian and can be controlled precisely.

In particular, our aim is to show that even if the energy of $ \varrho_{0} $ is
concentrated in high frequencies, a fraction of that energy will transfer to
low frequencies in a very short time, through the term $ \Phi[X, Y] $. Of
course, to make use of this bound we must show that the later
terms do not cancel out this first contribution. Therefore, our analysis in
Section~\ref{sec:gaussian-estimates} comprises of a lower bound on the energy
produced at low frequencies by the Gaussian term $ \| \mL_{1} \Phi[X, Y] \| $
(Lemma~\ref{lem:LBI1-new}) and an upper bound on the energy
produced at low frequencies by the other terms
(these are the content of all subsequent results in
Section~\ref{sec:gaussian-estimates}). 

However, in view of Remark~\ref{rem:short-time}, we observe that the 
fundamental idea that the first term
has the highest contribution seems unlikely to hold uniformly over the viscosity
constant $ \kappa \in (0, 1) $. In particular, for our expansion to work
we will require that the median frequency satisfies $ M(\varrho_{t}) \gtrsim
\kappa^{- \q} $ for some $ \q > 5/2 $, leading to the constraint on $M$ in
Proposition~\ref{prop:bds-new}.

\subsection{Spectral median dynamics}\label{sec:short-time-insta}

We are now ready to build on the small energy transfer that is proven in
Proposition~\ref{prop:bds-new}, in order to prove that a substantial amount of
the total energy of $ \varrho $ is transferred to low frequencies. Another perspective on this
phenomenon is that high frequency states are unstable for the projective
dynamics, so we also refer to this as high frequency stochastic
instability.
A consequence of this effect is that the spectral median behaves like a
stochastic process with a strong drift towards low frequencies, until it
reaches a level of order $ \kappa^{-\q} $ for $ \q>2 $.

Indeed, we will now show that the median is likely to decrease by one unit in a short
time $ \ts (M) $ for $ M = M (\varrho_{0}) $. We introduce the stopping time 
\begin{equ}[e:tau-n]
\tau(M) = \inf \{ t \geqslant 0  \; \colon \; M (\varrho_{t})  < M -1 \}
\;,
\end{equ}
and we first show that it has some order $1$ probability to be less
than the dissipative timescale $ \ts (M) $.

\begin{lemma}\label{lem:t-1}
Consider the same setting as in Theorem~\ref{thm:main}.
For any $ \q > 2$, there exist $ \r, C, \alpha > 0 $, and a $
\gamma \in (0, \a/2) $ such that uniformly $ u_{0} \in \mC^{\gamma},
\varrho_{0} \in L^{2}_{*}$ and $ \kappa
\in (0, 1] $
\begin{equ}
\PP \left( \tau(M) \leqslant t_{\star}^{\kappa}(M) \right) \geqslant \alpha \;, \qquad
\forall M \geqslant M^{(2)} (\varrho_{0}) \vee \kappa^{- \q} \vee C (\|
u_{0} \|_{\mC^{\gamma}}+1)^{\r} \;.
\end{equ}
\end{lemma}
Note that since $ M^{(2)} (\varrho_{0}) \leqslant
M^{(1)}( \varrho_{0}) $, the result can be applied for example to the median level $ M =
M(\varrho_{0}) $, provided it is larger than $ \kappa^{- \q} $: the assumption $ M
\geqslant M^{(2) } (\varrho_{0}) $ is only necessary to provide us some flexibility
that is needed in the upcoming proofs.

\begin{proof}
Let us assume that $ \| \varrho_{0} \|_{L^{2}} =1 $, since otherwise we can
normalise the initial condition and the result would not change because of the
linearity of the equation. Then Proposition~\ref{prop:bds-new} implies that
there exist $ C_{1}, \alpha, C_{2} , \r  > 0, \gamma \in (0, \a/2) $ such
that 
\begin{equ}[e:lbr]
\PP \Big( \| \mL_{\sqrt{2} } \varrho_{t_{\star}^{\kappa} (M) } \| <
C_{1} \kappa^{- \frac{1}{2}} M^{-\frac{\a}{2}- 4}  \Big) 
\leqslant  \alpha  \;,
\end{equ}
given that $ M \geqslant M^{(2)} (\varrho_{0}) \vee \kappa^{- \q} \vee
C_{2} (\| u_{0} \|_{\mC^{\gamma}}+1)^{\r} $.

Next, we show that that high modes are reasonably well controlled
up to time $\ts(M)$, in such a way that \eqref{e:lbr} implies a decrease of
the spectral median. In principle, this depends on the particular choice of
$ L [u] $ in \eqref{e:form-L}. However, since the advection case is simpler
than the linearised SNS case, we only consider the latter. For $t<
\tau(M)$ we have $M(\varrho_t) > M-1$ and we find (for a constant $C$ that may
change from line to line)
\[
\begin{aligned}
\partial_{t} \| \mH_{M-1} \varrho_{t} \|^{2} & \leqslant -2 \kappa
|M-1|^{2} \|
\mH_{M-1} \varrho_{t} \|^{2} - 2 \langle \mH_{M-1} \varrho_{t},
\div( \varrho_{t} u_{t} ) \rangle - 2 \langle \mH_{M-1} \varrho_{t}, \Delta
u_{t} \cdot \nabla^{-1} \varrho_{t} \rangle\\
& \leqslant  (-2\kappa |M-1|^2 + CM\|u_t\|_{ \mC^{\gamma}}) \|\mH_{M-1}\varrho_t\|^2\;,
\end{aligned}
\]
for any $ \gamma > 2 $.  Here we have used that
\[
   \langle \mH_{M-1} \varrho_{t}, \div( \varrho_{t} u_{t} ) \rangle =  \langle \mH_{M-1} \varrho_{t},
\div( (\mL_{M-1} \varrho_{t}) u_{t} ) \rangle  =  \langle \mH_{M-1}
\varrho_{t}, (\nabla \mL_{M-1} \varrho_{t}) \cdot u_{t} \rangle \;,
\]
together with the bounds 
\[
\begin{aligned}
| \langle \mH_{M-1} \varrho_{t}, (\nabla \mL_{M-1}
\varrho_{t}) \cdot u_{t} \rangle | & \lesssim M \|\mH_{M-1}\varrho_t\|^{2} \|
u_{t} \|_{\infty} \;, \\
   | \langle \mH_{M-1} \varrho_{t} , \Delta
u_{t} \cdot \nabla^{-1} \varrho_{t}\rangle | & \lesssim \|\mH_{M-1}\varrho_t\|^{2} \| \Delta u_{t} \|_{\infty} \;,
\end{aligned}
\]
since $M(\varrho_t) > M-1$. Hence, if $\ts(M) \leqslant  \tau(M)$ we obtain
that for some $ C> 0 $
\begin{equ}
\|\mH_{M-1}\varrho_{\ts(M)}\|  \leqslant  M^{-\lambda}
\exp\Bigl({CM\ts(M) \Bigl( \sup_{0 \leqslant  s \leqslant  \ts(M)}\|u_s\|_{
\mC^{\gamma}}} + \kappa \Bigr) \Bigr) \;,
\end{equ}
where the $ \kappa $ on the right-hand side appears because we have estimated
$ - 2\kappa | M -1 |^{2} \leqslant - 2 \kappa | M |^{2} + 4 M \kappa $.
Next, denoting by $ \mB_{M} $ the event
\[
\mathcal{B}_M = \left\{\| \mL_{\sqrt{2} } \varrho_{t_{\star}^{\kappa} (M) } \| \geqslant
C_{1} \kappa^{- \frac{1}{2}} M^{-\frac{\a}{2}- 4}\right\}\;,
\]
so that we can bound from above the desired probability as follows:
\begin{equs}
\PP(\ts(M) < \tau(M)) & \leqslant  \PP\left(\|\mL_{M-1}\varrho_{\ts(M)}\|
\leqslant C
M^{-\lambda}\exp\Bigl({CM \ts(M)\sup_{0\leqslant s \leqslant   \ts(M)}\|u_s\|_{\mC^{\gamma}}}\Bigr)\right)\\
& \leqslant  \PP\left( C_{1} \kappa^{- \frac{1}{2} } M^{\lambda -\a/2-4} \leqslant C 
\exp\Bigl({CM\ts(M)\sup_{0\leqslant s \leqslant  \ts(M)}\|u_s\|_{\mC^{\gamma}}}\Bigr)\right)
+ \PP(\mathcal{B}_M^c)\\
& \leqslant  \PP\left( C_{1} \leqslant C
\exp\Bigl({CM\ts(M)\sup_{0\leqslant s \leqslant  \ts(M)}\|u_s\|_{\mC^{\gamma}}}\Bigr)\right)
+ \PP(\mathcal{B}_M^c)\\
& \leqslant  \PP\Bigl(\sup_{0\leqslant s \leqslant  \ts(M)}\|u_s\|_{\mC^{\gamma}}
\geqslant CM^{1/\r}\Bigr) + \PP(\mathcal{B}_M^c) \;.
\end{equs}
Here in the second to last line we have used that $ \lambda =5 + \a/2 >  4 +
\a/2$, see \eqref{e:t-kappa}, together with the fact that $ \kappa \leqslant 1
$. Moreover, the value of the constant $ C $ was allowed to change in the last
estimate.

Now we conclude by observing that via Lemma~\ref{lem:mmt-unif}, for any $ \gamma \in
(2, \a/2) $ there exists an $ \r > 1 $ such that for all $ p \geqslant 0 $ and
some $ C_{3}(p) > 0 $:
\begin{equ}[e:finalBd]
   \PP\Bigl(\sup_{0\leqslant s \leqslant  \ts(M)}\|u_s\|_{\mC^{\gamma}} \geqslant CM^{1/\r}\Bigr) \leqslant C_{3} M^{- p} (\| u_{0} \|_{\mC^{\gamma}}+1)^{\r p}\;.
\end{equ}
Since \eqref{e:lbr} guarantees that $\PP(\mathcal{B}_M^c) \le \alpha$ and the right-hand side
of \eqref{e:finalBd} can be made smaller than $(1-\alpha)/2$ by choosing the constant $C$ 
in our statement large enough, this completes the proof.
\end{proof}

Lemma~\ref{lem:t-1} is one of the fundamental results of this
article, as it proves that the spectral median has a tendency to decrease, at least until
it reaches level $ \kappa^{-\q} $ for arbitrary $ \q> 2 $. The remainder of this
section, including the proof of Theorem~\ref{thm:lyap}, builds on the previous
lemma.

In particular, one issue with Lemma~\ref{lem:t-1} is that the probability with
which the spectral median decreases is not large, since $ \alpha $ can not be
chosen arbitrarily close to one. However, the relatively small probability with
which the median decrease takes place is balanced by the fact that the decrease
happens over a very short time. This allows us to iterate the argument to
obtain a decrease with high probability over a slightly longer timescale. Here
we introduce a small parameter $ \delta \in (0, 1) $ and consider
the timescale
\begin{equ}[e:tsd]
 \tsd (M) = \ts (M) M^{\delta} \;.
\end{equ}
Furthermore, we introduce the
stopping time
\begin{equ}[e:sigma]
\sigma(M) = \inf \{ t \geqslant 0  \; \colon \; M^{(2)}(\varrho_{t}) >
M \} \;.
\end{equ}
The time horizon $ \tsd(M) $ is slightly longer than the time
horizon $ \ts(M)  $ in Lemma~\ref{lem:t-1}, but still shorter than the time $
M^{-1} $, provided $ \delta $ is sufficiently small and $ \kappa \ll M^{-1} $, by which 
time the stopping
time $ \sigma $  might kick in (that this stopping
times takes at least a time of order $ M^{-1} $ to kick in is essentially a consequence
of Lemma~\ref{lem:hl}). To guarantee that $ \tsd (M) \ll M^{-1} $
under the assumption $ M \geqslant \kappa^{- \q} $, we choose $
\delta $ such that
\begin{equ}
\tsd (M) = \lambda \kappa^{-1} M^{-2+ \delta} \log{(M)} \ll M^{-1}  \iff M^{-1+1/
\q + \delta} \log{(M)}\ll 1 \;,
\end{equ}
which leads to the constraint $ \delta < 1-1/\q $ in the lemma below. 

\begin{lemma}\label{lem:t-2}
Consider the same setting as in Theorem~\ref{thm:main}, with
$ \tsd  , \sigma, \tau $ as defined in \eqref{e:tsd}, \eqref{e:sigma} and
\eqref{e:tau-n}.
Then, for any $ \q > 2  $  and $ \delta \in (0, 1-1/\q) $ there
exist $ \r > 1 $ and $ \gamma \in (0, \a/2) $ such that for any $ p
\geqslant 0 $ and some $ C(p) > 0 $, and uniformly over $ \kappa \in (0, 1] $
\begin{equ}[e:mainBound]
\PP \left( \tau(M) \geqslant \tsd(M) \wedge \sigma(M) \right) \leqslant
C M^{- p} (\| u_{0} \|_{\mC^{\gamma}}+1)^{\r p}  \;, \quad
\forall M \geqslant M (\varrho_{0}) \vee \kappa^{- \q}\;.
\end{equ}
\end{lemma}
We observe that compared to the condition in Lemma~\ref{lem:t-1} we ask
that $ M \geqslant M (\varrho_{0}) $ instead of $ M \geqslant M^{(2)}
(\varrho_{0}) $. This is not a typo: we require this additional wiggling space
so that we can expect $ M \geqslant
M^{(2)}(\varrho_{t}) $ for all times $ t \in [0, \tsd(M) ] $ and
therefore apply Lemma~\ref{lem:t-1} repeatedly over that time interval.

\begin{proof}
Our approach to prove this result is to iterate about $ M^{\delta} $ times  the bound obtained in
Lemma~\ref{lem:t-1}. For this purpose, it will be convenient to introduce an
additional stopping time $\sigma_{\u}$, which controls excursions of the
velocity field $ u $, and is given by
\begin{equ}
\sigma_{\u}(M) = \inf \{ t \geqslant 0  \; \colon \; \| u_{t}
\|_{\mC^{\gamma}} \geqslant M^{\ve} (\| u_{0} \|_{\mC^{\gamma}} +1) \} \;,
\end{equ}
where $ \gamma \in (0, \a/2) $ is the same as in Lemma~\ref{lem:t-1}, and $ \ve
\in (0, 1) $ is a parameter that we will fix later on.
We then introduce the events $ \mA_{i} $ defined for all integers $ i
\geqslant 0 $ by
\begin{equ}
\mA_{i} = \{ \min\{\tau (M),\sigma(M),\sigma_\u(M)\} \geqslant   t_{i} \}
\;,
\end{equ}
where we set $ t_{i} = \ts(M) i $.
These events are chosen such that, setting $ n = \lfloor M^{\delta}
\rfloor $ and ignoring the argument $M$ of $\tsd, \tau, \sigma$ and $
\sigma_{\mathrm{u}} $ from now on, we have the inclusion
\begin{equ}
\{ \tau \geqslant  \tsd \wedge \sigma \} \subseteq
  \mA_{n} \cup \mB
\cup \mC  \;,
\end{equ}
where 
\begin{equ}
\mB = \{ \sigma < \tsd \} \;, \qquad \mC = \{ \sigma_{\u} <
\tsd \} \;.
\end{equ}
Then  the desired estimate will follow from proving that the events $ \mB
$ and $ \mC $ occur with suitably small probability, and by employing
Lemma~\ref{lem:t-1} to iteratively bound the probability of $ \mA_{n} $.

Let us start by analysing the events $ \mB $ and $ \mC $.
The first step is to show that the stopping time $
\sigma $ is unlikely to kick in before time $\tsd $.
Indeed, by Lemmas~\ref{lem:hl} and~\ref{lem:mmt-unif}, since $ M
\geqslant M(\varrho_{0}) $, we obtain that for any $ \gamma >2 $, any $q
\geqslant 1$, and some $
C , \overline{\r} > 1 $:
\begin{equs}
\PP (\mB) &= \PP \left( \sigma  < \tsd\right)  \leqslant  \PP \left( \sup_{0 \leqslant s \leqslant
\tsd} \frac{\| \mH_M \varrho_s \|}{\|\mL_M  \varrho_s \|} > 2 \right)
\leqslant  \PP \Bigl( \sup_{0 \leqslant s \leqslant
\tsd} C M s \| u_{s} \|_{\mC^{\gamma}}
>  1 \Bigr)  \\ 
&\lesssim (M \tsd)^{q} \EE \Bigl[ \sup_{0 \leqslant s
\leqslant \tsd } \|
u_{s} \|_{\mC^{\gamma}}^{q} \Bigr]
 \lesssim (\kappa^{-1} M^{-1 +  \delta} \log (M))^{q} (\| u_{0}
\|_{\mC^{\gamma}}+1)^{ \overline{\r} q}\;.\label{e:bd-unifrpt}
\end{equs}
Since $ \kappa^{-1} \leqslant
M^{\frac{1}{\q}} $ and $ \delta \in (0, 1 - 1/\q) $ by assumption, this implies the bound
\begin{equ}
\PP (\mB)\lesssim M^{- p} ( \|
u_{0} \|_{\mC^{\gamma}}+1)^{\r p}\;,
\end{equ}
for some $\r > 1$, provided that we choose $q$ large enough.

Similarly, the
stopping time $ \sigma_{\mathrm{u}} $ is also unlikely to kick in before time $
\tsd $. Indeed, by Markov's inequality and
Lemma~\ref{lem:mmt-unif}, for any $ \ve \in (0, 1) $ and $ \gamma \in
(0, \a/2) $ there exists an $ \overline{\r} > 1 $ such that for all $ p \geqslant 0 $
\begin{equ}[e:bd-C]
\PP( \mC) = \PP \left( \sigma_{\mathrm{u}} < \tsd \right) \lesssim_{p}
M^{- \ve p} (\| u_{0} \|_{\mC^{\gamma} }+ 1)^{ \overline{\r} p}\;.
\end{equ}

Both upper bounds \eqref{e:bd-unifrpt} and \eqref{e:bd-C} are of the right
order for the proof of the lemma, up to choosing larger $ \overline{\r} = 
\overline{\r} (\ve) >1$ in the case of
\eqref{e:bd-C}, so it remains to obtain a similar upper bound on
the event $ \mA_{n} $. We observe that the events $ \mA_{i} $ amount to a series of identical
unsuccessful trials, each of which has small probability of order $ 1 - \alpha -
CM^{- p} (\| u_{t_{i}}  \|_{\mC^{\gamma}}+1)^{\r p} $ by Lemma~\ref{lem:t-1}, for some $
\alpha \in (0, 1) $ and $ \r > 1 $. 
Indeed, the tower property and the fact that the $\mA_i$ are decreasing yields 
for every $k \leqslant n$
\begin{equ}
\PP \left(\mA_{k+1} \right) = \EE \left[\one_{\mA_{k}}\one_{\mA_{k+1}} \right] = \EE \left[ \one_{\mA_{k}}  \PP_{t_{k}}(\mA_{k+1}) \right] \;.
\end{equ}
On the event $ \mA_{k}$ we have $
\sigma  > t_{k}  $ by definition, so that $ M \geqslant M^{(2)} (\varrho_{t_{k}}) $ and we
are in the setting of the Lemma~\ref{lem:t-1}, which implies that for some $ \alpha,
\r, C >0 $
\begin{equ}
\one_{\mA_{k}}  \PP_{t_{k}}(\mA_{k+1}) \leqslant \one_{\mA_{k}} \alpha \;,
\end{equ}
provided that in addition $M \geqslant  \kappa^{- \q} \vee C (\|
u_{t_{k}} \|_{\mC^{\gamma}}+1)^{\r}$. Here the bound $ M \geqslant
\kappa^{-\q} $ is satisfied by our initial assumption on $ M $. 
Instead, for the second bound assume that for some $ \overline{\r} >
0 $ we have at the initial time $ M \geqslant C(\| u_{0} \|_{\mC^{\gamma}}+1)^{
\overline{\r}}  $. Then, because of the presence of the stopping time $
\sigma_{\mathrm{u}} $, we deduce $ M \geqslant C M^{- \ve \overline{\r}}(\|
u_{t_{k}} \|_{\mC^{\gamma}}+1)^{ \overline{\r}}  $, because $ \|
u_{t_{k}} \|_{\mC^{\gamma}} \leqslant M^{\ve} (\| u_{0} \|_{\mC^{\gamma}}+1) $.
Then the desired estimate $M \geqslant C (\|
u_{t_{k}} \|_{\mC^{\gamma}}+1)^{\r} $ follows (up to modifying the value of $
C > 0 $), for example by choosing $ \ve( \overline{\r}) = 1/ \overline{\r} $
and by assuming that $ \overline{\r} \geqslant 2 \r $.

This shows that, provided we choose $ \overline{\r} > 1 $ sufficiently large
and $ M \geqslant C (\| u_{0} \|_{\mC^{\gamma}} +1)^{ \overline{\r}} $, we have: 
\begin{equ}
\PP \left(\mA_{k+1} \right) = \EE\bigl[\PP_{t_{k}} ( \mA_{k+1})
\one_{\mA_{k}}\bigr] \leqslant \alpha \PP(\mA_{k}) \;.
\end{equ}
Hence, iterating this estimate $ n $ times and combining the previous calculation with \eqref{e:bd-unifrpt} and
\eqref{e:bd-C} we obtain for some $ \overline{\r} >0 $, any $ p \geqslant 0 $ and some $ C(p)> 0 $
\begin{equs}
\PP \left(\mA_{n} \right) & \leqslant \alpha^{n}
\;, \quad \PP(\mB)  \leqslant C M^{-p} (\| u_{0} \|_{\mC^{\gamma} }+ 1)^{
\overline{\r} p} \;, \quad
\PP (\mC )  \leqslant C M^{- p} (\| u_{0} \|_{\mC^{\gamma} }+ 1)^{
\overline{\r} p}\;.
\end{equs}
Now we can bound the first term by $
\alpha^{n} \lesssim_{\alpha,p, \delta} M^{- p} $ (recalling that $n = \lfloor
M^{\delta} \rfloor$). Therefore, overall,
up to choosing a sufficiently large $ C, \overline{\r} > 1 $ we have obtained
\begin{equ}
\PP \left( \tau \geqslant  \tsd  \wedge \sigma  \right) \leqslant
C M^{-p} (\| u_{0} \|_{\mC^{\gamma} }+ 1)^{ \overline{\r} p}\;.
\end{equ}
Note that we can drop the assumption $ M \geqslant C (\| u_{0}
\|_{\mC^{\gamma}} +1)^{ \overline{\r}}  $, because if it is not satisfied, then
the upper bound is in any case satisfied because the right-hand side is greater
than one, once more up to increasing the value of $ C $.
The result is therefore proven: note that in this proof we are using $ \overline{\r} $ to denote the
$ \r $ that is appearing in the statement of the present lemma, while we use $
\r $ to refer to the constant appearing in Lemma~\ref{lem:t-1}. 
\end{proof}

\subsection{Proof of Theorem~\ref{thm:lyap}} \label{sec:prof-main-rslt}
The proof of Theorem~\ref{thm:lyap} builds on a combination of the drift to low frequencies of the spectral median,
which we have obtained in Lemma~\ref{lem:t-2}, together with high frequency regularity
estimates for the projective process in Section~\ref{sec:regularity}. Recall that we write $ M_{t} $ for the median process $ M_{t} = M
(\varrho_{t}) $, and similarly $ M^{(2)}_{t} = M^{(2)} (\varrho_{t}) $.

To obtain Theorem~\ref{thm:lyap} we must 
control the $ H^{1} $ norm of the projective process at a fixed time $t_0$, 
but Lemma~\ref{lem:t-2} controls the median only for very short times.
To reach times of order one, we will simply iterate the bounds in
Lemma~\ref{lem:t-2}. Since it takes at most time $ M^{-2+\delta} $ (ignoring
for the moment the dependence on the viscosity coefficient, as well as
logarithmic terms) for the median to
decrease from $ M $ to $ M-1 $, the evolution of the spectral median follows
roughly that of the ODE $ \partial_{t} m_{t} \leqslant  - m^{2-\delta}_{t} $, which 
comes down form infinity in finite time as long as $\delta < 1$.
We are able to
prove the same effect for the spectral median: iterating Lemma~\ref{lem:t-2} we
will show that, no matter the initial value $ M(\varrho_{0}) $, one
has a bound of order one on  $M(\varrho_{t}) $ at times of order one.

To prove this, we combine one after the other the stopping times
introduced in the previous subsection, thereby constructing a new time
$ \eta $ by which
the median will have reached levels of order one with high probability.
Of course, when we refer to ``levels of order one'', we are neglecting the
dependence on the viscosity parameter $ \kappa $. Indeed, our arguments only show
that the spectral median will settle at some frequency below level $
\kappa^{- \q} $ for any $ \q > 2 $, which makes the final time
depend on the viscosity. We remark that the final time $\eta$ is not required to be defined as a stopping time in our proof.  We define $ \eta $ as follows.

\begin{definition}\label{def:eta-kappa}
For any $ \q > 2$ and $ \kappa \in (0, 1] $, we set $\hat M_i = (M_0 - i) \vee \kappa^{-\q}$ and define the  time $ \eta$ as follows:
\begin{itemize}
\item Set $ \tau_{0} =0 $ and then for every $ i \geqslant 0 $:
\begin{equs}[e:def-tau-sigma-i]
\sigma_{i} &= \inf \{ t \geqslant
\tau_{i}  \; \colon \; M^{(2)}_{t} \geqslant \hat M_i \}\wedge \{ 
\tau_{i} + t_{\star}^{\kappa, \delta} (\hat M_i)\} \;,\\
\tau_{i+1} &= \inf \{ t \geqslant  \tau_{i}  \;\colon \; M_{t} \leqslant \hat M_{i+1}\}  \;,
\end{equs}
where $ \tsd (M)$ is defined in \eqref{e:tsd}. 
\item Define $ L = \min \{i \geqslant 0\,:\, \hat M_{i} \leqslant \kappa^{- \q} \}$ and the ``final step'' $i_{\mathrm{fin}}= \min \{ i \geqslant 0  \; \colon \; \sigma_{i} < \tau_{i+1} \} \wedge L$.
\item Finally, we set
$\eta = \sigma_{i_{\mathrm{fin}}}$.
\end{itemize}
\end{definition}

Before we move on, let us try to give a more intuitive explanation of this construction.
Our objective is to show that, by time $ \eta $,
the process $ M_{t} $ will have dipped below level $ \kappa^{- \q} $, at least with
high probability. In order to reach below $
\kappa^{- \q} $ starting from $ M_{0} $, the median must 
cross $L$
integers. The stopping times $ \{ \tau_{i} \}_{i \in \NN} $ indicate exactly
the times at which the crossing of each integer happens for the first
time. Of course, these stopping times are somewhat arbitrary: the
median process could downcross a given level arbitrarily often, reaching
arbitrarily high values before dropping down to the next lower integer. If this
were the case, there would be no value in splitting up time in the intervals $[
\tau_{i}, \tau_{i+1}] $, since anything could happen in between.

However, such excursions are unlikely.
The probability of a large chunk of energy moving to higher frequencies during the interval $
[\tau_{i}, \tau_{i+1}] $ is controlled by the probability that the 2-quantile
$ M^{(2)}_{t} $ grows above level $ \hat{M}_{i}$. While any given quantile can vary
rapidly (because quantiles are dramatically susceptible to fluctuations in the
energy spectrum), considering different quantiles at the same time allows to
control the fluctuations easily, as was  done already in the proofs of
Lemma~\ref{lem:t-1} and Lemma~\ref{lem:t-2}. The stopping times $
\sigma_{i} $ also kick in if the hitting time of the next lower level is
unexpectedly long (longer than the timescale $ \tsd (M) $ of
Lemma~\ref{lem:t-2}). In this way, we obtain a deterministic upper bound on $
\eta $.

Finally, we will prove that with with high probability $ \tau_{i+1} <
\sigma_{i} $ for all $ i < L$, so that $
i_{\mathrm{fin}} = L $. Once we reach level $
\kappa^{-\q} $, we will additionally wait a time at most $\tsd (\hat{M}_{L}) =
\tsd(\kappa^{-\q})$, which is why $\eta = \sigma_{i_{\mathrm{fin}}}$. This
additional waiting time guarantees a lower bound on the duration of $\eta$, so that we can employ parabolic regularity estimates to bound the high frequency
regularity of $ \pi_{t} $. Note in this regard that we are not able to prove
any lower bound on the time it takes for the stopping times $ \tau_{i} $ to kick
in. 

\begin{remark}\label{rem:st}
It is possible that $ \tau_{i} = \tau_{i+1} $ for some $i$. When
this happens, although we know that $ M_{\tau_{i}} \leqslant \hat{M}_{i}$, it is
unclear exactly what value $ M_{\tau_{i}} $ takes.
However, when $ \tau_{i+1} > \tau_{i} $, then
$ M_{\tau_{i}} =  \hat{M}_{i}$. In addition, for every $ t \in [ \tau_{i},
\tau_{i+1} \wedge \sigma_{i}) $, we have $ M_{t}^{(2)} \leqslant \hat{M}_{i}$.
\end{remark}
Note that $\eta $ depends on the diffusivity and the initial median. However, we have a deterministic bound for $\eta$, uniformly over $\kappa\in(0,1]$ and $M_{0}\in \NN_{*}$.
\begin{lemma}\label{lem:ub-eta}
Consider $ \eta $ as in Definition~\ref{def:eta-kappa} with any $\delta
<1/4$ and
$ \q > 2 $.
Then there
exists a (deterministic) $t_{0}> 0$ such that uniformly over all $\kappa\in(0,1]$ and all $M_{0}\in\NN_{*}$, we have that
$\eta\leqslant t_{0}$.
\end{lemma}
\begin{proof}
From Definition~\ref{def:eta-kappa}, if $L\geqslant 1$, then we have $\hat{M}_{i}= M_{0}-i$ for all $0\leqslant i \leqslant L$ so that 
\begin{equs}
\eta\leqslant \sum_{i =0}^{L } \tsd(\hat{M}_{i})
& \lesssim \sum_{i=0}^{L } \kappa^{- 1} (\hat{M}_{i})^{-2 +
2\delta}  \lesssim \kappa^{-1} \int_{\kappa^{-\q}}^{\infty} x^{-2+2\delta} \ud x = \kappa^{-1} (\kappa^{-\q})^{-1+2 \delta} =  \kappa^{\q-2
\delta \q -1} \;,
\end{equs}
where we have absorbed the logarithmic factors into a small power $
(\hat{M}_{i})^{\delta} $. In the case $L=0$, we obtain the same bound using that $\hat{M}_{0} = \kappa^{-\q}$ and $\kappa\in(0,1]$. Therefore, provided $ \delta < 1/4$ and $ \q > 2 $, then we have  $ \q-2
\delta \q-1 >0$, so that $ \eta \leqslant t_{0}$ uniformly over $
\kappa\in(0,1]$ and all $M_{0}\in\NN_{*}$. 
\end{proof}
In this setting, the main result of this section is the following. Throughout
the remainder, we fix $\delta \in (0,1/4)$. 

\begin{proposition}\label{prop:stopped}
Fix any $ \bar{\q} > 4$,  and $
p \geqslant 0$. Then there exist $
C( p,\bar{\q}), \r > 1 $ and $ \gamma \in (0, \a/2) $ such that uniformly over $  \kappa \in (0, 1] $, $u_{0}\in \mC^{\gamma}$, we can estimate
\begin{equ}[e:wantedStopped]
\EE \left[ \| \pi_{\eta} \|^{2p}_{H^{1}} \right] \leqslant C(p,\bar{\q}) \kappa^{- \bar{\q} p} (\| u_{0} \|_{\mC^{\gamma}}+1)^{\r p} \;,
\end{equ}
where the time $ \eta $ is defined as in
Definition~\ref{def:eta-kappa} for any $ \q > 2 $ and any $ \delta \in
(0, 1/4) $.
\end{proposition}

\begin{proof}
To estimate the desired quantity, we partition over the events for all integers $i\geqslant 0$
\begin{equ}
\mA_{i} = \{ \eta=
\sigma_{i}\} \cap \{i_{\mathrm{fin}}=i<L  \} \;, \qquad \mB = \{ \eta = \sigma_{i_{\mathrm{fin}}}\} \cap\{i_{\mathrm{fin}}=L \} \;.
\end{equ} Our bound then works somewhat in analogy to the proof of Lemma~\ref{lem:t-2}.

We start by estimating (recall that $ L $ is deterministic while
$ i_{\mathrm{fin}} $ is not) via Cauchy--Schwarz
\begin{equs}
\EE \left[ \| \pi_{\eta} \|^{2p}_{H^{1}} \right] \leqslant & \sum_{i = 0}^{\infty} \EE \left[
\| \pi_{\eta} \|^{4p}_{H^{1}} 
 \one_{\mA_{i}} \right]^{\frac{1}{2}} \PP ( \mA_{i})^{\frac{1}{2}
}  + \EE \left[ \| \pi_{\eta} \|^{4p}_{H^{1}} \one_{\mB} \right]^{\frac{1}{2}} \PP (
\mB)^{\frac{1}{2}} \;,
\end{equs} 
Now, we bound each term in the sum. To lighten the notation we allow as usual the
value of generic constants such as $ C$ and $ \r $ to change increasingly from
line to line. We first bound the probability $
\PP (\mA_{i}) $.  Here we apply Lemma~\ref{lem:t-2}, which implies
that there exists an $ \r > 1 $, and for any $ q \geqslant  1 $ there exists a $C=
C(q)>0 $ such that
\begin{equ}
\PP_{\tau_{i}} (\mA_{i}) \leqslant C (\hat{M}_{i})^{- 2 q}(\| u_{\tau_{i}} \|_{\mC^{\gamma}}+1)^{2 \r q} \;,
\end{equ} where we have used that $\hat{M}_{i} \geqslant M_{\tau_{i}} \vee \kappa^{-\q} $ on $\mA_{i}$.
From this, via Lemma~\ref{lem:mmt-unif} we further deduce that up to
choosing larger $ C, \r > 1 $ we have
\begin{equ}[e:bd-PA]
\PP (\mA_{i}) \leqslant C (\hat{M}_{i})^{- 2 q}(\| u_{0} \|_{\mC^{\gamma}}+1)^{2 \r q} \;.
\end{equ}
Instead, for the $4p$th moment of the norm we proceed as follows. 
To conclude our estimate for the events $ \mA_{i} $, we control the
regularity of $ \pi_{t} $ through Proposition~\ref{prop:regu},
which yields
\begin{equ}
\EE_{\tau_{i}} \left[ \| \pi_{\eta}
\|^{4p}_{H^{1}} \one_{\mA_{i}} \right] \leqslant  \EE_{\tau_{i}}
\left[ \| \pi_{\sigma_{i}} \|^{4p}_{H^{1}} \right] \lesssim
(\hat{M}_{i})^{4p} (\| u_{0} \|_{\mC^{\gamma}}+1)^{2 \r p} \;,
\end{equ} 
where we again used that $\hat{M}_{i}\geqslant M_{\tau_{i}}\vee \kappa^{-\q}$ on $\mA_{i}$.
Overall, assuming that $q$ is sufficiently large so that  $ q > 2p+1 $ and $ (\kappa^{\q})^{q-2p-1} \leqslant 1 $,  we obtain
\begin{equs}
\sum_{i = 0}^{\infty} \EE \left[
 \| \pi_{\eta}
\|^{4p}_{H^{1}} \one_{\mA_{i+1}} \right]^{\frac{1}{2}} \PP ( \mA_{i})^{\frac{1}{2}}
& \lesssim \sum_{i =0}^{\lceil M_{0} - \kappa^{-\q} \rceil \vee 0 } (\hat{M}_{i})^{-q +
2p}  ( \| u_{0} \|_{\mC^{\gamma}}+1)^{\r p+\r q} \\
& \lesssim (\kappa^{\q})^{q - 2p-1}  (\| u_{0}
\|_{\mC^{\gamma}}+1)^{\r (p+ q) } \lesssim (\| u_{0} \|_{\mC^{\gamma}}+1)^{\r (p+ q) }\;.
\end{equs}
Finally, the term involving $ \mB $ is estimated similarly as before via
Proposition~\ref{prop:regu}. Indeed, on $\mB$ we have $ M_{\tau_{L}} \vee \kappa^{-\q}\leqslant \hat{M}_{L}\leqslant \kappa^{-\q} $ from the definition of $L$. The only difference is that we do not have
any upper bound on $ \PP(\mB) $, other than the trivial $ \PP (\mB) \leqslant 1
$. For this reason, our estimate explodes in $ \kappa $:
\begin{equs}
\EE_{\tau_{L}}  \left[  \| \pi_{\eta}
\|^{4p}_{H^{1}}  \one_{\mB} \right]^{\frac{1}{2}
} & \lesssim \kappa^{- 2p \q}  ( \| u_{0}
\|_{\mC^{\gamma} }+1)^{\r p}\;.
\end{equs}
Hence we obtain the desired moment estimate with $ \overline{\q} = 2 \q $.
\end{proof}
We are now ready to deduce Theorem~\ref{thm:lyap}.

\begin{proof}[of Theorem~\ref{thm:lyap}]
We start by fixing $ t_{0} $ as in Lemma~\ref{lem:ub-eta} as well as the time $ \eta $ from Definition~\ref{def:eta-kappa}, for some arbitrary
$ \q>2 $ (these values of $ t_{0} $ and $ \q $ are not yet the ones that
appear in the statement of the theorem). Then, since by Lemma~\ref{lem:ub-eta}
we have $ \eta \leqslant t_{0} $, Lemma~\ref{lem:ub-h1-Lns} implies that
\begin{equs}
\EE \|  \pi_{t_{0}} \|_{H^{1}}^{2} \lesssim  \EE \left[
\exp\left(C\int_{\eta}^{t_{0}} \|u_{r}\|_{H^{\gamma}}\ud r
\right)\left(\|\pi_{\eta}\|_{H^{1}}+ t_{0}
\sup_{0\leqslant r \leqslant t_0}\|u_{r}\|_{H^{\gamma +1}}\right)^{2} \right] \;,
\end{equs}
for any value of $ \gamma \in (2, 3) $. Now we can apply
Lemma~\ref{lem:lyap-sns} to bound the exponential of the norm, and we can use
Lemma~\ref{lem:mmt-unif} to bound moments of $ \| u \|_{H^{\gamma+1}} $ uniformly
in time. We therefore obtain
\begin{equ}
\EE \|  \pi_{t_{0}} \|_{H^{1}}^{2} \lesssim \EE  \left[ \| \pi_{\eta}
\|_{H^{1}}^{4}\right]^{\frac{1}{2}} + (\| u_{0}
\|_{\mC^{\gamma}}+1)^{\r} \;,
\end{equ}
for some $ \gamma \in (0, \a/2) $ (different from the previous one) and $ \r>1
$. Then by Proposition~\ref{prop:stopped}
we further estimate, for any $ \bar{\q} > 4 $
\begin{equ}
\EE \|  \pi_{\eta} \|_{H^{1}}^{4} \lesssim  \kappa^{- 2\bar{\q}} (\| u_{0}
\|_{\mC^{\gamma}}+1)^{2\r} \;.
\end{equ}
Note that all we did in this step was to reduce the moment bound at a
deterministic time to that at a stopping time with a deterministic upper bound.
The overall estimate that we have obtained
\begin{equ}
\EE \|  \pi_{t_{0}} \|_{H^{1}}^{2} \lesssim \kappa^{- \bar{\q}} (\| u_{0}
\|_{\mC^{\gamma}}+1)^{\r }
\end{equ}
is not quite enough for the statement of the theorem, since $ \| u_{0}
\|_{\mC^{\gamma}} $ is not bounded by the Lyapunov functional $ V_{\r} $ as
defined in \eqref{e:lyap-sns}. Therefore we consider a slightly longer time
horizon $ \overline{t}_{0} = t_{0}+1 $ and apply the Markov property together
with the first bound in Lemma~\ref{lem:mmt-unif} to obtain for some $
\overline{\r} > 1 $
\begin{equ}
\EE \|  \pi_{ \overline{t}_{0}} \|_{H^{1}}^{2} \lesssim \kappa^{- \bar{\q}} \EE
(\| u_{1} \|_{\mC^{\gamma}}+1)^{\r } \lesssim \kappa^{- \bar{\q}} (1 +
V_{ \overline{\r}}(u_{0})) \;.
\end{equ}
This is the desired estimate for the theorem, by setting $
\overline{t}_{0} = t_{0} $, $ \overline{\r}=\r $, and $ \overline{\q} = \q $.
\end{proof}

\section{Gaussian estimates for passive scalar advection} \label{sec:gaussian-estimates}
This section is devoted to the proof of Proposition~\ref{prop:bds-new}, by
analysing the terms in the expansion described in
Section~\ref{sec:chaos-expansion}. However, there are some subtle differences
between the case of passive scalar advection and that of linearised
Navier--Stokes. Therefore, we first treat the case of passive scalar advection:
this will allow us to set up the entire approach. We then later revisit this
approach in the context of linearised SNS, in Section~\ref{sec:gauss-sns}.

Furthermore, it will be convenient to rewrite
\eqref{e:main} and \eqref{e:sns} in Fourier coordinates, and introduce a
suitable setting for estimates of Gaussian integrals.
We start by representing the Leray projection of the driving noise $ \xi $ in
Fourier coefficients by
\begin{align*}
   \mathbf{P} \xi(t,x) = \sum_{k \in \ZZ^{2}_{*}}  e_{k}(x)
    \frac{\iota k^{\perp}}{|k|^{1+\frac{\a}{2}}} \partial_{t} \zeta^{k}_{t}\;,
\end{align*}
where $ k^{\perp} = (k^{(2)}, - k^{(1)}) $ and 
$ \{ \zeta^{k} \}_{k \in \ZZ^{d}} $ a sequence of i.i.d.\ complex Brownian
motions with covariance
   \begin{equ}[e:cov-zeta]
      \EE [ \partial_{t} \zeta^{k}_{t} \partial_{t}
      \zeta^{k^{\prime}}_{s}] = \sigma^{2} \delta (t-s) 1_{\{ k =- k^{\prime} \}} \;,\qquad
      \EE [ \partial_{t} \zeta^{k}_{t} \partial_{t}
      \bar \zeta^{k^{\prime}}_{s}] = \sigma^{2} \delta (t-s) 1_{\{ k =
k^{\prime} \}} \;.
   \end{equ}
Here the parameter $ \sigma > 0 $ is the same as in \eqref{e:assu-noise}.
However, since we are not interested in obtaining bounds that are uniform over
$ \sigma $, we fix without loss of generality $ \sigma =1 $ for the entirety of
the upcoming sections.

In this setting, we can for example represent the Gaussian term $ X $ appearing
in \eqref{e:def-X} via
\begin{equ}
X_{t}(x)  = \sum_{k \in \ZZ^{2}_{*} } e^{- | k |^{2} t} \hat{u}_{0}(k) +  \frac{1}{| k
|^{\frac{\a}{2}}} \left(  \int_{0}^{t} e^{- | k |^{2} (t-s)}
\ud \zeta^{k}_{s} \right) \frac{\iota k^{\perp}}{| k |} e_k(x) \;.
\end{equ}
For the sake of clarity, and to introduce
the notation that we will use later on, let us state the It\^o isometry in
our setting. For any $ t >0 $ we consider, on the space $$ E_{t}= [0, t] \times
\ZZ^{2}_{*} \;, $$ the product measure $ \ud s \otimes \ud k $ between Lebesgue on $ [0, t] $ and the
counting measure on $ \ZZ^{2}_{*} $. We will write $ L^{2}(E_{t}) $ for $
L^{2}(E_{t}, \ud s \otimes \ud k) $, where the latter is the real Hilbert space
defined, through the inner product $ \langle f, g \rangle_{E_{t}} = \int_{0}^{t}
\sum_{k} f(s, k) g (s, -k) \ud s $, as the set of complex-valued functions $
f $ such that $ \overline{f}(s, k) = f (s, - k) $ and $ \| f
\|_{L^{2}(E_{t})} < \infty $.
Then, for any $ f \in L^{2}(E_{t}) $ the following
stochastic integral is defined
\begin{equ}
\int_{E_{t}} f(s, k) \ud \zeta(s, k) = \sum_{k \in \ZZ^{2}_{*}}
\int_{0}^{t}f(s, k) \ud \zeta_{s}^{k} \;,
\end{equ}
where $ \zeta^{k}_{s} $ are the complex-valued Brownian motions from
\eqref{e:cov-zeta}. In particular
\begin{equ}[e:ito]
\EE \left\vert \int_{E_{t}} f(s, k) \ud \zeta(s, k) \right\vert^{2} = \| f
\|_{L^{2}(E_{t})}^{2} \;.
\end{equ}
We will also find it useful to denote the Malliavin derivative of the $k$-th Brownian motion at time $s<t$ by $\mathcal{D}_{s,k}$ so that $\mathcal{D}_{s,k}\zeta(r,\ell) = \delta(s-r)\delta_{k,\ell}$. In particular, we have for any $f\in L^2(E_t)$ that
\[
\mathcal{D}_{s,k} \int_{E_t} f(r,\ell)\dee\zeta(r,\ell) = f(s,k) \;.
\]
Now we are ready to work on the proof of Proposition~\ref{prop:bds-new}.

\subsection{Lower bound}

Now we are ready to start our analysis. The first step is to establish the
lower bound in Proposition~\ref{prop:bds-new}.
Recall the definition of $ \Phi, X $ and $ Y $ from
\eqref{e:def-phi}, \eqref{e:def-X} and \eqref{e:def-Y} respectively. Since we
are only interested in the contribution of the term in
$ \Phi[X, Y] $ let us further decompose $ X = X^{0} + \widetilde{X} $, with $
X^{0}_{t} = P_{t} u_{0} $ and $ \widetilde{X} $ the (Gaussian) stochastic
integral $ \widetilde{X} = \int_{0}^{t} P_{t-s} \mathbf{P} \xi \ud s $. Then we
find the following.

\begin{lemma}\label{lem:LBI1-new} 
Consider the map $ \Phi $ as defined in \eqref{e:def-phi} in the case $ L
[u] = u \cdot \nabla $, and fix any $ \beta \geqslant 1 $. Then
the following estimates hold uniformly over $ \varrho_{0} \in L^{2}_{*}, \kappa \in (0, 1] $ and $M \in
\NN_{*}, M \geqslant M^{(\beta)}(\varrho_{0}) $:
\begin{equ}
\EE \left\| \mL_{1} \, \Phi [ \widetilde{X}, Y]_{\ts (M)}
\right\|^{2} \gtrsim \kappa^{- 1} M^{ -6} \|
\varrho_{0} \|_{H^{- \frac{\a}{2}}}^{2} \;.
\end{equ}
\end{lemma}

\begin{proof}
We must estimate the following norm
\begin{equ}
\left\| \mL_{1}\, \Phi[ \widetilde{X}, Y]_{\ts (M)} \right\|^{2}= 2 \sum_{|\ell|\leqslant 1} |
\mF \Phi[ \widetilde{X}, Y]_{\ts (M)}(\ell) |^2 \;,
\end{equ}
where for $\mF \Phi[ \widetilde{X}, Y]_{t} (\ell)$, denotes the Fourier transform of $\Phi[ \widetilde{X}, Y]_{t}$ at frequency $\ell \in \Z^2_*$. Note that since $\tilde{X}$ is Gaussian, $\mF \Phi[ \widetilde{X}, Y]_{t} (\ell)$ can be written as a stochastic integral
\begin{equ}
   \mF \Phi[ \widetilde{X}, Y]_{t} (\ell) = \int_{E_{t}} f_t^\ell (s, k) \ud \zeta(s, k) \;,
\end{equ}
where, by Stroock's formula, the integrand $f_t^\ell$ is given by
\begin{equ}
f^{\ell}_t(s,k) = \mF\Phi[\mathcal{D}_{s,k}\widetilde{X}, Y]_t(\ell)\;.
\end{equ}
In this setting, we can write the Fourier transform as
\begin{equ}
\mF \Phi[\tilde{X},Y]_t(\ell) = - \sum_{j\in\Z^2_*}\int_0^t
e^{-\kappa|\ell|^2(t-r)}e^{-\kappa|\ell -
j|^2}\iota\ell\cdot\mF\widetilde{X}_r(j)\widehat{\varrho_0}(\ell-j)\dr \;.
\end{equ}
Therefore, since the Malliavin derivative of $\widetilde{X}$ is given by
\begin{equ}
\mathcal{D}_{s,k}\widetilde{X}(r,x) =
e^{-|k|^2(r-s)}|k|^{-\frac{\a}{2}-1}k^\perp e_k(x)\1_{\{0\leqslant s \leqslant
r\}} \;,
\end{equ}
we have 
\begin{equ}
   f^{\ell}_t(s,k) = - \iota\ell\cdot k^\perp
\widehat{\varrho_0}(\ell-k)|k|^{-\a/2-1}\int_s^te^{-\kappa|\ell|^2(t-r)}e^{-\kappa|\ell
- k|^2}e^{-|k|^2(r-s)}\dr \;.
\end{equ}
It follows by It\^{o}'s isometry that
\begin{equ}
   \E\|\mL_1\Phi[\widetilde{X},Y]_{t}\|^2 = \sum_{|\ell|\leqslant 1}\sum_{k\in\Z^2_*}\frac{|\ell\cdot k^\perp|^2|\widehat{\varrho_0}(\ell-k)|^2}{|k|^{\a +2}} \int_0^t\left|\int_s^te^{-\kappa|\ell|^2(t-r)}e^{-\kappa|\ell - k|^2}e^{-|k|^2(r-s)}\dr\right|^2\ds\;.
\end{equ}
Now we can bound the time integral as follows:
\begin{equs}
   \int_{0}^{t}  \left\vert \int_{s}^{t} e^{- \kappa (t-r)}
   e^{- \kappa |\ell- k|^{2}r} e^{- |k|^{2}(r -s)} \ud r 
   \right\vert^{2} \ud s &\gtrsim \int_{0}^{t}  \left\vert \int_{s}^{t} 
   e^{- (\kappa |\ell- k|^{2} + | k |^{2})r}  \ud r 
   \right\vert^{2} e^{2 |k|^{2} s} \ud s \\
   & \gtrsim \int_{0}^{t}  \left\vert \int_{0}^{t-s} 
   e^{- ( 1+ c\kappa) |k|^{2} r}  \ud r 
   \right\vert^{2} e^{-2c\kappa |k|^{2} s} \ud s \\
   & \gtrsim \left\vert \int_{0}^{t/2} 
   e^{- ( 1+ c\kappa) |k|^{2} r}  \ud r 
   \right\vert^{2} \int_{0}^{t/2}  e^{-2c\kappa |k|^{2} s} \ud s \\
   & \gtrsim \bigl||k|^{-2} \wedge t\bigr|^2 \bigl|\kappa^{-1} |k|^{-2} \wedge t\bigr| \;.
\end{equs}
Here the first inequality follows because $ \exp(- \kappa (t -r)) $ is
bounded from below for $ t \leqslant t_{\star}^{\kappa}(M) \leqslant 1 $, 
 the second inequality follows from the fact that  there exists $ c > 0 $ such 
 that $ |\ell -
k|^{2} \leqslant c |k|^{2} $ uniformly over $
k \in \ZZ^{2}_{*} $, and the last inequality follows from the fact that 
$\int_0^t e^{-as}\ds \gtrsim a^{-1} \wedge t$, uniformly over $a,t >0$.

If we restrict ourselves to $|k| \leqslant M+1$ and consider 
$t = t_{\star}^{\kappa} (M) \geqslant \kappa^{-1}M^{-2}$, we find that 
\begin{equ}
\bigl||k|^{-2} \wedge t\bigr|^2 \bigl|\kappa^{-1} |k|^{-2} \wedge t\bigr| 
\gtrsim \kappa^{-1} M^{-6}\;,
\end{equ}
so that overall
\begin{equs}
\E\|\mL_1 \Phi[\widetilde{X},Y]_{t}\|^2 &\gtrsim \kappa^{-1} M^{-6} \sum_{|\ell| \leqslant 1} \sum_{|k|\leqslant M+1}\frac{|\ell\cdot k^\perp|^2|\widehat{\varrho_0}(\ell-k)|^2}{|k|^{\a +2}}\\
 &\gtrsim \kappa^{-1} M^{-6} \sum_{\ell \leq 1}\sum_{|k|\leq M+1 }\frac{|\ell\cdot (k-\ell)^\perp|^2|}{|k- \ell|^2}\frac{\widehat{\varrho_0}(\ell-k)|^2}{|k-\ell|^{\a}}\\
 &\gtrsim  \sum_{|k|\leq M}\Biggl(\sum_{|\ell|\leq 1}\frac{|\ell\cdot k^\perp|^2}{|k|^2}\Biggr)\frac{|\widehat{\varrho_0}(k)|^2}{|k|^{\a}}\\
 &\gtrsim \|\mL_M \varrho_0\|_{H^{-\frac{\a}{2}}}^2 \;,
\end{equs}
where we have used the fact that $ k \neq \ell $ (otherwise $
\hat{\varrho}_{0} (\ell - k) = 0 $), implying $ | k | \sim | \ell -k | $, and
$ |\ell \cdot k^{\perp}| = | \ell\cdot(k-\ell)^\perp| $. Additionally, we have
\[
   \sum_{|\ell|\leq 1}\frac{|\ell\cdot k^\perp|^2}{|k|^2} \gtrsim 1 \;.
\]
The proof is complete upon noting that for $M \geq M^{(\beta)}(\varrho_0)$, we have $\|\mL_M \varrho_0\|_{H^{-\frac{\a}{2}}} \gtrsim_\beta \|\varrho_0\|_{H^{-\frac{\a}{2}}}$.
\end{proof}

\subsection{Upper bounds}\label{sec:upper-bounds}

Now we show that Lemma~\ref{lem:LBI1-new} is sufficient to deduce that by
time $ t_{\star}^{\kappa}(M) $ we have $ \| \mL_{1}
\varrho_{t_{\star}^{\kappa}(M)} \| \gtrsim \kappa^{-1/2} M^{
 -3} \| \varrho_{0} \|_{H^{- \frac{\a}{2}}} $ (provided $ M \geqslant M(\varrho_{0}) $ and $ \| \varrho_{0} \| =1 $).
For this, we show that our Gaussian approximation is not cancelled
by the remainder term, namely that $ \| \mL_{1}
\varrho_{t_{\star}^{\kappa}(M)}  \| \gtrsim \| \mL_{1}  \Phi[ \widetilde{X},
Y]_{\ts (M)}\| $. By the definition of $ \varrho $, we have
\begin{equs}
\varrho = Y + \Phi [u, \varrho] &= Y + \Phi[u, Y] + \Phi [u, \Phi[u,\varrho]] \\
& = Y+ \Phi [X, Y] + \Phi[\Psi[u,u], Y] + \Phi [u, \Phi[u,\varrho]] \;. \label{e:hdr}
\end{equs}
We have obtained an estimate for the first two terms (deterministic
effects, that depend only on the initial data do not affect our stochastic
lower bounds), and now we must show that the
remaining terms are of lower order through suitable upper bounds. However, this
requires some care, as analytic bounds (such as \eqref{e:bd-paraproduct}) alone do not seem sufficient and we must
leverage additional probabilistic cancellations.
The power of \eqref{e:bd-paraproduct} is
that it allows us for example 
to obtain a bound on the norm of
$ \varrho $ in $ H^{- \alpha} $ for $ \alpha < \frac{\a}{2} $. Indeed,
we find
\begin{equ}
\| \vr_{t} \|_{H^{- \alpha}} \leqslant \| \varrho_{0}
\|_{H^{- \alpha}} + \int_{0}^{t} \| P^{\kappa}_{t-s} (u_{s} \cdot \nabla
\varrho_{s}) \|_{H^{- \alpha}}  \ud s \;.
\end{equ}
Since $ u $ is incompressible, we further bound through Schauder estimates an
\eqref{e:bd-paraproduct}:
\begin{equ}
\| P^{\kappa}_{t-s} (u_{s}  \cdot \nabla \varrho_{s})
\|_{H^{- \alpha}}^{2} 
\lesssim  \kappa^{-1} (t-s)^{-1} \| \varrho_{s} \|_{H^{- \alpha}}^{2} \|
u_{s} \|_{\mC^{\beta}}^{2} \;. \label{e:grad-bd}
\end{equ}
Therefore, we conclude that
\begin{equ}
\sup_{0 \leqslant s \leqslant t} \| \vr_{s} \|_{H^{- \alpha}} \leqslant \|
\varrho_{0} \|_{H^{- \alpha}} + C  \sqrt{t / \kappa } \left( \sup_{0
\leqslant
s \leqslant t} \| \varrho_{s} \|_{H^{- \alpha}} \right) \left( \sup_{0
\leqslant s \leqslant t}
\| u_{s} \|_{\mC^{\beta}}\right) \;,
\end{equ}
and by a Gronwall-type argument we obtain the a priori bound 
\begin{equ}[e:a-priori-scalar]
\sup_{0 \leqslant s \leqslant t} \| \vr_{s} \|_{H^{- \alpha}} \leqslant C \|
\varrho_{0} \|_{H^{- \alpha}} \exp \left( C  t  \kappa^{-1} \sup_{0 \leqslant s \leqslant t} \|
u_{s} \|_{\mC^{\beta}}^{2}  \right)  \;.
\end{equ}
Now let us proceed by studying one of the two rest terms (the most difficult
one): our objective is to obtain an upper bound on $\| \mL_{1} \left[ \Phi [u,
\Phi[u, \varrho]]\right]_{t} \|$ at $ t = \ts (M) $.
Since $ u $ is incompressible we have $ u
\cdot \nabla (\varrho - Y) = \div(u (\varrho-Y)) $, so that we estimate:
\begin{equ}[e:tobdn]
\left\| \mL_{1} \Phi [u, \Phi[u, \varrho]]_{t} \right\| = \left\| \mL_{1}
\int_{0}^{t} \div \left( P^{\kappa}_{t-s} u_{s} (\varrho - Y)_{s} \right)\ud s\right\| \lesssim \| \mL_{1}
\widetilde{\Phi}[u, \Phi[u, \varrho]]_{t}\| \;,
\end{equ}
because we could absorb the derivative through the projection onto low
frequencies, and where for later convenience we have defined $ \widetilde{\Phi}
$ through   $ \Phi[u, \varrho] = \div ( \widetilde{\Phi}[u, \varrho]) $:
\begin{equ} 
\widetilde{\Phi}[\varphi, \psi]_{t} = - \int_{0}^{t} P^{\kappa}_{t -s}
\varphi_{s} \psi_{s}  \ud s \;.
\end{equ}
Now all of our future bounds will build on the following estimates on $ \Phi
$. To lighten the
notation, for any function $ \varphi \colon [0, t] \to \mC^{\beta} $
(respectively $ \varphi \colon [0, t] \to H^{\beta} $), for some
$ \beta \in \RR $, we write $ \| \varphi \|_{\mC^{\beta}_{t}} $ or $ \| \varphi
\|_{H^{\beta}_{t}} $ for the norms
\begin{equ}[e:cbt]
\| \varphi \|_{\mC^{\beta}_{t}} := \sup_{0 \leqslant s \leqslant t} \|
\varphi_{s} \|_{\mC^{\beta}} \;, \qquad \| \varphi \|_{H^{\beta}_t} := \sup_{0 \leqslant s \leqslant t} \|
\varphi_{s} \|_{H^{\beta}}\;.
\end{equ}

\begin{lemma}\label{lem:bds-lb}
Let $ \Phi $ be defined as in \eqref{e:def-phi}, with $ L[u] = u \cdot \nabla  $ as in
\eqref{e:form-PSA}.
For any $ t > 0 $ and $ 0 <\alpha < \beta $. Let $\varphi \in L^\infty([0,t],\mC^\beta)$ be a vector field such that $\div(\varphi) = 0$ and let $\psi\in L^\infty([0,t],L^2_*)$, the following estimates hold:
\begin{equs}
\| \mL_{1}\Phi[ \varphi, \psi]_t \| & \lesssim  t \| \varphi
\|_{\mC^{\beta}_t} \|
\psi \|_{H^{-\alpha}_t}\;,\\
 \| \Phi [\varphi,
\psi]_t \|_{H^{-\alpha}} & \lesssim \sqrt{t/ \kappa}\| \varphi
\|_{\mC^{\beta}_t}\| \psi\|_{H^{-\alpha}_t} \;.
\end{equs}
\end{lemma}
Both estimates follow from \eqref{e:bd-paraproduct} together with
the previous calculations.
A similar estimate is useful also for the map $ \Psi $ appearing in
\eqref{e:def-psi}: the following result is an immediate consequence of Schauder
estimates.

\begin{lemma}\label{lem:bd-psi}
Let $ \Psi $ be defined as in \eqref{e:def-psi}. Then for any $\beta>0$ and $t>0$, the following holds
\[
 \| \Psi [\varphi, \psi]
\|_{\mC^{\beta}_{t}} \lesssim \sqrt{t} \| \varphi \|_{\mC^{\beta}_{t}} \| \psi
\|_{\mC^{\beta}_{t}}\;.
\]
\end{lemma}
Armed with these estimates we can now proceed to bound \eqref{e:tobdn} and all
other terms in \eqref{e:hdr}. As for \eqref{e:tobdn}, we find for any pair $
\alpha < \beta < \a/2 $, via Lemma~\ref{lem:bds-lb}:
\begin{equ}[e:need-t-ref]
\left\| \mL_{1} \Phi [u, \Phi[u, \varrho]]_{t} \right\| \lesssim t \|
u \|_{\mC^{\beta}_{t}} \| \Phi[u, \varrho] \|_{H^{-\alpha}_t} \lesssim t^{\frac{3}{2}} \kappa^{- \frac{1}{2}} \| u
\|_{\mC^{\beta}_{t}}^{2} \| \varrho \|_{H^{-\alpha}_t} \;.
\end{equ}
When evaluated at time $ t = \ts(M) $ and through \eqref{e:a-priori-scalar}, we
see that the upper bound is of order $ M^{- 3 + \ve} \| \varrho_{0}
\|_{H^{- \alpha}}  $ for any $ \ve > 0 $ and $ \alpha < \a/2 $ (omitting the
dependence on $ \kappa $ and $ u $). Hence, this upper bound is just short of
matching the lower bound in Lemma~\ref{lem:LBI1-new}, and thus not sufficient
for our needs.
A similar estimate holds for the
term $ \Phi[u-X, Y] $ in the decomposition \eqref{e:hdr}, since for any $ \alpha <
\beta < \a/2 $ we estimate via Lemma~\ref{lem:bds-lb} and
Lemma~\ref{lem:bd-psi}:
\begin{equ}
\| \mL_{1} \left[ \Phi[\Psi[u,u], Y]_{t} \right] \| \lesssim t \|
\Psi[u, u]\|_{\mC^{\beta}_{t}} \| \varrho_{0} \|_{H^{-\alpha}}  \lesssim  t^{\frac{3}{ 2}} 
 \| \varrho_{0} \|_{H^{-\alpha}} \| u \|_{\mC^{\beta}_{t}}^{2} \;.
\end{equ}
Once more, even disregarding the viscosity, this is just barely not sufficient to obtain an upper bound of the correct
order, since its contribution is at time $ \ts (M) $ of order $ M^{ -3 +
\ve} \| \varrho_{0} \|_{H^{- \alpha}} $. Therefore, we must proceed to a finer decomposition.

\subsubsection{Higher order decomposition}
To obtain a more precise estimate on the energy at low frequencies, we can keep
expanding the remainder terms in \eqref{e:hdr}. 
Expanding $ \varrho $ up to third order through $ \Phi$ gives
\[
\varrho = Y + \Phi[u,Y] + \Phi[u,\Phi[u,Y]] + \Phi[u,\Phi[u,\Phi[u,\varrho]]] \;,
\]
and similarly expanding $u$ up to second order through $ \Psi $ yields
\[
u = X + \Psi[u,u] = X + \Psi[X,X] + \Psi[u+X,\Psi[u,u]] \;.
\]
We now want to expand the two terms $\Phi[u,Y]$ and $\Phi[u,\Phi[u,Y]]$ of the $\varrho$ expansion into stochastic terms of up second order and remainder terms of third order. First we use the third order expansion of $u$
\[
   \Phi[u,Y] = \Phi[X,Y] + \Phi[\Psi[X,X],Y] + \Phi[\Psi[u+X,\Psi[u,u]],Y] \;,
\]
and we use the first order expansion of $u$
\[
   \Phi[u,\Phi[u,Y]] = \Phi[X,\Phi[X,Y]] + \Phi[\Psi[u,u],\Phi[u,Y]] +
\Phi[u,\Phi[\Psi[u,u],Y]]\;.
\]
Finally, to obtain an optimal estimate in the viscosity
coefficient $ \kappa $ (therefore, this piece of the decomposition can be
omitted in a first read, and $ \Phi[u,\Phi[u,\Phi[u,\varrho]]]  $ can simply be
considered a rest term: the reader can see Remark~\ref{rem:further-dec} for an
additional explanation), we further decompose the term $  
\Phi[u,\Phi[u,\Phi[u,\varrho]]] $ as
\begin{equ}
\Phi[u,\Phi[u,\Phi[u,\varrho]]] = \Phi[X,\Phi[X,\Phi[X, Y]]] + \overline{R} \;,
\end{equ}
where
\begin{equs}
\overline{R} = & \Phi[u,\Phi[u,\Phi[u, \Phi[u, \varrho]]]] + \Phi[u, \Phi[u, \Phi[ \Psi[u,u],
Y]]] \\
& +\Phi[u, \Phi[\Psi[u,u], \Phi[X, Y]]]  +
\Phi[\Psi[u, u], \Phi[X, \Phi[X, Y]]] \;.
\end{equs}
Collecting everything we obtain the stochastic terms
\begin{equ}[e:stoch-terms]
   S =  \Phi[\Psi[X,X],Y]+ \Phi[X,\Phi[X,Y]] + \Phi[X,\Phi[X,\Phi[X, Y]]]\;,
\end{equ}
and the remainder terms
\begin{equ}[e:rest-terms]
   R = \Phi[\Psi[\Psi[u,u], u+X],Y] +
\Phi[\Psi[u,u], \Phi[u,Y]]+ \Phi[u,\Phi[\Psi[u,u],Y]] +  \overline{R} \;,
\end{equ}
so that overall we have the decomposition
\begin{equ}\label{e:decomp-rho}
\varrho = Y + \Phi[X,Y] + S + R \;.
\end{equ}
We treat these groups of terms in the two following subsections. For the
stochastic terms we follow a similar approach as for the lower bound obtained
in Lemma~\ref{lem:LBI1-new}. For the rest terms we use analytic upper bounds in
the spirit of Section~\ref{sec:upper-bounds}.

\subsection{Higher order Gaussian estimates} \label{sec:hog}

Let us start by obtaining upper bounds on the terms in \eqref{e:stoch-terms}.
Because terms in the zeroth (deterministic) and second homogeneous It\^o chaos
are uncorrelated from the first (Gaussian) chaos term appearing in
Lemma~\ref{lem:LBI1-new}, we need not address these terms:
they will be dealt with through an application of the Paley--Zygmund
inequality in the proof of Proposition~\ref{prop:bds-new}. Instead, we
concentrate on bounding the contribution of \eqref{e:stoch-terms} to the first
chaos. Therefore, let us further write $ X^{0}_{t}  = P_{t} u_{0}  $ and $
\widetilde{X} = X - X^{0} $. Then the
first chaos component of \eqref{e:stoch-terms} is given by
\begin{equ}
\Pi^{(1)}S = \Phi \bigl[ \widetilde{X},
\Phi[ X^{0}, Y] \bigr] + \Phi \bigl[ X^{0}, \Phi[ \widetilde{X}, Y] \bigr]+
2\Phi \bigl[ \Psi[ X^{0}, \widetilde{X}], Y \bigr]  + \Pi^{(1)}
\Phi[X,\Phi[X,\Phi[X, Y]]]\;,
\end{equ}
where $\Pi^{(1)}  $ is the projection on the first homogeneous Wiener--It\^o chaos.
In order to bound these terms we start with some preliminary results. 
Our next results are an improvement on (and a stochastic analogue to) the
deterministic bounds of Lemma~\ref{lem:bds-lb}.

\begin{lemma}\label{lem:Gaussian-n}
   Consider the map $ \Phi $ as in \eqref{e:def-phi} with $ L[u] =u \cdot
\nabla  $ as in \eqref{e:form-PSA}.
   For any $ t > 0 $ and $ 1< \alpha \leqslant  \a/2 $, we estimate uniformly over $\kappa\in(0,1]$ and $\psi \in H^{- \alpha }_{t}$:
   \begin{equ}
      \E\|\mL_1\Phi[\widetilde{X},\psi]_{t}\|^2\lesssim
t^3 \|\psi\|_{H^{- \frac{\a}{2} }_{t}}^2\;, \qquad
\E\|\Phi[\widetilde{X},\psi]_{t}\|^2_{H^{- \alpha}} \lesssim
t^2 \kappa^{-1}  \|\psi\|_{H^{- \alpha }_{t}}^2 \;.
   \end{equ}
\end{lemma}
We observe that Lemma~\ref{lem:bds-lb} would only imply $ \| \mL_{1} [
\widetilde{X}, \psi]_{t} \| \lesssim t \| \widetilde{X} \|_{\mC^{\beta}} \|
\psi \|_{H^{- \alpha}_{t}}$ for any $ \alpha < \beta < \a/2 $. Therefore, we
are improving the bound by a power of $ t $. Similarly, the second estimate is
improved by a power $ \sqrt{t} $.
\begin{proof}
For any $\ell\in \Z^2_*$ we write the Fourier transform of $ \Phi
[ \widetilde{X}, \psi]_{t} $ as a
stochastic integral:
\[
\mF \Phi[\widetilde{X},\psi]_t(\ell) = \int_{E_t} f_t^\ell(s,k) \ud
\zeta(s,k)\;.
\]
Then, let $\mathcal{D}_{s,k}$ denote the Malliavin derivative of the $k$-th Brownian
motion at time $s$. By Stroock's formula, we have for every $s \in [0,t]$ and $k \in \Z^2_*$:
\[
f_t^\ell(s,k) = \mF \Phi[\mathcal{D}_{s,k}\widetilde{X},\psi]_t(\ell) \;,
\]
where
\[
\mathcal{D}_{s,k}\widetilde{X} (r ,x)=
e^{-|k|^2(r-s)}|k|^{-\frac{\a}{2}}\sigma_k(x)\1_{\{0\leq s\leq r\}}\;, \quad \sigma_k(x) = \frac{\iota k^\perp}{|k|}e^{\iota k\cdot x}\;.
\]
It follows by It\^{o}'s isometry that
\[
   \E\|\mL_1\Phi[\widetilde{X},\psi]_{t}\|^2 = \sum_{|\ell|\leq 1}\sum_{k\in\Z^2_*}\int_0^t|\mF \Phi[\mathcal{D}_{s,k}\widetilde{X},\psi]_t(\ell)|^2\ds = \sum_{k\in \Z^2_*}\int_0^t\|\mL_1 \Phi[\mathcal{D}_{s,k}\widetilde{X},\psi]_t\|^2\ds\;.
\]
Now, in our setting
\begin{equ}
\Phi[\mathcal{D}_{s,k}\widetilde{X},\psi]_t = \int_{s}^{t} | k |^{-
\frac{\a}{2}} e^{- | k |^{2} (r -s)}\div (P_{t -r}^{\kappa} (\sigma_{k} \psi_{r}) )\ud r
\;,
\end{equ}
so that we can overall estimate
\begin{equs}
 \E\|\mL_1\Phi[\widetilde{X},\psi]_{t}\|^2 & \leqslant  \int_{0}^{t} \sum_{k \in \ZZ^{2}_{*}}
\frac{1}{| k |^{\a}}  \left( \int_{s}^{t}  e^{- | k |^{2} (r -s)} \|
\mL_{1} \sigma_{k} \psi_{r} \| \ud r \right)^{2} \ud s \\
& \lesssim t \int_{0}^{t} \int_{s}^{t}  \sum_{k \in \ZZ^{2}_{*}}
\frac{1}{| k |^{\a}}  \|
\mL_{1} \sigma_{k} \psi_{r} \|^{2} \ud r \ud s \\
& \lesssim t \int_{0}^{t} \int_{s}^{t}  \sum_{k \in \ZZ^{2}_{*}}
\frac{1}{| k |^{\a}}  \|
\sigma_{k} \psi_{r} \|^{2}_{H^{- \frac{\a}{2} }} \ud r \ud s \lesssim t^{3} \|
\psi \|_{H^{- \frac{\a}{2} }_{t}}^{2} \;,
\end{equs}
where we used Jensen in the second line and the last step follows from the norm equivalence in
Lemma~\ref{lem:equivalencen}: note that $ \a/2 >1$ from our assumption $ \a > 10
$. 

The second estimates follows in exactly the same way, the only difference being
that we estimate
\begin{equ}
\| \Phi[ \mD_{s, k} \widetilde{X}, \psi]_{t} \|_{H^{- \alpha}} \leqslant
\kappa^{- \frac{1}{2}} \int_{s}^{t} | k |^{- \frac{\a}{2}} (t -r)^{- \frac{1}{2}}  \| \sigma_{k} \psi_{r}
\|_{H^{- \alpha}}   \ud r\;.
\end{equ}
From here we then bound following the same steps as before:
\begin{equs}
 \E\|\Phi[\widetilde{X},\psi]_{t}\|^2_{H^{- \alpha}}  & \leqslant
\kappa^{- 1}  \int_{0}^{t} \sum_{k \in \ZZ^{2}_{*}}
\frac{1}{| k |^{\a}}  \left( \int_{s}^{t}  (t -s)^{- \frac{1}{2}} \|
 \sigma_{k} \psi_{r} \|_{H^{- \alpha}} \ud r \right)^{2} \ud s \\
& \lesssim \kappa^{-1} t^{\frac{1}{ 2}} \int_{0}^{t}   \int_{s}^{t}  (t -s)^{- \frac{1}{2}} \sum_{k \in \ZZ^{2}_{*}}
\frac{1}{| k |^{\a}}\|
 \sigma_{k} \psi_{r} \|_{H^{- \alpha}}^{2} \ud r \ud s \lesssim
\kappa^{-1} t^{2} \| \psi \|_{H^{- \alpha}_{t}}^{2} \;.
\end{equs}
This completes the proof.
\end{proof}
As a consequence of the previous result, we are now able to bound the first
chaos terms appearing in \eqref{e:stoch-terms}.

\begin{proposition}\label{prop:chaos-together}
Fix $ \q >1 $ and consider the map $ \Phi $ as in \eqref{e:def-phi} with $ L [u] = u \cdot \nabla
$.  For any $ 1 <\alpha < \beta < \a/2 $ we have, uniformly
over $ \kappa \in (0, 1] $ and $ M \geqslant \kappa^{-\q} $,
\begin{equ}
\EE \bigl\| \mL_{1}  \bigl[  \Pi^{(1)}S_{\ts (M)} \bigr] \bigr\|^{2}\lesssim
(\ts (M))^{4} \kappa^{-1} (\| u_{0}
\|+1)^{4}_{\mC^{\beta}} \| \varrho_{0} \|_{H^{- \alpha}}^{2} \;.
\end{equ} 
\end{proposition}
Because we follow slightly different arguments to bound the different terms, we
have split the proof of the proposition into 
Lemmas~\ref{lem:first-chaosn}, \ref{lem:Frst-Chaos-2n}, and~\ref{lem:third-psa} below. Note that the upper bound in
Lemma~\ref{lem:third-psa} is of order $ \ts(M)^{5} \kappa^{-2} \lesssim
\ts(M)^{4} \kappa^{-1} $ under the assumption $ M \geqslant \kappa^{- \q} $ for
some $ \q > 1 $. We start by bounding the first two terms in
\eqref{e:stoch-terms}.

\begin{lemma}\label{lem:first-chaosn}
Consider the map $ \Phi $ as in \eqref{e:def-phi} with $ L [u] = u \cdot \nabla
$.  For any $ 1 <\alpha < \beta < \a/2 $, we can estimate uniformly
over $ \kappa \in (0, 1] $:
\begin{equs}
\EE \bigl\| \mL_{1}  \bigl[  \Phi \bigl[ \widetilde{X},
\Phi[ X^{0}, Y] \bigr]_{\ts (M)}+ \Phi \bigl[ X^{0}, \Phi[ \widetilde{X}, Y]
\bigr]_{\ts (M)} \bigr] \bigr\|^{2}\lesssim (\ts (M))^{4} \kappa^{-1} \| u_{0}
\|^{2}_{\mC^{\beta}} \| \varrho_{0} \|_{H^{- \alpha}}^{2} \;.
\end{equs} 
\end{lemma}

\begin{proof}
Combining Lemma~\ref{lem:Gaussian-n} with
Lemma~\ref{lem:bds-lb} we estimate:
\begin{equ}
\EE \| \mL_{1}\Phi \bigl[ \widetilde{X},
\Phi[ X^{0}, Y] \bigr]_{t} \|^{2} \lesssim t^{3} \| \Phi[ X^{0}, Y]
\|_{H^{- \frac{\a}{2}}}^{2} \lesssim t^{4} \kappa^{-1} \| u_{0}
\|_{\mC^{\beta}}^{2} \| \varrho_{0} \|_{H^{- \frac{\a}{2}}}^{2} \;.
\end{equ}
And similarly for the second term:
\begin{equ}
\EE \| \mL_{1} \Phi \bigl[ X^{0}, \Phi[ \widetilde{X}, Y]_{t} \bigr]\|^{2} \lesssim
t^{2} \| u_{0} \|_{\mC^{\beta}}^{2} \sup_{0 \leqslant s \leqslant t}
\EE \| \Phi[ \widetilde{X}, Y]_{s} \|_{H^{- \alpha}}^{2} \lesssim t^{4} \kappa^{-1}
\| u_{0} \|_{\mC^{\beta}}^{2} \| \varrho_{0} \|_{H^{- \alpha}}^{2}\;.
\end{equ}
This completes the proof.
\end{proof}
Similarly, we can bound the third term in \eqref{e:stoch-terms}.

\begin{lemma}\label{lem:Frst-Chaos-2n}
Consider the map $ \Phi $ as in \eqref{e:def-phi} with $ L [u] = u \cdot \nabla
$.
For any $ 1 < \alpha  < \a/2 $, we can estimate uniformly
over $ \kappa \in (0, 1] $:
\begin{equ}
\EE \bigl\| \mL_{1}  \bigl[  \Phi \bigl[ \Psi [X^{0}, \widetilde{X}], Y
\bigr]_{\ts (M)} \bigr] \bigr\|^{2}\lesssim (\ts (M))^{4}  \| u_{0} \|^{2}_{\mC^{\alpha}}  \| \varrho_{0}
\|_{H^{- \alpha}}^{2}\;.
\end{equ} 
\end{lemma}
\begin{proof}
As in the previous lemma, we find that
\begin{equs}
\EE \bigl\| \mL_{1}  \bigl[  \Phi \bigl[ \Psi [X^{0}, \widetilde{X}], Y
\bigr]_{t} \bigr] \bigr\|^{2}& = \sum_{|\ell|\leq
1}\sum_{k\in\Z^2_*}\int_0^t|\mF
\Phi \bigl[ \Psi [X^{0}, \mD_{s, k} \widetilde{X} ], Y
\bigr]_{t}(\ell)|^2\ds \\
& = \sum_{k\in
\Z^2_*}\int_0^t\|\mL_1 \Phi \bigl[ \Psi [X^{0}, \mD_{s, k} \widetilde{X} ], Y
\bigr]_{t}\|^2\ds \;.
\end{equs}
Then, following the same type of argument as in the previous upper bounds and using
Cauchy--Schwarz, we obtain the following bound
\[
\E\|\mL_1\Phi[\Psi[X^0,\widetilde{X}],Y]_t\|^2 \leq \sum_{k \in
\ZZ^{2}_{*}} \int_0^t\biggl(\int_s^t\int_s^{r}g(k, s,r,\tau)\dee \tau\dee r\biggr)^2\dee s\,,
\]
where $g$ is given by
\[
g(k, s,r,\tau) = |k|^{- \frac{\a}{2} }\|\mL_1[(\Div P_{r-\tau}\mathbf{P}(X^0_\tau\otimes_s\sigma_k))Y_{r}]\|\,,
\]
with $ \mathbf{P} $ the Leray projection. Now we bound
\begin{equs}
\|\mL_1[(\Div P_{r-\tau}\mathbf{P}(X^0_\tau\otimes_s\sigma_k))Y_{r}]\| & \lesssim
\|[ \Div P_{r-\tau}\mathbf{P}(X^0_\tau\otimes_s\sigma_k)] e_{- k}\|_{H^{\alpha}}
\| e_{k}Y_{r} \|_{H^{-
\alpha}} \\
& \lesssim (r - \tau)^{- \frac{1}{ 2}}  \| u_{0} \|_{H^{\alpha}} \|
e_{k} \varrho_{0} \|_{H^{- \alpha}} \;.
\end{equs}
Here we have used that
\begin{equs}
\|[ \Div P_{r-\tau}\mathbf{P}(X^0_\tau\otimes_s\sigma_k)] e_{-
k}\|_{H^{\alpha}}^{2} & = \sum_{l \in \ZZ^{2}_{*}} | l |^{2 \alpha} | \mF [ \Div
P_{r-\tau}\mathbf{P}(X^0_\tau\otimes_s\sigma_k)]|^{2} (l + k) \\
& \lesssim (r - \tau)^{-1}\sum_{l \in \ZZ^{2}_{*}} | l |^{2 \alpha} | \mF
(X^0_\tau\otimes_s\sigma_k)|^{2} (l + k)\\
& \lesssim (r - \tau)^{-1}\sum_{l \in \ZZ^{2}_{*}} | l |^{2 \alpha} | \mF
X^0_\tau |^{2} (l) \lesssim (r - \tau)^{- 1}\| u_{0} \|_{H^{\alpha}}^{2}
\;.
\end{equs}
Hence applying Jensen with the measure $(r-\tau)^{-1/2}\dee\tau\dr$ we obtain overall
\begin{equs}
\E\|\mL_1\Phi[\Psi[X^0,\widetilde{X}],Y]_t\|^2 & \lesssim t^{\frac{3}{2} 
}  \int_0^t\int_s^t\int_s^{r} (r - \tau)^{- \frac{1}{2}} \sum_{k \in
\ZZ^{2}_{*}} | k |^{- \a}  \|
u_{0} \|_{H^{\alpha}}^{2} \| e_{k}
\varrho_{0} \|_{H^{- \alpha}}^{2}\dee \tau\dee r \dee s \\
& \lesssim t^{4} \| u_{0} \|_{H^{\alpha}}^{2} \| \varrho_{0} \|_{H^{-
\alpha}}^{2} \;,
\end{equs}
where the last estimate follows from the
second norm equivalence in Lemma~\ref{lem:equivalencen}.
This is the desired bound, when evaluated at $ t = \ts(M) $.
\end{proof}
And finally we can also bound the last term in \eqref{e:stoch-terms}.
\begin{lemma}\label{lem:third-psa}
Consider the map $ \Phi $ as in \eqref{e:def-phi} with $ L [u] = u \cdot \nabla $.
For any $ 1 < \alpha  < \a/2 $, we can estimate uniformly
over $ \kappa \in (0, 1] $:
\begin{equ}
\EE \bigl\| \mL_{1}  \bigl[  \Pi^{(1)}
\Phi[X,\Phi[X,\Phi[X, Y]]]_{\ts (M)} \bigr] \bigr\|^{2}\lesssim (\ts (M))^{5} \kappa^{-2} \| \varrho_{0}
\|_{H^{- \alpha}}^{2} (\| u_{0} \|^{4}_{\mC^{\alpha}}+1)\;.
\end{equ} 

\end{lemma}
\begin{proof}
The first step in proving this result is to write out the first chaos
components of the process $ \Phi[X,\Phi[X,\Phi[X, Y]]] $. Since $ X =
X^{0} + \widetilde{X} $, there are terms that appear as a combination of one $
\widetilde{X} $ and two $ X^{0} $, and other three terms that appear from the
contractions between pairs of $ \widetilde{X} $. In this setting it will be
convenient to write $ \overline{X} $ for an independent copy of $ \widetilde{X}
$ and $ \overline{\EE} $ for the expectation over $ \overline{X} $.  Then
\begin{equs}
\Pi^{(1)}\Phi[X,\Phi[X,\Phi[X, Y]]] = & \Phi[
\widetilde{X},\Phi[X^{0},\Phi[X^{0}, Y]]]+ \Phi[
X^{0},\Phi[ \widetilde{X},\Phi[X^{0}, Y]]]\\ 
&+ \Phi[ X^{0},\Phi[X^{0},\Phi[ \widetilde{X}, Y]]] + \overline{\EE}\Phi[
\widetilde{X} ,\Phi[ \overline{X},\Phi[ \overline{X}, Y]]] \\
& +  \overline{\EE}\Phi[
\overline{X} ,\Phi[ \widetilde{X},\Phi[ \overline{X}, Y]]] + \overline{\EE}\Phi[
\overline{X} ,\Phi[ \overline{X},\Phi[ \widetilde{X}, Y]]] \;.
\end{equs}
Now each of these terms can be treated similarly, for any fixed $ \alpha <
\beta < \a/2 $. For the first term we find
via Lemma~\ref{lem:Gaussian-n} and Lemma~\ref{lem:bds-lb}:
\begin{equ}
\EE \bigl\| \mL_{1}  \bigl[ \Phi[
\widetilde{X},\Phi[X^{0},\Phi[X^{0}, Y]]]_{t} \bigr] \bigr\|^{2} \lesssim
t^{3} \| \Phi[X^{0},\Phi[X^{0}, Y]] \|_{H^{- \frac{\a}{2}}_{t}}^{2} \lesssim
t^{5} \kappa^{- 2} \| u_{0} \|_{\mC^{\beta}}^{4} \| \varrho_{0}
\|_{H^{- \alpha}}^{2} \;.
\end{equ}
For the second term
\begin{equs}
\EE \bigl\| \mL_{1}  \bigl[ \Phi[
X^{0},\Phi[ \widetilde{X},\Phi[X^{0}, Y]]]_{t} \bigr] \bigr\|^{2} & \lesssim
t^{2} \| u_{0} \|_{\mC^{\beta}}^{2} \sup_{0 \leqslant s \leqslant t} \EE  \| \Phi[
\widetilde{X},\Phi[X^{0}, Y]]_{s}\|_{H^{- \alpha}}^{2} \\
& \lesssim t^{4} \kappa^{- 1} \| u_{0} \|_{\mC^{\beta}}^{2} \| \Phi[X^{0}, Y]
\|_{H^{- \alpha}_{t}}^{2} \lesssim t^{5} \kappa^{- 2} \| u_{0}
\|_{\mC^{\beta}}^{4} \| \varrho_{0} \|_{H^{- \alpha}}^{2}\;.
\end{equs}
And similarly for the third term
\begin{equs}
\EE \bigl\| \mL_{1}  \bigl[ \Phi[ X^{0},\Phi[X^{0},\Phi[ \widetilde{X},
Y]]]_{t} \bigr] \bigr\|^{2} & \lesssim t^{3} \kappa^{-1} \| u_{0}
\|_{\mC^{\beta}}^{4} \sup_{0 \leqslant s \leqslant t} \EE \|\Phi[ \widetilde{X},
Y]_{s} \|^{2}_{H^{- \alpha}}\\
& \lesssim t^{5} \kappa^{- 2} \| u_{0}
\|_{\mC^{\beta}}^{4} \| \varrho_{0} \|_{H^{- \alpha}}^{2}\;.
\end{equs}
Now we are left with the terms that appear from contractions. These are however
bounded in exactly the same way as the previous three. For clarity we bound the
first term only. We find via Jensen's inequality and the same calculation as
for the first term above:
\begin{equs}
\EE \bigl\| \mL_{1}  \bigl[  \overline{\EE}\Phi[
\widetilde{X} ,\Phi[ \overline{X},\Phi[ \overline{X}, Y]]]_{t} \bigr]
\bigr\|^{2} & \leqslant \overline{\EE} \EE \bigl\| \mL_{1}  \bigl[  \Phi[
\widetilde{X} ,\Phi[ \overline{X},\Phi[ \overline{X}, Y]]]_{t} \bigr]
\bigr\|^{2} \\
& \lesssim t^{5} \kappa^{- 2}  \| \varrho_{0}
\|_{H^{- \alpha}}^{2} \overline{\EE} \left[  \| u_{0} \|^{4}_{\mC^{\beta}}
\right] \lesssim t^{5} \kappa^{- 2} \| \varrho_{0} \|_{H^{- \alpha}}^{2} \;.
\end{equs}
All other terms follow in a similar way. Hence, the proof is complete.
\end{proof}
The last element of this subsection is a deterministic result regarding the
equivalence of certain Sobolev norms, which is useful in the preceding
stochastic estimates.

\begin{lemma}\label{lem:equivalencen}
The following equivalence between norms holds for any $ \alpha >1 $
\begin{equ}
\sum_{k \in \ZZ^{2}_{*}} | k |^{-2 \alpha} \| e_{k} \varphi \|_{H^{- \alpha}}^{2} \simeq
\| \varphi \|^{2}_{H^{- \alpha} }  \;.
\end{equ}
\end{lemma}
\begin{proof}
Let us start with the first estimate. Note that $\sum_{k \in
\ZZ^{2}_{*}} | k |^{-2 \alpha} \| e_{k} \varphi \|_{H^{- \alpha}}^{2}
\gtrsim \| \varphi \|^{2}_{H^{- \alpha} } $ is trivial by choosing $ | k | \simeq 1 $.
For the upper bound we have instead that
\begin{equs}
\sum_{k\in \Z^2_*}|k|^{- 2\alpha } \|e_k \varphi \|_{H^{-\alpha}}^2 & \lesssim
\sum_{\ell,k \in \Z^2_{*}\atop \ell \neq j}
|k|^{- 2 \alpha }|\ell|^{-2\alpha}|\widehat{\varphi}(\ell-k)|^2
\leqc \sum_{\ell, k \in \Z^2_{*}\atop \ell \neq j}
|\ell|^{-2\alpha}|\ell-k|^{-2\alpha}|\widehat{\varphi} (\ell-k)|^2\;.
\end{equs}
where in the last inequality we used symmetry of the summand, together with the fact that
\[
   |k|^{- 2 \alpha }|\ell|^{-2\alpha} =  |k|^{-
2\alpha}|\ell|^{-2\alpha}|\ell-k|^{-2\alpha}|\ell-k|^{2\alpha} \leqc
|\ell|^{-2\alpha}|\ell-k|^{-2\alpha} + |k|^{-2\alpha}|\ell-k|^{-2\alpha}\;,
\] 
where the last inequality holds via the triangle inequality $ | \l - k |
\leqslant | k | + | \l - k | $. Finally, we obtain as desired
\begin{equ}
\sum_{\ell, k \in \Z^2_{*}\atop \ell \neq j}
|\ell|^{-2\alpha}|\ell-k|^{-2\alpha}|\widehat{\varphi} (\ell-k)|^2 = \| \varphi
\|_{H^{- \alpha}}^{2}\sum_{\ell \in \Z^2_{*}}
|\ell|^{-2\alpha} \lesssim \| \varphi \|_{H^{- \alpha}}^{2} \;,
\end{equ}
where we have used that $ \alpha > 1 $ and we are in dimension $ d=2 $.
\end{proof}
This concludes our stochastic estimates, and we can move on to bound the
remainder terms through analytic estimates.

\subsection{Upper bound on the remainder terms} \label{sec:ubg}
Now we are left with bounding the remainder terms appearing in
\eqref{e:rest-terms}. Here our estimates build only on Lemma~\ref{lem:bds-lb}
and Lemma~\ref{lem:bd-psi}.

\begin{lemma}\label{lem:rest-complete}
Fix any $ \q >1 $. Consider the map $ \Phi, \Psi $ as in \eqref{e:def-psi} and\eqref{e:def-phi} with $ L [u] = u \cdot \nabla
$, and let $ R $ be defined as in \eqref{e:rest-terms}.  For any $ \alpha <
\beta < \a/2 $, there exists a $ C > 0 $ such
that we can estimate uniformly over $ M \geqslant \kappa^{-\q} $ and $ \kappa \in (0, 1] $:
\begin{equs}
\left\| R_{\ts (M)} \right\| \lesssim (\ts (M))^{2} \kappa^{-
\frac{1}{2} }  \| \varrho_{0} \|_{H^{- \alpha}}  \exp \left( C \ts(M) \kappa^{-1} \| u
\|_{\mC^{\beta}_{t}}^{2} \right) ( \| X \|_{\mC^{\beta}_{t}}+ \| u
\|_{\mC^{\beta}_{t}} )^{4} \;.
\end{equs} 
\end{lemma}

\begin{remark}\label{rem:further-dec}
The reason for considering $ \overline{R} $ rather than $ \Phi[u, \Phi[u,
\Phi[u, \varrho]] $ in the decomposition \eqref{e:rest-terms}, is that the
upper bound on the latter term we obtain is of order
\begin{equ}
\left\| \mL_{1}  \left[  \Phi[u, \Phi[u, \Phi[u, \varrho]]]_{t} \right]
\right\| \lesssim t^{2} \kappa^{- 1 }  \| u
\|_{\mC^{\beta}_{t}}^{3} \| \varrho \|_{H^{- \alpha}_{t}} \;,
\end{equ}
which is worse than the upper bound in the lemma above, and in particular is of
higher order with respect to the upper bound in
Proposition~\ref{prop:chaos-together}.
\end{remark}

\begin{proof}
We bound one by one the terms in \eqref{e:rest-terms} by applying successively
Lemma~\ref{lem:bds-lb} and Lemma~\ref{lem:bd-psi}. Let us start with 
the first term $ \Phi[\Psi[\Psi[u,u], u+X],Y]  $, we have
\begin{equs}
\left\| \mL_{1}  \left[   \Phi \left[\Psi[ \Psi[u,u], u+ X], Y \right]_{t} \right]
\right\| & \lesssim  t\| \Psi[ \Psi[u,u], u+ X] \|_{\mC^{\beta}_{t}} \|
\varrho_{0} \|_{ H^{- \alpha}} \\
& \lesssim  t^{\frac{3}{ 2}} \| \Psi[u, u]
\|_{\mC^{\beta}_{t}} \| u+X \|_{\mC^{\beta}_{t}} \| \varrho_{0}
\|_{H^{- \alpha}}\\
& \lesssim  t^{2} \| u \|_{\mC^{\beta}_{t}}^{2} \| u+X \|_{\mC^{\beta}_{t}} \| \varrho_{0}
\|_{H^{- \alpha}} \;.
\end{equs}
For the second term $ \Phi[\Psi[u,u], \Phi[u,Y]] $ we have
\begin{equ}
\left\| \Phi[\Psi[u,u], \Phi[u,Y]]_{t} 
\right\|\lesssim  t^{2} \kappa^{- \frac{1}{2}} \| u
\|_{\mC^{\beta}_{t}}^{3} \| \varrho \|_{H^{- \alpha}_{t}} \;.
\end{equ}
For the third term $  \Phi[u,\Phi[\Psi[u,u],Y]]$ we have
\begin{equs}
 \left\| \mL_{1}  \left[   \Phi[u,\Phi[\Psi[u,u],Y]]_{t} \right]
\right\| & \lesssim t^{\frac{3}{2}} \kappa^{- \frac{1}{2}}  \| u
\|_{\mC^{\beta}_{t}} \| \Psi[u, u] \|_{\mC^{\beta}_{t}} \| \varrho
\|_{H^{- \alpha}_{t}}  \lesssim  t^{2} \kappa^{- \frac{1}{2}} \| u \|_{\mC^{\beta}_{t}}^{3} \| \varrho
\|_{H^{- \alpha}_{t}} \;.
\end{equs}
Now we are left with bounding $ \overline{R} $, which consists itself of four
terms, which we can bound as follows:
\begin{equs}
\left\| \mL_{1}  \left[   \Phi[u,\Phi[u,\Phi[u, \Phi[u, \varrho]]]]_{t} \right]
\right\| & \lesssim t^{\frac{5}{2}} \kappa^{- \frac{3}{2}} \| u
\|_{\mC_{t}^{\beta}}^{4} \| \varrho \|_{H^{- \alpha}_{t}} \;, \\
\left\| \mL_{1}  \left[   \Phi[u, \Phi[u, \Phi[ \Psi[u,u],
Y]]] _{t} \right]
\right\| & \lesssim t^{\frac{5}{2}} \kappa^{- 1} \| u
\|_{\mC_{t}^{\beta}}^{4} \| \varrho \|_{H^{- \alpha}_{t}} \;, \\
\left\| \mL_{1}  \left[ \Phi[u, \Phi[\Psi[u,u], \Phi[X, Y]]]_{t} \right]
\right\| & \lesssim t^{\frac{5}{2}} \kappa^{- 1} \| u
\|_{\mC_{t}^{\beta}}^{4} \| \varrho \|_{H^{- \alpha}_{t}} \;, \\
\left\| \mL_{1}  \left[ \Phi[\Psi[u, u], \Phi[X, \Phi[X, Y]]]_{t} \right]
\right\| & \lesssim t^{\frac{5}{2}} \kappa^{- 1} \| u
\|_{\mC_{t}^{\beta}}^{4} \| \varrho \|_{H^{- \alpha}_{t}} \;.
\end{equs}
Now, the final estimate follows from \eqref{e:a-priori-scalar}, and by setting
$ t = \ts(M) $. Indeed, we observe that $ \ts (M)^{\frac{5}{2}} \kappa^{-
\frac{3}{2} } \lesssim \ts (M)^{2} \kappa^{-1} $ under the assumption that $
M \geqslant \kappa^{- \q} $ for some $ \q > 1 $.
\end{proof}
This concludes our study of the upper bounds, and we are ready to conclude the
section with the proof of the main result.

\subsection{Proof of Proposition~\ref{prop:bds-new}}
\label{sec:prf-prop}

We are now ready to prove the high-frequency instability result that was the
main objective of this section.

\begin{proof}[of Proposition~\ref{prop:bds-new}]

Let us recall from \eqref{e:decomp-rho} that we have decomposed $\varrho$ as
\begin{equ}
\varrho = A + S + R\;,
\end{equ} where the stochastic term $S$ and the remainder term $R$ are defined
in \eqref{e:stoch-terms} and \eqref{e:rest-terms} respectively, and $A$ is
defined by $A_{t} = Y_{t}+\Phi[X,Y]_{t}$. See also \eqref{e:def-X} and
\eqref{e:def-Y} for definitions of $X$ and $Y$. The strategy of the proof is to
obtain a lower bound for $A+S$ and an upper bound for $R$. Moreover, the proof is slightly different in the case of passive
scalar advection and in that of linearised SNS, because the upper and lower
bounds differ (despite the overall strategy being unchanged). We will provide a
detailed proof in the case of passive scalar advection and only point out the
main differences in the case of linearised SNS after the proof for passive
scalar advection. Note that for advection it suffices to project on a ball of
Fourier modes of radius $ 1 $, while for technical reasons in linearised SNS we
require a ball of radius $ \sqrt{2} $. In the statement of the proposition we
chose to project on the largest of the two balls (and the worst of the two lower bounds) to
avoid two separate statements.

\emph{The lower bound.} We will obtain the lower bound by the Paley--Zygmund inequality, thus it suffices to consider a lower bound for the first chaos component of $A+S$ from which the other chaos components are uncorrelated:
\begin{equ}
 \left (\EE \| \mL_{1} [A_{t} + S_{t}] \|^{2}\right)^{\frac{1}{2}} \geqslant  \left( \EE \| \mL_{1} [\Pi^{(1)}(A_{t} + S_{t})] \|^{2}\right)^{\frac{1}{2}} \geqslant \left( \EE \| \mL_{1} [\Pi^{(1)}A_{t} ] \|^{2} \right)^{\frac{1}{2}} - \left( \EE \| \mL_{1} [\Pi^{(1)}S_{t} ] \|^{2} \right)^{\frac{1}{2}}  \;,
\end{equ} where $\Pi^{(1)}$ denotes the projection on the first chaos and the second inequality follows from the triangle inequality. We first consider the lower bound on $\EE \| \mL_{1} [\Pi^{(1)}A_{t} ] \|^{2}$. Fix $\beta \geqslant 1$ and recall that $X = X^{0} + \widetilde{X}$ where $X^{0}_{t} = P_{t} u_{0}$. By Lemma~\ref{lem:LBI1-new}, for all integers $M\geqslant M^{(\beta)}(\varrho_{0})$ we have 
\begin{equ}
\EE \| \mL_{1} [\Pi^{(1)}A_{\ts(M)} ] \|^{2} = \EE \| \mL_{1}\Phi[\widetilde{X}, Y]_{\ts(M)}\|^{2} \gtrsim \kappa^{-1} M^{-6} \| \varrho_{0} \|^{2}_{H^{-\frac{\a}{2}}}\;.
\end{equ} On the other hand, for any fixed $q>1$, from Proposition~\ref{prop:chaos-together} we have that for any $1<\alpha<\gamma < \a/2$ and $M\geqslant \kappa^{-q}$ 
\begin{equ}
\EE \| \mL_{1} [\Pi^{(1)}S_{\ts(M)} ] \|^{2} \lesssim  \kappa^{-5}M^{-8}(\log(M))^{4}( \|u_{0}\| +1 )^{4}_{\mC^{\gamma}} \| \varrho_{0}\|^{2}_{H^{-\alpha}}\;.
\end{equ} Collecting these estimates, there exist $C_{1}, C_{2}>0$ such that 
\begin{equs}
\left (\EE \| \mL_{1} [A_{\ts(M)} + S_{\ts(M)}] \|^{2}\right)^{\frac{1}{2}} &\geqslant C_{1} \kappa^{-\frac{1}{2}}M^{-3} \| \varrho_{0} \|_{H^{-\frac{\a}{2}}} - C_{2} \kappa^{-\frac{5}{2}}M^{-4}(\log(M))^{2}( \|u_{0}\| +1 )^{2}_{\mC^{\gamma}} \| \varrho_{0}\|_{H^{-\alpha}}\\
&\geqslant \frac{1}{2} C_{1} \kappa^{-\frac{1}{2}}M^{-3} \| \varrho_{0} \|_{H^{-\frac{\a}{2}}}\;,
\end{equs} where the last inequality follows assuming that in addition $ M \geqslant
M^{(\beta)}(\varrho_{0}) \vee 
\kappa^{- \q} \vee (\| u_{0} \|_{\mC^{\gamma}}+1)^{\r} $ for some $ \q > 2 $, and
by choosing $ \alpha, \gamma $ sufficiently close to $ \a/2 $ and $ \r $
sufficiently large. Indeed we note that 
$ \kappa^{- \frac{1}{2}} M^{ - 3 } \gg \kappa^{- \frac{5}{2} } M^{- 4} $ as
long as $ M \gg \kappa^{- 2} $. In addition, we have used that since $
M \geqslant M^{(\beta)}(\varrho_{0}) $, we can estimate
\begin{equs}
\| \varrho_{0} \|_{H^{ - \alpha}} \lesssim \| \mL_{M} \varrho_{0} \|_{H^{-
\alpha}} \lesssim M^{\frac{\a}{2} - \alpha} \| \mL_{M} \varrho_{0} \|_{H^{-
\frac{\a}{2}}} \lesssim M^{\frac{\a}{2} - \alpha} \| \varrho_{0} \|_{H^{-
\frac{\a}{2}}} \;.
\end{equs} Finally, we can pass from a lower bound on the second moment to a lower bound
on the probability of the desired norm being sufficiently large. To this aim,
let us define $ Z = \| \mL_{1} [ A_{\ts(M)} + S_{\ts(M)}] \|^{2} $ and $ \mu = \EE[Z] $.
Then by the Paley--Zygmund inequality, for any $ \vt \in (0,1) $ we can
estimate
\begin{equ}
\PP (Z \geqslant \vt \mu) \geqslant (1- \vt )
\frac{\mu^{2}}{\EE[Z^{2}]} \;.
\end{equ}
Now, since $ \sqrt{Z} $ lives in a finite inhomogeneous Wiener--It\^o chaos, we deduce by Gaussian
hypercontractivity that $ \EE[ Z^{2}] \leqslant c^{-1} \mu^{2} $ for some $
c \in (0, 1) $ that does not depend on the law of $ Z $. Therefore, we obtain the estimate
\begin{equ}
\PP (Z \geqslant \vt \mu) \geqslant c (1- \vt ) \;.
\end{equ}
Setting for example $ \vt = 1/2 $, this is the kind of lower bound that we are
looking for. Indeed if we would have $ R = 0 $, this would imply the statement
of the proposition with $ \alpha = c/2 $, since our previous calculations have
proven that
\begin{equ}
\sqrt{\mu}  \gtrsim \kappa^{- \frac{1}{2}} M^{ -3} \|
\varrho_{0} \|_{H^{- \frac{\a}{2}}} \;, \qquad \forall M \geqslant 
M^{(\beta)}(\varrho_{0}) \vee \kappa^{- \q} \vee C(\| u_{0}
\|_{\mC^{\gamma}}+1)^{\r} \;,
\end{equ}
for some suitable $ C > 0 $.

\emph{The upper bound.} The proof of the proposition is completed
if we can show that for any $ D, p > 0 $, there exists a $ C(D, p)> 0 $ such
that
\begin{equ}
\PP \left( \| \mL_{1} R_{\ts (M)} \| \geqslant D \kappa^{- \frac{1}{2}} M^{ -3} \|
\varrho_{0} \|_{H^{- \frac{\a}{2}}} \right) \leqslant C
M^{- p} (\| u_{0} \|_{\mC^{\gamma}}^{\r}+1)^{p} \;,
\end{equ}
for some $ \r > 1 $. Now we apply Lemma~\ref{lem:rest-complete} to obtain that for any $
 \ve \in (0, 1) $ and $ \gamma \in \left( \frac{\a - \ve}{2} ,
\frac{\a}{2} \right) $, there exists a $ \widetilde{C} (\ve) >0 $ such that 
\begin{equ}
 \| \mL_{1} R_{\ts (M)} \| \leqslant \widetilde{C} (\ve)  \kappa^{-\frac{5}{2}} M^{- 4  + \ve} \|
\varrho_{0} \|_{H^{- \frac{\a}{2}}} \exp \left( C \ts(M)  \kappa^{-1} 
\| u \|_{\mC^{\gamma}_{\ts (M)}}^{2} \right) \left( \| u \|_{\mC^{\gamma}_{\ts (M)}} + \| X
\|_{\mC^{\gamma}_{\ts (M)}} \right)^{4} \;,
\end{equ} 
where we have used the spaces $ \mC^{\gamma}_{t} $ introduced in
\eqref{e:cbt}.
Here we have absorbed all logarithmic terms in the
power $ M^{\ve} $.  Therefore, we bound the previous probability through the
following:
\begin{equs}
 \PP & \left(   \widetilde{C} (\ve)  \kappa^{-\frac{5}{2}} M^{- 4 + \ve} \exp \left( C
\ts (M) \kappa^{-1}  \| u \|_{\mC^{\gamma}_{\ts (M)}}^{2}  \right) \left( \| u
\|_{\mC^{\gamma}_{\ts(M)}} + \| X
\|_{\mC^{\gamma}_{\ts(M)}} \right)^{4} \geqslant D \kappa^{- \frac{1}{2}} M^{ -3} \right) \\
& = \PP \left(  \kappa^{- 2} M^{- 1 + \ve} \exp \left( C \ts
(M)  \kappa^{-1}\| u \|_{\mC^{\gamma}_{\ts(M)}}^{2}  \right) \left( \| u \|_{\mC^{\gamma}_{\ts
(M)}} + \| X \|_{\mC^{\gamma}_{\ts (M)}} \right)^{4} \geqslant  \overline{C}  \right) \;,
\end{equs}
with $ \overline{C} = \widetilde{C} (\ve)^{-1} D $. Now fix any $\q>2$. Then by the assumption of the proposition we can find $\zeta=\zeta(\q)>0$ and $\ve=\ve(\q)>0$ small so that $\kappa^{-2}M^{-1+\ve}\leqslant M^{-4\zeta}$  and $ \ts (M) \kappa^{-1}  \leqslant M^{-2\zeta} $ for all $M\geqslant \kappa^{-\q}$. Therefore, we
conclude that the previous probability is bounded from above by 
\begin{equs}
\PP  \bigg( \exp \left( C M^{-2 \zeta}  \| u \|_{\mC^{\gamma}_{\ts(M)}}^2  \right)
& \left( M^{ - \zeta} \| u \|_{\mC^{\gamma}_{\ts
(M)}} + M^{- \zeta}\| X \|_{\mC^{\gamma}_{\ts (M)}} \right)^{4} \geqslant
\overline{C} \bigg) \\
& \leqslant \PP \left( M^{- \zeta} \| u \|_{\mC^{\gamma}_{t}} \geqslant E
\right) + \PP \left( M^{- \zeta} \| X \|_{\mC^{\gamma}_{t}} \geqslant E \right)
\;,
\end{equs}
for a suitable $ E > 0 $ that depends on the values of $ C $ and $ \overline{C}
$. Now by applying Markov's inequality, together with
Lemma~\ref{lem:mmt-unif}, we have proven that
\begin{equs}
\PP \left( \left\| \mL_{1} \varrho_{\ts (M)} \right\| \geqslant C_{1} \kappa^{-
\frac{1}{2}} M^{ -3} \| \varrho_{0} \|_{H^{- \frac{\a}{2}}} \right) & \geqslant \alpha -
C_{2} M^{- p} (\| u_{0} \|_{\mC^{\gamma}}+1)^{\r p} \;.
\end{equs} 
Since $ M \geqslant M^{(\beta)} (\varrho_{0}) $ we have $ \|
\varrho_{0} \|_{H^{- \frac{\a}{2}} } \gtrsim M^{- \frac{\a}{2}}  $, so that the proof is
complete.

\emph{Adaptation to linearised SNS.} The case $ L [u] = u \cdot \nabla + \Delta
u \cdot \nabla^{-1} $ follows through similar calculations, the only changes
occurring in the lower and upper bounds. In particular, by
Lemma~\ref{lem:LBI1-sns} and Proposition~\ref{prop:chaos-lns}, for any $ \alpha -1 < \gamma < \a/2  $ and $\q>2$ there exist constants $ C_{1}, C_{2} > 0 $ such that 
\begin{equs}
\left (\EE \| \mL_{\sqrt{2}} [A_{\ts(M)} + S_{\ts(M)}]
\|^{2}\right)^{\frac{1}{2}}&\geqslant C_{1} \kappa^{-\frac{1}{2}}M^{-3} \|
\varrho_{0} \|_{H^{-(\frac{\a}{2}+1)}} \\
& \qquad \qquad - C_{2} \kappa^{-\frac{5}{2}}M^{-4}(\log(M))^{2}( \|u_{0}\| +1 )^{2}_{\mC^{\gamma}} \| \varrho_{0}\|_{H^{-\alpha}}\\
&\geqslant \frac{1}{2} C_{1} \kappa^{-\frac{1}{2}}M^{-3} \| \varrho_{0} \|_{H^{-(\frac{\a}{2}+1)}}\;,
\end{equs}
up to choosing suitable $ \alpha, \gamma $, and for $  M \geqslant 
M^{(\beta)}(\varrho_{0}) \vee \kappa^{- \q} \vee C(\| u_{0}
\|_{\mC^{\gamma}}+1)^{\r}$ with some sufficiently
large $ C > 0 $.
As for the upper bound, we must prove that there exists an $ \r $ such that
for any $ D, p > 0 $ there exists a $ C(D, p) \geqslant 1 $ such that
\begin{equ}
\PP \left( \| \mL_{\sqrt{2} } R_{\ts (M)} \| \geqslant D \kappa^{- \frac{1}{2}} M^{ -3} \|
\varrho_{0} \|_{H^{- (\frac{\a}{2} +1)}} \right) \leqslant C
M^{- p} (\| u_{0} \|_{\mC^{\gamma}}^{\r}+1)^{p} \;.
\end{equ}
Here we use Lemma~\ref{lem:rest-1-sns} to obtain for any $ \alpha, \gamma > 2 $ such that $
\alpha -1 < \gamma < \a/2$, and some $ C > 0 $:
\begin{equ}
\| \mL_{\sqrt{2}} R_{\ts (M)} \| \lesssim \kappa^{- \frac{5}{2}  
 } M^{-4 +\ep} \exp \left( C \ts(M) \kappa^{-1}  \|
u \|_{\mC^{\gamma}_{\ts(M)}}^{2} \right) \left( \| u
\|_{\mC^{\gamma}_{\ts(M)}}^{4}+ \| X
\|_{\mC^{\gamma}_{\ts(M)}}^{4}\right) \| \varrho_{0} \|_{H^{- \alpha }} \;.
\end{equ}
This is of lower order with respect to the lower bound we obtained before,
provided $ M \geqslant
M^{(\beta)}(\varrho_{0}) \vee \kappa^{- \q} \vee  \widetilde{C}(\| u_{0}
\|_{\mC^{\gamma}}+1)^{\r}$ for some $
\widetilde{C} > 0 $ and for any $ \q > 2 $. From this point onward, the proof
follows verbatim the case of pure advection, and we avoid repeating
it.
\end{proof}

\section{Gaussian estimates for linearised SNS} \label{sec:gauss-sns}
The estimates for energy transfer for linearised SNS follow by and large the
same approach as those for the transport equation. However, at low frequencies,
there is a subtle cancellation between the transport and the stretching term.
This can be seen by a crude power counting, since both $ u \cdot \nabla \varrho $
and $ \Delta u \cdot \nabla^{-1} \varrho $ have a ``total'' of one derivative,
combining the Laplacian with the inverse gradient. Because of this
cancellation, some care has to be taken, and we will detail the changes in both
the lower bounds an upper bounds. Whenever the calculations are similar to
those in Section~\ref{sec:gaussian-estimates}, we will try to avoid repeating
them.

\subsection{Lower bound}
We start by obtaining a lower bound on the first, Gaussian, term in the
decomposition of $ \varrho $. This result stands in analogy to
Lemma~\ref{lem:LBI1-new}. The main difference is the presence of some
cancellations, which are due to the particular structure of the equation for
linearised  Navier--Stokes. These lead to a worse lower bound than the one
obtained in Lemma~\ref{lem:LBI1-new}, and therefore require separate attention.
In particular, also the upper bounds must be treated suitably to match the
lower bound: this task is carried out in the upcoming sections.

\begin{lemma}\label{lem:LBI1-sns} 
Consider the map $ \Phi $ from \eqref{e:def-phi} in the case $ L [u] = u \cdot
\nabla + \Delta u \cdot \nabla^{-1}  $, and fix any $ \beta \geqslant 1 $. Then the following estimates hold uniformly over
$ \kappa \in (0, 1] $, $\varrho_{0} \in L^{2}_{*} $ and $M \in
\NN, M \geqslant M^{(\beta)}(\varrho_{0}) $:
\begin{equ}
\EE \left\| \mL_{\sqrt{2}} \left[ \, \Phi [ \widetilde{X}, Y]_{\ts (M)} \right]
\right\|^{2} \gtrsim \kappa^{- 1} M^{-6} \|
\varrho_{0} \|_{H^{-(\frac{\a}{2}+1)}_{M}}^{2} \;.
\end{equ}
\end{lemma}
The choice of projecting on frequencies $ k $ with $ | k | \leqslant
\sqrt{2} $ is due to a geometric non-degeneracy condition, see
Lemma~\ref{lem:geometric}. We must guarantee that the set of frequencies on
which we project is such that we can find $ k_{1}, k_{2} $ that are neither
perpendicular nor parallel. For example $ k_{1} = (1, 0), k_{2} = (1, 1) $
satisfy the condition (note here that $ | k_{2} | = \sqrt{2} $).

\begin{proof}
Consider $ \ell \in \{ \ell \in \ZZ^{2}  \; \colon \; | \ell | \leqslant \sqrt{2}
 \} $ and let us write $
\mI_{t} (\ell) = \mF \bigl[ \, \Phi [ \widetilde{X}, Y]_{t} \bigr] (\ell) $, and
observe that $ \mI_{t} (\ell) $ is Gaussian and given by
\begin{equ}
\mI_{t} (\ell) = \int_{0}^{t} e^{- \kappa | \ell |^{2} (t -s)}
\sum_{k \in \ZZ^{2}_{*} } \mF [ P^{\kappa}_{s} \varrho_{0} ](\ell -
k) \; \ell \cdot   \mF \widetilde{X}_{s}(k) \left( 1 - \frac{| k |^{2}}{|
\ell - k |^{2}} \right) \ud s \;.
\end{equ}
To simplify the upcoming calculations, let us define
\begin{equ}
\mf{c} (\ell, k) = \left(1- \frac{| k |^{2}}{| \ell - k |^{2}}\right) \langle k^{\perp}, \ell \rangle\;, \qquad \forall k, \ell \in
\ZZ^{d} \;,
\end{equ}
with the convention that $ \mf{c} (k, k) = 0 $.
Recall that $ \mF \widetilde{X}$  is given by
\begin{equ}
\mF \widetilde{X}_{s}(k)= \int_{0}^{s} e^{-\nu | k |^{2} (s-r)} \frac{k^{\perp}}{| k |^{1 + \frac{\a}{2}
}}  \ud \zeta^{k}_{r}  \ud s \;,
\end{equ} 
and note that we have $ k \cdot \mF \widetilde{X}_{s}(k)=0$ due to
incompressibility of $ \widetilde{X} $. So the desired term can be rewritten as
\begin{equ}
\mI_{t} (\ell) = - \iota \int_{0}^{t} e^{- \kappa | \ell |^{2}  (t -s)}
\sum_{k \in \ZZ^{2}_{*}} \mF [ P^{\kappa}_{s} \varrho_{0} ] (\ell -
k) \, \ell \cdot \mF \widetilde{X}_{s}(k) \, \mf{c} (\ell, k) \;.
\end{equ}
Now we represent $ \mI (\ell) $ as a stochastic integral of the form
\begin{equ}
\mI_{t} (\ell) = \int_{E_{t}} f_{t}^{\ell} (s, k) \ud \zeta(s,k) \;,
\end{equ}
where the integrand $ f^{\ell}_{t} $ is given by
\begin{equ}
f_{t}^{\ell}(s, k) = \mf{c} (\ell, k) \frac{ \hat{\varrho}_{0}(\ell - k) }{| k |^{\frac{\a}{2}+1}}\int_{s}^{t} e^{- \kappa | \ell |^{2} (t-r)}
e^{- \kappa |\ell - k|^{2} r} e^{- | k |^{2} (r-s)}\ud r \;.
\end{equ} 
Now by \eqref{e:ito} we find that
\begin{equ}
\mathrm{Var} \left( \mI_{t} (\ell) \right)  = \sum_{k} |\mf{c} (\ell, k)|^{2} \frac{|  \hat{\varrho}_{0}(\ell - k) |^{2}}{| k
|^{\a+2}} \int_{0}^{t}  \left\vert \int_{s}^{t}
e^{- \kappa | \ell |^{2} (t-r)}
e^{- \kappa |\ell- k |^{2} r} e^{- | k |^{2} (r -s)} \ud r 
\right\vert^{2} \ud s \;.
\end{equ}
Following the same calculations as in the proof of Lemma~\ref{lem:LBI1-new}. If we consider $ | k | \leqslant M +
\sqrt{2} $ and $ t = \ts (M) $ we can lower bound the time integral
\begin{equ}
   \int_{0}^{t}  \left\vert \int_{s}^{t}
   e^{- \kappa | \ell |^{2} (t-r)}
   e^{- \kappa |\ell- k |^{2} r} e^{- | k |^{2} (r -s)} \ud r 
   \right\vert^{2} \ud s \gtrsim \kappa^{-1} M^{-6} \;.
\end{equ}
Additionally, by defining 
\begin{equ}
\widetilde{c} (\ell, k) =  (| \ell |^{2}- 2 \langle k, \ell \rangle ) \langle
k^{\perp}, \ell \rangle  \;,
\end{equ}
we can compute
\begin{equ}
| \mf{c} (\ell, k) |^{2}= \left( \frac{|\ell - k|^{2}- | k |^{2}}{|
\ell - k|^{2}} \right)^{2}\langle
k^{\perp}, \ell \rangle^2 =  \frac{( | \ell |^{2}- 2 \langle k, \ell \rangle)^{2}}{|\ell - k |^{4}}\gtrsim \frac{| \widetilde{c} (\ell, k) |^{2}}{|\ell - k |^{4}} \;.
\end{equ} 
Hence, shifting indices in the sum, we have obtained the lower bound
\begin{equ}
\mathrm{Var} \left( \mI_{t_{\star}^{\kappa} (M)} (\ell) \right)  \gtrsim \kappa^{-1} M^{-6} \sum_{ |k| \leqslant M   } \frac{|  \hat{\varrho}_{0}(k) |^{2}}{|k
|^{\a+6}}  | \widetilde{c}(\ell, \ell- k) |^{2}\;.
\end{equ}
At this point, the issue is that the second degree (in $ k $) polynomial  $
\widetilde{c}(\ell, \ell-k) $ can vanish at some points $ k $, for fixed $ \ell
$: for instance for $ k $ proportional to $ \ell $. However,
it cannot vanish for all $ \ell $'s at once, namely we have
\begin{equ}
\sum_{| \ell | \leqslant \sqrt{2}} |\widetilde{\mf{c}} (\ell, \ell-k)|^{2} \gtrsim |
k |^{4} \;.
\end{equ}
This follows from Lemma~\ref{lem:geometric} below, noting that $
\widetilde{\mf{c}}(\ell, \ell - k) = \widetilde{c}(\ell, k) $. Using this
bound we obtain that
\begin{equ}
\EE \bigl\| \mL_{\sqrt{2}} \bigl[ \, \Phi [ \widetilde{X}, Y]_{\ts (M)} \bigr]
\bigr\|^{2} \gtrsim \kappa^{- 1} M^{-6}  \| \mL_{M}
\varrho_{0} \|^{2}_{H^{- (\a/2 +1)}} \gtrsim \kappa^{- 1} M^{-6}  \|
\varrho_{0} \|^{2}_{H^{- (\a/2 +1)}} \;,
\end{equ}
which is the desired lower bound. Here, in the last lower bound we have used
that $ M \geqslant M^{(\beta)}(\varrho_{0}) $. The proof is complete.
\end{proof}
We conclude the subsection with a geometric non-degeneracy lemma, which is
used to derive the lower bound above.

\begin{lemma}\label{lem:geometric}
For any $ \ell \in \RR^{2} $ define $ Q_{\ell} (k) = || \ell |^{2} - 2 \langle k, \ell
\rangle |^{2} | \langle k^{\perp}, \ell \rangle |^{2}$.
Then there exists a $ c > 0 $ such that for all $ k \in \ZZ^{2} $
\begin{equ}
\sum_{| \ell | \leqslant \sqrt{2}} Q_{\ell}(k) \geqslant c | k |^{4} \;.
\end{equ}
\end{lemma}
\begin{proof}
We start by observing that there exists a $ c > 0 $ such that for every $ k \in
\ZZ^{2} $ there exists a $ \ell = \ell_{k} \in \ZZ^{2}$ (depending on $ k $) with $ |
\ell_{k} | \leqslant \sqrt{2} $  such that
\begin{equ}
| \langle \ell, k \rangle | | \langle \ell, k^{\perp} \rangle | > c | k
|^{2} \;.
\end{equ}
This follows because the only zeros on the sphere of the map $ | \langle \ell, k
\rangle | | \langle \ell, k^{\perp} \rangle | $ are  $ k \in \{ \pm \ell / |
\ell |, \pm \ell^{\perp} / | \ell |\}$, and these sets are disjoint for $ \ell =
(1, 0) $ and $ \ell = (1, 1) $. In particular, it follows that there exists a $
C, c > 0 $ such that for $ | k | > C $ the desired lower bound $
\sum_{| \ell | \leqslant \sqrt{2}} Q_{\ell}(k) \geqslant c | k |^{4} $ holds. Now
to conclude it suffices to show that for $ k \neq 0 $ the sum can not vanish.
Indeed, for fixed $ k $, choose $ \ell = \ell_{k} $ as above such that $ | \langle \ell, k
\rangle | | \langle \ell, k^{\perp} \rangle | > c | k |^{2} $. Then the quantity
$ Q_{\ell}(k) $ can only vanish if $ | \ell |^{2} -2 \langle k, \ell \rangle = 0 $,
which amounts to $ | k | = | k - \ell | $. But then it must holds that $
Q_{\ell^{\prime}} (k) > 0 $, for $ \ell^{\prime} = - \ell $, since otherwise we would
have $ | k - \ell | = | k + \ell | $, which can only be the case if $ \langle k, \ell
\rangle = 0 $, contradicting the fact that $ \ell $ is chosen so that it is not
orthogonal to $ k $.
\end{proof}

\subsection{Upper bounds}

The strategy for the upper bounds on the rest terms is the same as in
Section~\ref{sec:upper-bounds}. However, we must take into account the
cancellations that are due to the linearised Navier--Stokes nonlinearity.
In particular, we still start from the representation
\begin{equs}
\varrho = Y + \Phi [X, Y ] + \Phi[\Psi[u,u],Y] + \Phi[u,\Phi[u,\varrho]] \;,
\end{equs}
and our objective is to show that the last two terms are of lower order through
suitable upper bounds. To identify the issues and set up the analysis, let us
start with the first term.
Here we would like to obtain an upper bound on
\begin{equ}
    \| \mL_{\sqrt{2}
} \left[ \Phi [u, \Phi[u, \varrho]]\right] \| \;.
\end{equ}
Now, we can rewrite $ \Phi[\varphi, \psi] $ as
\begin{equ}[e:Phi-Fourier]
\mF \left[ \Phi [\varphi, \psi]_{t} \right] (\ell)  = \int_{0}^{t} e^{- \kappa
| \ell|^{2} (t -s)} \sum_{k \in \ZZ^{2}_{*}  } \hat{\psi}_{s}  (\ell -
k) \, \mf{c} (\ell, k) \,  \ell\cdot  \hat{ \phi}_{s}(k)  \ud s \;,
\end{equ}
where $ \mf{c}(\ell, k) = 1- | k |^{2} / | \ell - k |^{2} $ and we have used that
$ \varphi $ is divergence free, so that $ (\ell - k) \cdot \hat{\varphi}
(k) = \ell \cdot \hat{\varphi} (k) $. Note that for $ |
\ell | \lesssim 1 $ we have $ |\mf{c} (\ell, k)| \lesssim | k |^{-1} $.
Therefore we can bound via \eqref{e:bd-paraproduct}:
\begin{equ}
\| \mL_{\sqrt{ 2}} \Phi [u, \Phi[u, \varrho ]]_{t} \|  \lesssim 
\int_{0}^{t} \| (- \Delta)^{- \frac{1}{2}} \Phi[u, \varrho]_{s}
\|_{H^{- \alpha}} \| u_{s} \|_{\mC^{\beta}} \ud s \lesssim  t M^{- \alpha -1}
\| u\|_{\mC^{\beta}_{t}} \| \Phi [u, \varrho] \|_{H^{-(\alpha+1)}_{t}} \;,
\end{equ}
for any $ 0 < \alpha < \beta < \a/2$. Now the main difference with respect to
the advection case is that we must control $ \Phi[u, \varrho] $ in terms of the
$ H^{- (\alpha +1)} $ norm of $ \varrho_{0} $, for any $ \alpha < \a/2$. This
is one degree less regularity than in the estimates in the previous section. 

In the present case of linearised SNS, we can instead improve this, by means of the cancellations that
appear in the linearisation, and which are captured in
Lemma~\ref{lem:lin-sns-product}. Indeed, we can rewrite $ \Phi $ as:
\begin{equ}
\Phi[u, \varrho]_{s} = \int_{0}^{s} P^{\kappa}_{s-r} \div (R [u_{r}] \varrho_{r}) \ud r \;,
\end{equ}
with
\begin{equ}[e:R-nF]
R [u] \varrho = u \varrho + (\Delta u) (- \Delta)^{-1} \varrho \;.
\end{equ}
Then, Lemma~\ref{lem:lin-sns-product} guarantees that $ \| R[u] \varrho
\|_{H^{- (\alpha +1)}} \lesssim \| u \|_{\mC^{\beta}} \| u
\|_{H^{- (\alpha+1)} } $ for any $ 2< \alpha < \beta$. Note that by contrast, the
product $ u \varrho $ by itself can be estimated at best via
\eqref{e:bd-paraproduct} through $ \| u \varrho \|_{H^{- \alpha}} \lesssim \|
u \|_{\mC^{\beta}} \| \varrho \|_{H^{- \alpha}}  $.
As a consequence of Lemma~\ref{lem:lin-sns-product}, via Schauder estimates, we can bound
\begin{equs}[e:bd-phi-sns]
\| \Phi [u, \varrho]_{t} \|_{H^{- (\alpha +1)}}  & \lesssim \sqrt{t/ \kappa}  \|
R[u] \varrho \|_{H^{- (\alpha +1)}_{t}} \lesssim
\sqrt{t/ \kappa}\| u \|_{\mC^{\beta}_{t}} \| \varrho
\|_{H^{- (\alpha+1)}_{t}} \;, \\
\| \mL_{\sqrt{2}} \Phi [u, \varrho]_{t} \|_{H^{- (\alpha +1)}}  & \lesssim \|
u \|_{\mC^{\beta}_{t}} \| \varrho \|_{H^{- (\alpha+1)}_{t}} \;,
\end{equs}
which stands in analogy to Lemma~\ref{lem:bds-lb} in the case of pure
advection.
In particular, we observe that this estimate together with a Gronwall-type
argument guarantees that for any $ 2 < \alpha < \beta < \a/2 $
\begin{equ}[e:a-prior-sns]
\| \varrho \|_{H^{- (\alpha +1)}_{t}} \lesssim
\| \varrho_{0} \|_{H^{- (\alpha +1)}_{t}} \exp \left( C t \kappa^{-1}  \| u
\|_{\mC^{\beta}_{t}}^{2} \right) \;.
\end{equ}
Hence, we have overall obtained that
\begin{equ}
\| \mL_{\sqrt{2}} \left[ \Phi[u, \Phi[u,\varrho]] \right]_{t} \|
\lesssim_{u, \kappa} t^{\frac{3}{2}} \| \varrho_{0}
\|_{H^{- (\alpha +1)}} \;,
\end{equ}
for any $ \alpha < \a/2 $. This is exactly analogous to \eqref{e:need-t-ref},
and in particular it is just barely not matching the lower bound of
Lemma~\ref{lem:LBI1-sns}, when $ t = \ts (M) $. Therefore, as in the previous
section, we must proceed with a finer decomposition.

\subsection{Higher order Gaussian and remainder estimates}
The next results are the analogues of the ones established in
Section~\ref{sec:hog} in the context of linearised SNS. The main difference is a modification of
Lemma~\ref{lem:Gaussian-n}, which allows to capture the cancellations that
appear in linearised SNS, in the spirit of Lemma~\ref{lem:lin-sns-product}.

\begin{lemma}\label{lem:Gaussian-lns}
   Consider the map $ \Phi $ as in \eqref{e:def-phi} with $ L[u] =u \cdot
\nabla + \Delta u \cdot \nabla^{-1}  $ as in \eqref{e:form-LNS}. For any $ t > 0 $ and $ 1< \alpha \leqslant  \a/2 $, we estimate uniformly over $\kappa\in(0,1]$ and $\psi$:
   \begin{equ}
      \E\|\mL_{\sqrt{2}}\Phi[\widetilde{X},\psi]_{t}\|^2\lesssim
t^3 \|\psi\|_{H^{- (\frac{\a}{2}+1) }_{t}}^2\;, \qquad
\E\|\mL_{\sqrt{2}}\Phi[\widetilde{X},\psi]_{t}\|^2_{H^{- (\alpha+1)}} \lesssim
t^2 \kappa^{-1}  \|\psi\|_{H^{- (\alpha+1) }_{t}}^2 \;.
   \end{equ}
\end{lemma}
\begin{proof} 
We proceed as in the proof of Lemma~\ref{lem:Gaussian-n}, to find that
\[
   \E\|\mL_{\sqrt{2}}\Phi[\widetilde{X},\psi]_{t}\|^2 = \sum_{|\ell|\leq 1}\sum_{k\in\Z^2_*}\int_0^t|\mF \Phi[\mathcal{D}_{s,k}\widetilde{X},\psi]_t(\ell)|^2\ds = \sum_{k\in \Z^2_*}\int_0^t\|\mL_{\sqrt{2}} \Phi[\mathcal{D}_{s,k}\widetilde{X},\psi]_t\|^2\ds\;.
\]
Now, in the present setting
\begin{equ}
\Phi[\mathcal{D}_{s,k}\widetilde{X},\psi]_t = \int_{s}^{t} | k |^{-
\frac{\a}{2}} e^{- | k |^{2} (r -s)}\div (P_{t -r}^{\kappa} (R[\sigma_{k}] \psi_{r}) )\ud r
\;,
\end{equ}
Next observe that $ R [u] $, as defined in \eqref{e:R-nF}, satisfies:
\begin{equ}
R [ \sigma_{k}] \psi = \sigma_{k} (1 - | k |^{2} (- \Delta)^{-1} )\psi \;.
\end{equ}
Now we bound
\begin{equ}
\| \mL_{\sqrt{2}} \Phi[\mathcal{D}_{s,k}\widetilde{X},\psi]_t\| \leqslant | k
|^{- \frac{\a}{2}} \int_{s}^{t} \| \mL_{\sqrt{2}}
R[\sigma_{k}] \psi_{r} \|  \ud r\;,
\end{equ}
so that via Jensen's inequality
\begin{equs}
\E\|\mL_{\sqrt{2}}\Phi[\widetilde{X},\psi]_{t}\|^2 & \leqslant t \int_{0}^{t}
\int_{s}^{t} \sum_{k \in \ZZ^{2}_{*}} \frac{1}{| k |^{\a}}  \| \mL_{\sqrt{2}}
R[\sigma_{k}] \psi_{r} \|^{2} \ud r \ud s \\
& \lesssim t \int_{0}^{t}
\int_{s}^{t} \sum_{k \in \ZZ^{2}_{*}} \frac{1}{| k |^{\a}}  \| 
R[\sigma_{k}] \psi_{r} \|^{2}_{H^{- (\a/2 +1)}} \ud r \ud s \;.
\end{equs}
Now we use Lemma~\ref{lem:lns-norm} to conclude that
\begin{equ}
\E\|\mL_{\sqrt{2}}\Phi[\widetilde{X},\psi]_{t}\|^2  \lesssim t^{3} \| \psi
\|_{H^{- (\a/2 +1)}_{t}}^{2} \;,
\end{equ}
which is the desired estimate.

The second estimates follows in exactly the same way, the only difference being
that we estimate
\begin{equ}
\| \Phi[ \mD_{s, k} \widetilde{X}, \psi]_{t} \|_{H^{- (\alpha+1)}} \lesssim
\kappa^{- \frac{1}{2}} \int_{s}^{t} | k |^{- \frac{\a}{2}} (t -r)^{-
\frac{1}{2}}  \| R[\sigma_{k}] \psi_{r}
\|_{H^{- (\alpha+1)}}   \ud r\;.
\end{equ}
From here we then bound following the same steps as before:
\begin{equs}
 \E\|\Phi[\widetilde{X},\psi]_{t}\|^2_{H^{- (\alpha+1)}}  & \lesssim
\kappa^{- 1}  \int_{0}^{t} \sum_{k \in \ZZ^{2}_{*}}
\frac{1}{| k |^{\a}}  \left( \int_{s}^{t}  (t -s)^{- \frac{1}{2}} \|
 R [\sigma_{k}] \psi_{r} \|_{H^{- (\alpha+1)}} \ud r \right)^{2} \ud s \\
& \lesssim \kappa^{-1} t^{\frac{1}{ 2}} \int_{0}^{t}   \int_{s}^{t}  (t -s)^{- \frac{1}{2}} \sum_{k \ in \ZZ^{2}_{*}}
\frac{1}{| k |^{\a}}\|
 R[ \sigma_{k}] \psi_{r} \|_{H^{- (\alpha+1)}}^{2} \ud r \ud s \\
& \lesssim
\kappa^{-1} t^{2} \| \psi \|_{H^{- (\alpha+1)}_{t}}^{2} \;,
\end{equs}
where we have once more applied Lemma~\ref{lem:lns-norm}.
This completes the proof.
\end{proof}
As a consequence of this result, together with Lemma~\ref{lem:lin-sns-product},
we obtain an analogue of Proposition~\ref{prop:chaos-together}. Note that we are using the same decomposition as in \eqref{e:decomp-rho}: $\varrho = Y + \Phi[X,Y] + S + R$.

\begin{proposition}\label{prop:chaos-lns}
Fix any $ \q >1 $. Consider the map $ \Phi $ as in \eqref{e:def-phi} with $ L [u] = u \cdot \nabla
+ \Delta u \cdot \nabla^{-1} $.  For any $ 2 <\alpha < \beta < \a/2 $, we can estimate uniformly
over $ \kappa \in (0, 1] $ and $ M \geqslant \kappa^{-\q} $:
\begin{equ}
\EE \bigl\| \mL_{\sqrt{2}}  \bigl[  \Pi^{(1)}S_{\ts (M)} \bigr] \bigr\|^{2}\lesssim
(\ts (M))^{4} \kappa^{-1} (\| u_{0}
\|+1)^{4}_{\mC^{\beta}} \| \varrho_{0} \|_{H^{- (\alpha+1)}}^{2} \;.
\end{equ}
\end{proposition}
\begin{proof}
The proof of this result follows verbatim the proof of
Proposition~\ref{prop:chaos-together}, by using Lemma~\ref{lem:Gaussian-lns}
instead of Lemma~\ref{lem:Gaussian-n} and \eqref{e:bd-phi-sns} instead of Lemma~\ref{lem:bds-lb}.
The only difference lies in the treatment of the term 
$ \Phi \bigl[ \Psi [X^{0}, \widetilde{X}], Y
\bigr] $. Here as in the proof of Lemma~\ref{lem:Frst-Chaos-2n}, we find 
\begin{equ}
\EE \bigl\| \mL_{\sqrt{2}}  \bigl[  \Phi \bigl[ \Psi [X^{0}, \widetilde{X}], Y
\bigr]_{t} \bigr] \bigr\|^{2} = \sum_{k\in
\Z^2_*}\int_0^t\|\mL_{\sqrt{2}} \Phi \bigl[ \Psi [X^{0}, \mD_{s, k} \widetilde{X} ], Y
\bigr]_{t}\|^2\ds \;.
\end{equ}
Then, following the same type of argument as in the previous upper bounds and using
Cauchy--Schwarz, we obtain the following bound
\[
\E\|\mL_{\sqrt{2}}\Phi[\Psi[X^0,\widetilde{X}],Y]_t\|^2 \leq \sum_{k \in
\ZZ^{2}_{*}} \int_0^t\left(\int_s^t\int_s^{r}g(k, s,r,\tau)\dee \tau\dee r\right)^2\dee s\,,
\]
where $g$ is given by
\[
g(k, s,r,\tau) = |k|^{- \frac{\a}{2} }\|\mL_{\sqrt{2}}[R[\Div
P_{r-\tau}\mathbf{P}(X^0_\tau\otimes_s\sigma_k)]Y_{r}]\|\,,
\]
with $ \mathbf{P} $ the Leray projection.
Now let us write $U_{r,\tau,k} = \Div P_{r-\tau}\mathbf{P}(X^0_\tau\otimes_s\sigma_k)$. Then we have
\begin{equs}
   \|\mL_{\sqrt{2}}[R[U_{r,\tau,k}]Y_{r}]\|^2 &= \sum_{|\ell|\leq \sqrt{2}}\biggl(\sum_{j+m = \ell}\left(1-\frac{|j|^2}{|m|^2}\right)\widehat{U}_{r,\tau,k}(j)\widehat{Y}_r(m)\biggr)^2\\
   &= \sum_{|\ell|\leq \sqrt{2}}\biggl(\sum_{j+m = \ell}\biggl(1-\frac{|j+k|^2}{|m-k|^2}\biggr)\widehat{U}_{r,\tau,k}(j+k)\widehat{Y}_r(m-k)\biggr)^2\\
   &\leqc \sum_{|\ell|\leq \sqrt{2}}\biggl(\sum_{j}|\ell- j- k|^{-1}||\widehat{U}_{r,\tau,k}(j+k)||\widehat{Y}_r(\ell - j-k)|\biggr)^2\\
   &\leqc (r-\tau)^{-1}\sum_{|\ell|\leq \sqrt{2}}\biggl(\sum_{j}|\ell- j- k|^{-1}|\widehat{X}^0_\tau(j)||\widehat{Y}_r(\ell - j-k)|\biggr)^2\\
   &\leqc (r-\tau)^{-\frac{1}{2}}
\|X^0_\tau\|^2_{H^{\alpha}}\|e_k(-\Delta)^{-\frac{1}{2}}Y_{r}\|_{H^{-\alpha}}^2
\;,
\end{equs}
where we used that
\begin{equs}
   \left(1-\frac{|j+k|^2}{|\ell-j-k|^2}\right) &= \frac{|\ell|^2 - 2\langle\ell,j+k\rangle}{|\ell-j-k|^2}\\
   &\leq \frac{|\ell|^2}{|\ell-j-k|^2} + \frac{|\ell||j+k|}{|\ell - j-k|^2} \leqc |\ell-j-k|^{-1},
\end{equs}
since $|\ell - j-k|\approx |j+k|$. To get from the third to the fourth line we also used that
\begin{equs}
|\widehat{U}_{r,\tau,k}(j+k)| &= |\mF[\Div P_{r-\tau}\mathbf{P}(X^0_\tau\otimes_s\sigma_k)](j+k)|\\
&\leqc (r- \tau)^{-\frac{1}{2}}|\mF(X^0_\tau\otimes_s \sigma_k)|(j+k) \leqc (r-\tau)^{-\frac{1}{2}}|\widehat{X}^0_\tau|(j).
\end{equs}
It follows by an application of Lemma \ref{lem:equivalencen} 
\begin{equs}
\sum_{k \in \ZZ^{2}_{*}} |k|^{-\a} \|\mL_{\sqrt{2}}[R[U_{r,\tau,k}]Y_r]\|^2 &\lesssim (r-\tau)^{-1}\|u_0\|_{H^{\alpha}}^2\sum_{k}|k|^{-\a}\|e_k(-\Delta)^{-\frac{1}{2}}Y_r\|_{H^{-\alpha}}^2\\
&\leqc (r-\tau)^{-1}\|u_0\|_{H^{\alpha}}^2 \|(-\Delta)^{-\frac{1}{2}}Y_r\|_{H^{-\alpha}}^2\\
&\leqc (r-\tau)^{-1}\|u_0\|_{H^{\alpha}}^2 \|Y_r\|_{H^{-(\alpha+1)}} \;.
\end{equs}
Putting everything together and applying Jensen with respect to the measure $(r-\tau)^{-1/2}\dee \tau\dee r$, we obtain
\begin{equs}
\|\mL_{\sqrt{2}}\Phi[\Psi[X^0,\widetilde{X}],Y]_t\|^2 &\lesssim t^{\frac{3}{2}} \int_0^t\int_{s}^{t}\int_s^{r}(r-\tau)^{1/2}\biggl(\sum_{k \in \ZZ^{2}_{*}} |k|^{-\a} \|\mL_{\sqrt{2}}[R[U_{r,\tau,k}]Y_r]\|^2\biggr)\dee \tau\dee r\\
&\lesssim t^{\frac{3}{2}}\biggl( \int_0^t\int_{s}^{t}\int_s^{r}(r-\tau)^{-1/2}\dee \tau\dee r\biggr)  \|u_0\|_{H^{\alpha}}^2\|\varrho_0\|_{H^{-(\alpha+1)}}^2\\
&\lesssim t^4 \|u_0\|_{H^{\alpha}}^2\|\varrho_0\|_{H^{-(\alpha+1)}}^2 \;.
\end{equs}
This completes the proof of the lemma.
\end{proof}
With the previous result, we have completed our stochastic estimates, and we can move on to bound the
remainder terms through pathwise estimates. We find the following bound.

\begin{lemma}\label{lem:rest-1-sns}
Consider the map $ \Phi $ as in \eqref{e:def-phi} with $ L[u] = u \cdot \nabla
+ \Delta u \cdot \nabla^{-1}  $, and let $ R $ be defined as in
\eqref{e:rest-terms}.
For any $ 2< \alpha < \beta< \a/2 $, there exists a $ C > 0 $ such
that we can estimate uniformly
over $ \kappa \in (0, 1] $:
\begin{equ}
\left\| \mL_{\sqrt{2} }  \left[ R_{\ts (M)}
\right] \right\| \lesssim  (\ts (M))^{2}  \kappa^{- \frac{1}{2} }\| \varrho_{0}
\|_{H^{- (\alpha +1)}} \exp \bigl( C t \kappa^{-1}  \|
u \|_{\mC^{\beta}_{t}}^{2} \bigr)  ( \| X \|_{\mC^{\beta}_{t}}+ \| u
\|_{\mC^{\beta}_{t}} )^{4} \;.
\end{equ} 
\end{lemma}
The proof of this result is now identical to that of
Lemma~\ref{lem:rest-complete}, so we skip it. The only difference lies in the
application of \eqref{e:bd-phi-sns} instead of Lemma~\ref{lem:bds-lb}.
This concludes our analysis of energy transfer for linearised SNS. Before we
proceed, let us prove a norm equivalence that is fundamental to obtain
Lemma~\ref{lem:Gaussian-lns} above.

\begin{lemma}\label{lem:lns-norm}
The following equivalence of norms holds for any $ \alpha > 1 $ and $ R
$ defined in \eqref{e:R-nF}:
\begin{equ}
\sum_{k} \frac{1}{| k |^{2 \alpha} } \| R[\sigma_{k}] \psi \|_{H^{-
(\alpha +1)}}^{2} \simeq \| \psi \|_{H^{- (\alpha +1)}}^{2} \;.
\end{equ}

\end{lemma}
\begin{proof}
Since the upper bound is trivial, we focus on proving the lower bound. Let us
rewrite the left-hand side as
\begin{equs}
\sum_{k} \frac{1}{| k |^{2 \alpha}}  \| R[\sigma_{k}] \psi \|_{H^{-
\alpha-1}}^{2} & = \sum_{k, l} \frac{1}{| k |^{2 \alpha }} \frac{1}{| l
|^{2 \alpha +2}} |(1 - | k |^{2} / | l-k |^{2}) \hat{\psi} (l-k) |^{2} \\
& = \sum_{k, l} \frac{1}{| k |^{2 \alpha}} \frac{| k |^{4}}{| l
|^{2 \alpha+2}} |(1/ | k |^{2} - 1 / | l-k |^{2})\hat{\psi} (l-k) |^{2}
\;.
\end{equs}
Now we analyse the sum over the three sets $ | k | \gg | l | $, $ | l
| \gg | k | $ and $ | k | \sim | l | $, since it exhibits different behaviours.

When $ | k | \geqslant 2| k-l | $ we bound the sum by
\begin{equs}
\sum_{| k | \geqslant 2 | k - l |} \frac{1}{| k |^{2 \alpha }} \frac{| k |^{4}}{| l
|^{2 \alpha +2} | l- k |^{4}} | \hat{\psi} (l-k) |^{2} \;.
\end{equs}
Then since $ | k | \geqslant 2 | k - l | $ we have $ | l | \sim | k | $, so the
sum is additionally bounded by
\begin{equs}
\sum_{| k | \geqslant 2 | k - l |} \frac{1}{| k |^{2 \alpha }} \frac{1}{| l
|^{2 \alpha -2} | l- k |^{4}} | \hat{\psi} (l-k) |^{2} & \lesssim 
\sum_{| k | \geqslant 2 | k - l |} \frac{1}{| k |^{2 \alpha }} \frac{1}{| l-k
|^{2 \alpha +2} } | \hat{\psi} (l-k) |^{2} \lesssim \| \psi \|_{H^{-
(\alpha +1)}}^{2} \;,
\end{equs}
where we have used that $ \alpha > 1 $ and we are in dimension $ d=2 $.

Now we pass to the set $ 2 | k | \leqslant | k-l |$. Here we estimate the sum by
\begin{equs}
\sum_{ | k | \leqslant 2 | k-l |} \frac{1}{| k |^{2 \alpha }} \frac{1}{| l
|^{2 \alpha +2}} | \hat{\psi} (l-k) |^{2} \lesssim \sum_{ | k | \leqslant 2 | k-l |} \frac{1}{| k
|^{2 \alpha }} \frac{1}{| l-k |^{2 \alpha +2}} | \hat{\psi} (l-k) |^{2}
\lesssim \| \psi \|_{H^{- (\alpha+1)}}^{2}\;,
\end{equs}
which is again the desired estimate, and where we have used that $ | l- k
| \sim | l | $.

Then we have the final set $ | k | \sim | k-l | $, which is shorthand for $
2 | k - l | > | k |> (1/2) | k -l| $. Here we have
\begin{equs}
\sum_{| k | \sim | k-l | } \frac{1}{| k |^{2 \alpha}} \frac{| k |^{4}}{| l
|^{2 \alpha +2}} |(1/ | k |^{2} - 1 / | l-k |^{2}) \hat{\psi} (l-k) |^{2} & \lesssim 
\sum_{| k | \sim | k-l | } \frac{1}{| k |^{2 \alpha }} \frac{| k |^{4}}{| l
|^{2 \alpha +2}} | \l |^{2} | k |^{-6}|\hat{\psi} (l-k) |^{2} \\
& \simeq \sum_{| k | \sim | k-l | } \frac{1}{| k-l |^{2 \alpha +2}} \frac{1}{| l
|^{2 \alpha }} | \hat{\psi} (l-k) |^{2} \\
& \lesssim \| \psi \|^{2}_{H^{- (\alpha +1)}}\;,
\end{equs}
where we have used the same estimate as in the proof of
Lemma~\ref{lem:lin-sns-product}, in \eqref{e:bd-reso}. This completes the proof
of the estimate.
\end{proof}

\section{High-frequency regularity}\label{sec:regularity}

In this section we collect several results concerning the high-frequency regularity of
$ \pi_{t} $. The rule of thumb is that  $ \pi_{t} $ follows the rules of parabolic
regularity theory (via Schauder estimates) above the spectral median. As a first
step, we provide control of high- versus low-frequency energy. 
\begin{lemma}\label{lem:hl} 
In the setting of Theorem~\ref{thm:main} there exists a $ C > 0 $ such that
uniformly over $ M \in \NN $ and $ \kappa \in (0, 1] $, the following
differential inequalities hold,
for all $ t \leqslant \inf \{ s \geqslant 0  \; \colon \; \| \mL_{M}
\varrho_{s} \| = 0 \} $:
\begin{itemize}
\item If $ \varrho $ is the solution to \eqref{e:main} with  $L[u]  = u\cdot
\nabla  $, then
\begin{equ}
\frac{\ud}{\ud t} \frac{\| \mH_{M} \varrho_{t} \|}{\| \mL_{M} \varrho_{t} \|}
\leqslant C M \| u_{t} \|_{\infty} \left(  1 +  \frac{\| \mH_{M} \varrho
\|^{2}}{\| \mL_{M} \varrho \|^{2}} \right) \;.
\end{equ}
\item If $ \varrho $ is the solution to \eqref{e:main} with $L[u]  = u\cdot
\nabla  + \Delta u\cdot \nabla^{-1}$, then
\begin{equ}
\frac{\ud}{\ud t} \frac{\| \mH_{M} \varrho_{t} \|}{\| \mL_{M} \varrho_{t} \|}
\leqslant C M ( \| u_t \|_{\infty} + \|\Delta  u_{t} \|_{\infty} ) \left(  1 +  \frac{\| \mH_{M} \varrho
\|^{2}}{\| \mL_{M} \varrho \|^{2}} \right) \;.  
\end{equ}
\end{itemize}

\end{lemma}
\begin{proof}
Let us start with the passive scalar advection case $L[u]  = u\cdot \nabla $.
Here, by the chain rule we estimate
\begin{equs}
\frac{\ud}{\ud t} \frac{\| \mH_{M} \varrho \|}{ \| \mL_{M} \varrho \|} &  =
\frac{\langle \mH_{M} \varrho, \kappa \Delta \varrho - u\cdot \nabla \varrho
\rangle}{\| \mH_{M} \varrho_{t} \| \| \mL_{
M} \varrho \|} -  \frac{\| \mH_{M} \varrho \|}{\|
\mL_{M} \varrho \|^{3}} \langle \mL_{M} \varrho, \kappa \Delta
\varrho - u\cdot \nabla \varrho \rangle  \\
& \leqslant  - \frac{\langle \mH_{M}
\varrho, u\cdot \nabla \varrho\rangle}{\| \mH_{M} \varrho \| \| \mL_{M} \varrho
\|} +  \frac{\| \mH_{M} \varrho \|}{\|
\mL_{M} \varrho \|^{3}} \langle \mL_{M} \varrho,
u\cdot \nabla \varrho\rangle \\
&= - \frac{\langle \mH_{M}
\varrho, u\cdot \nabla (\mL_{M}\varrho)\rangle}{\| \mH_{M} \varrho \| \| \mL_{M} \varrho
\|} +  \frac{\| \mH_{M} \varrho \|}{\|
\mL_{M} \varrho \|^{3}} \langle \mL_{M} \varrho,
u\cdot \nabla (\mH_{M}\varrho)\rangle  \;. 
\end{equs} 
Here we have used that $u$ is divergence free, so that $\langle \mH_{M} \varrho,  u\cdot \nabla (\mH_{M}\varrho) \rangle =0$. Therefore, we obtain that 
\begin{equ}
\frac{\ud}{\ud t} \frac{\| \mH_{M} \varrho \|}{ \| \mL_{M} \varrho \|}  \lesssim M \| u_{t}\|_{\infty} \left(  1 +  \frac{\| \mH_{M} \varrho\|^{2}}{\| \mL_{M} \varrho \|^{2}} \right)\;,
\end{equ} 
where we again used that since $u$ is divergence free, so that
\begin{equ}
\langle \mL_{M} \varrho,
u\cdot \nabla (\mH_{M}\varrho)\rangle  = \langle \nabla (\mL_{M}\varrho), (\mH_{M}\varrho) u\rangle \;.
\end{equ}
Now consider the case $L[u]\varrho  = u\cdot \nabla \varrho + \Delta u\cdot
\nabla^{-1}\varrho $. As above, we have 
\begin{equs}
\frac{\ud}{\ud t} \frac{\| \mH_{M} \varrho \|}{ \| \mL_{M} \varrho \|} & \leqslant  - \frac{\langle \mH_{M}
\varrho, u\cdot \nabla \varrho+\Delta u\cdot \nabla^{-1}\varrho\rangle}{\| \mH_{M} \varrho \| \| \mL_{M} \varrho\|} +  \frac{\| \mH_{M} \varrho \|}{\|\mL_{M} \varrho \|^{3}} \langle \mL_{M} \varrho,u\cdot \nabla \varrho+\Delta u\cdot \nabla^{-1}\varrho\rangle \;.
\end{equs} 
Now via the triangle inequality we obtain
\begin{equs}
\left|\langle \mH_{M}\varrho, \Delta u\cdot \nabla^{-1}\varrho\rangle\right|&\leqslant \left|\langle \mH_{M}\varrho, \Delta u\cdot \nabla^{-1}(\mH_{M}\varrho)\rangle\right|+\left|\langle \mH_{M}\varrho, \Delta u\cdot \nabla^{-1}(\mL_{M}\varrho)\rangle\right|\\
&\lesssim \| \mH_{M}\varrho \| \| \Delta u \|_{\infty}   \| \mH_{M} \varrho\| + \| \mH_{M}\varrho \|\| \Delta u \|_{\infty}   \| \mL_{M} \varrho\|\;,
\end{equs} and similarly 
\begin{equ}
\left|\langle \mL_{M}\varrho, \Delta u\cdot \nabla^{-1}\varrho\rangle\right|\leqslant \| \mL_{M} \varrho\| \|\Delta u\|_{\infty} (\| \mH_{M} \varrho\| + \| \mL_{M} \varrho\|)\;.
\end{equ} The desired result follows by combining these bounds with the
estimate for the previous case $L[u]\varrho  = u\cdot \nabla \varrho $.
\end{proof}
We can use the previous result to obtain a bound on negative
moments for the stopping time $\sigma(M)$ defined in \eqref{e:sigma}. In other words, the following lemma says that the high-frequency regularity of the projective process $\pi_{t}$ does not increase too quickly.
\begin{lemma}\label{lem:ngemmts}
Consider any $ M \geqslant  M (\varrho_{0}) $ and let $ \sigma= \sigma (M) $ be
defined as in \eqref{e:sigma}. Then for any $ \gamma > 2 $ there exists an $\r>1$ such that the
following holds: for any $p> 1$, uniformly over $ M \geqslant M(\varrho_{0}) $
there exists a $C(p)$ such that uniformly over $ \kappa \in (0,1] $
\begin{equ}
\EE [ (M \sigma)^{- p}] \leqslant C(p) ( \| u_{0} \|_{\mC^{\gamma}}+1)^{\r p } \;.
\end{equ}
\end{lemma}
\begin{proof} 
We find by Lemma~\ref{lem:hl} that for some $ C > 0 $
\begin{equ}
\frac{\| \mH_{M} \varrho_{\sigma}  \|}{\| \mL_{M} \varrho_{\sigma} \|}
\leqslant 1 +  \sigma C M \left( \sup_{0 \leqslant s \leqslant \sigma} \|
u_{s} \|_{\infty} +\sup_{0\leqslant s \leqslant \sigma} \|\Delta  u_{s}\|_{\infty}  \right) \;,
\end{equ}  in both cases of non-linearity $L[u]\varrho$. Provided that $\a >
4$, by Lemma~\ref{lem:mmt-unif} and Markov's inequality there exists an $\r>1$
such that for any $ \gamma > 2 $
\begin{equ}
\PP ( M \sigma \leqslant t) \leqslant \PP \left( \sup_{0 \leqslant s \leqslant
t} \| u_{s} \|_{\mC^{\gamma}} \geqslant (C t)^{-1}  \right) \lesssim (Ct)^{p} (
\| u_{0} \|_{\mC^{\gamma}}+1)^{\r p }\;, 
\end{equ}
so that the result follows.
\end{proof}
Finally, building on the previous lemmata, we are able to study the regularity for the projective
process above the spectral median. Recall the definitions of $ \ts, \tsd $ and $ \sigma $ in \eqref{e:t-kappa}, \eqref{e:tsd} and \eqref{e:sigma}. We set
\begin{equ}[e:sigma-bar]
\overline{\sigma} (M)= \sigma(M) \wedge t_{\star}^{\kappa, \delta} (M)
 =  \inf \{ t \geqslant 0 \; \colon \; M^{(2)}(\varrho_{t}) > M \} \wedge \lambda\kappa^{-1} M^{-2+\delta }\log (M)\;,
\end{equ} for arbitrary $\kappa, \delta \in (0, 1] $ and $ \lambda $ as in
\eqref{e:t-kappa}.
In this setting we obtain the following
regularity estimate.
\begin{proposition}\label{prop:regu} 
Fix any $ \q > 2 $.  For any $ p \geqslant 1 $ there exists a $ C (p) > 0 $ such that the following holds.
Fix any $ \varrho_{0} \in L^{2}_{*} $. Then there exists $ \r > 1 , \gamma \in
( 2, \a/2)$ such that for $ M \geqslant M (\varrho_{0})\vee \kappa^{-\q} $, and uniformly over $
\kappa \in (0, 1] $
\begin{equ}
\EE \left[ \| \pi_{ \overline{\sigma} (M) } \|_{H^{1}}^{2p} \right]
\lesssim_{p, \delta, \q} M^{p}(\|
u_{0}\|_{\mC^{2}}+1)^{\r p}\;,
\end{equ}
where $ \overline{\sigma} $, depending on $ \delta \in (0, 1) $, is defined in \eqref{e:sigma-bar}. 
\end{proposition}
\begin{proof}
We only prove the case $L[u] = u\cdot \nabla + \Delta u \cdot \nabla^{-1}$ (linearised stochastic Navier--Stokes equation) since the proof of the other case (passive scalar advection) is simpler. 

Up to rescaling the initial condition, let us assume without loss of generality
that $ \| \varrho_{0} \| = 1 $.
Let us define the process $ \mZ $ as follows, for all $ t \leqslant
\overline{\sigma}$:
\begin{equ}
\mZ_{t} = \frac{\mH_{2M} \varrho_{t}}{\| \mL_{ M} \varrho_{t} \|} \;.
\end{equ} Observe that we have $\|\mZ_{t} \| \leqslant 1/2 $ since  $ \| \varrho_{t} \| \leqslant 2 \| \mL_{M}\varrho_{t} \| $ for $ t \leqslant \overline{\sigma}$ . Now, it suffices to control the regularity of $ \mZ $, since we can estimate
for $ t \leqslant \overline{\sigma} $:
\begin{equ}
\| \pi_{t} \|_{H^{1}} \lesssim M + \| \mH_{2M} \varrho_{t} / \|
\mL_{M} \varrho_{t} \| \|_{H^{1}}  \simeq  M + \| \mZ_{t}
\|_{H^{1}} \;,
\end{equ}
where we again used that $ \| \varrho_{t} \| \leqslant 2 \| \mL_{M}
\varrho_{t} \| $ for $ t \leqslant \overline{\sigma} $. Therefore, our
objective is to obtain a regularity estimate on $ \mZ $. For this purpose, let
us view $ \mZ $ as the solution to the following parabolic equation:
\begin{equs}
\partial_{t} \mZ & = \kappa \Delta \mZ - \frac{\mH_{2 M} \div
(R[u]\varrho ) }{\| \mL_{ M} \varrho \|} -
\frac{\mH_{2 M} \varrho}{\| \mL_{  M} \varrho
\|^{3}} \langle \mL_{ M} \varrho, \kappa \Delta \varrho - \div
(R[u]\varrho ) \rangle \\
& = \kappa (  \Delta - \zeta )\mZ - \frac{\mH_{2 M} \div
(R[u]\varrho) }{\| \mL_{ M} \varrho \|} +
\frac{\mH_{2 M} \varrho}{\| \mL_{  M} \varrho
\|^{3}} \langle \mL_{ M} \varrho, \div
(R[u]\varrho)  \rangle \;, \label{e:Z}
\end{equs}
where $R[u]\varrho  = \varrho u + \Delta u 
(-\Delta)^{-1}\varrho$ as in \eqref{e:R-nF} and 
\begin{equ}
\zeta_{t} =  \frac{\| \nabla \mL_{ M} \varrho_{t} \|^{2}}{ \|
\mL_{M} \varrho_{t}\|^{2}} \leqslant  M^{2} \;.
\end{equ}
Now we view \eqref{e:Z} as an equation driven by the time-dependent linear operator $ \kappa
(\Delta - \zeta_{t}) $. Therefore, we define the time-inhomogeneous operator
and associated semigroup
\begin{equ}
L_{t} = \kappa (\Delta - \zeta_{t}) \;, \qquad Q_{s, t}^{\kappa} = e^{\kappa \int_{s}^{t} L_{r} \ud r}
\;,
\end{equ}
so that we can represent
\begin{equs}
\mZ_{t} & = Q^{\kappa}_{0, t} \mZ_{0} + \int_{0}^{t} Q^{\kappa}_{s, t} [f_{s}] \ud s \;, \\
f & = - \frac{\mH_{2M} \div
(R[u]\varrho) }{\| \mL_{M} \varrho \|} +
\frac{\mH_{2M} \varrho}{\| \mL_{ M} \varrho
\|^{3}} \langle \mL_{ M} \varrho, \div
(R[u]\varrho)  \rangle \;. \label{e:force}
\end{equs}
In order to establish appropriate bounds on $ \mZ $, we will use the parabolic
smoothing effect of $ Q^{\kappa}_{s, t} $. Indeed, since $ \zeta_{t} \leqslant
M^{2} $, we can bound
\begin{equ}[e:schauder]
\| Q_{s,t}^{\kappa} \mH_{2M} \varphi \|_{H^{\alpha + \beta}} \lesssim (\kappa (t-s) )^{-
\frac{\beta}{2}} \| \mH_{2M} \varphi \|_{H^{\alpha}} \;,
\end{equ}
so that at frequencies higher than $ 2M $, the operator $ Q $ has the same
smoothing effect of the heat semigroup $ P^{\kappa} $.
Therefore, we obtain the following bound for the evolution of the initial
condition:
\begin{equ}
\| Q^{\kappa}_{0, \overline{\sigma}} \mZ_{0} \|_{H^{1}}  \lesssim
\kappa^{- \frac{1}{2}} \overline{\sigma}^{- \frac{1}{2}}
 \| \mZ_{0} \|_{L^{2}} \lesssim \kappa^{-
\frac{1}{2}} \overline{\sigma}^{- \frac{1}{2}}\;.
\end{equ}
Similarly, we can bound uniformly in time for any $ \beta > 0 $
\begin{equ}[e:regu-ic]
\sup_{0 \leqslant s \leqslant \overline{\sigma}} s^{\frac{\beta}{2} }\| Q^{\kappa}_{0, s}
\mZ_{0} \|_{H^{\beta}}  \lesssim \kappa^{- \frac{\beta}{2}} \;.
\end{equ}
The latter bound will be useful in the remainder of the discussion. In any
case, our first bound guarantees that we may apply Lemma~\ref{lem:ngemmts} to obtain
\begin{equs}
\EE  \left[ \| Q^{\kappa}_{0, \overline{\sigma}} \mZ_{0}
\|^{p}_{H^{1}} \right]
\lesssim \kappa^{- \frac{p}{2}} \EE \,  \overline{\sigma}^{- \frac{p}{2}}
 &\lesssim \kappa^{- \frac{p}{2}} M^{\frac{p}{2}} (\|
u_{0}\|_{\mC^{2}}+1)^{\frac{\r p}{2}} + \kappa^{- \frac{p}{2}} \tsd
(M)^{- \frac{p}{2}} \\
& \lesssim_{\delta} \kappa^{- \frac{p}{2}} M^{\frac{p}{2}} (\|
u_{0}\|_{\mC^{2}}+1)^{\frac{\r p}{2}} + M^{p}\;. \label{e:mmt-ic}
\end{equs}
After having controlled the initial condition, we must control the force on the
right hand-side of \eqref{e:Z}. We start with the first term in
\eqref{e:force}. Here we bound for any $ \beta \in [0,1) $:
\begin{equ}
\bigg\| \int_{0}^{t} Q^{\kappa}_{s, t} \left[ \frac{\mH_{2M} \div
(R[u_{s}]\varrho_{s})  }{\| \mL_{ M} \varrho_{s} \|}  \right] \ud s
\bigg\|_{H^{\beta}} \leqslant \int_{0}^{t} \| Q^{\kappa}_{s, t} \varphi_{s}
\|_{H^{1+ \beta}}   \ud s \;.
\end{equ}
Here we have defined $ \varphi_{s} = \mH_{2M} 
(R[u_{s}]\varrho_{s})  / \| \mL_{M} \varrho_{s} \|$. Now using the bound  
\begin{equ}
\| R[u]\varrho \|\lesssim \| u\|_{\infty} \| \varrho \| + \| \Delta
u\|_{\infty} \| (-\Delta)^{-1} \varrho \| \lesssim \|u\|_{\mC^{\gamma }} \|\varrho\|\;,
\end{equ}  
for all $ \gamma \in (2, \a/2) $, we obtain that for all $ t \leqslant
\overline{\sigma} $, following the notation of \eqref{e:cbt}:
\begin{equ}
\bigg\| \int_{0}^{t} Q^{\kappa}_{s, t} \left[ \frac{\mH_{2M} \div
(R[u_{s}]\varrho_{s}) }{\| \mL_{ M} \varrho_{s} \|}  \right]  \ud s
\bigg\|_{H^{\beta}}  \lesssim 
 \| u \|_{\mC^{\gamma}_{t}}  \cdot \int_{0}^{t}  (\kappa s)^{- \frac{1 + \beta}{2}} \ud s  \lesssim_{\beta} \kappa^{- \frac{1 + \beta}{2}} t^{\frac{1 - \beta}{2}
}\| u \|_{\mC^{\gamma}_{t}}  \;.\label{e:bd-frc-1}
\end{equ}
As for the second term in \eqref{e:force}, we start with the estimate
\begin{equ}
\frac{|\langle \mL_{ M} \varrho_{s},\div
(R[u_{s}]\varrho_{s}) \rangle | }{\| \mL_{ M} \varrho_{s} \|^{2}} = \frac{|\langle \nabla(\mL_{ M} \varrho_{s}),R[u_{s}]\varrho_{s} \rangle | }{\| \mL_{ M} \varrho_{s} \|^{2}} \lesssim
M \| u_{s} \|_{\mC^{ \gamma}} \;,
\end{equ}
for any $ \gamma \in (2, \a/2) $,
where once more we used that by assumption $ \| \mL_{ M } \varrho_{s} \|
\gtrsim \| \varrho_{s} \| $ for all $ s \leqslant \overline{\sigma} $.
Therefore via \eqref{e:schauder} we obtain the upper bound
\begin{equs}
\left\| \int_{0}^{t} Q^{\kappa}_{s, t} \left[ \frac{\mH_{2M}
\varrho_{s}}{\| \mL_{M} \varrho_{s}
\|^{3}} \langle \mL_{ M} \varrho_{s},\div(R[u_{s}]\varrho_{s})  \rangle \right]
\ud s \right\|_{H^{1}} & \lesssim \| u \|_{\mC^{\gamma}_{t}}
 \kappa^{- \frac{1}{2}} M    \int_{0}^{t}   s^{- \frac{1}{2}} \ud s \\
& \lesssim \| u \|_{\mC^{\gamma}_{t}} M \kappa^{- \frac{1}{2}} t^{\frac{1}{2} }\\
&\lesssim_{\delta, \ve} \| u \|_{\mC^{\gamma}_{t}} \kappa^{-1}M^{\delta/2+\ve}\;,\label{e:bd-frc-2}
\end{equs}
for all $ t \in [0, \overline{\sigma}]  $, and where we have used that $ t
\leqslant \tsd(M) $ and have absorbed the logarithmic terms in a small power $
M^{\ve} $, for arbitrary $ \ve \in (0, 1) $. Now combining \eqref{e:regu-ic} with \eqref{e:bd-frc-1} and
\eqref{e:bd-frc-2} and using that $(\tsd(M))^{\beta/2}M^{\delta/2 +\ve} \leqc
M^{-\beta/2 +\delta/2 +\ve}\leqc 1$ if $\beta \in (\delta, 1) $ and $ \ve
$ is sufficiently small, we find that there exists a constant $C(\beta)>0$ such that 
\begin{equ}[e:bd-intermediate]
\sup_{0 \leqslant s \leqslant \overline{\sigma}} s^{ \frac{\beta}{ 2}} \| \mZ_{s}
\|_{H^{\beta}} \leqslant C(\beta) \kappa^{-
1}  (1 + \| u \|_{\mC^{ \gamma}_{t}})\;.
\end{equ} 
Since our objective is to obtain an estimate on the $ H^{1}$ norm of
$ \mZ_{ \overline{\sigma}} $, we must still improve our bound to allow for $
\beta = 1 $. Here we bootstrap our argument and go back to \eqref{e:bd-frc-1}
and bound. In view of \eqref{e:bd-intermediate} we obtain:
\begin{equs}
\bigg\| \int_{0}^{t} Q^{\kappa}_{s, t}  \left[ \frac{\mH_{2M} \div
(R[u]\varrho) }{\| \mL_{ M} \varrho \|}  \right]  \ud s \bigg\|_{H^{1}} & \lesssim \int_{0}^{t} \| Q^{\kappa}_{s, t} \varphi_{s}
\|_{H^{2}}   \ud s \\
& \lesssim  \| u \|_{\mC^{\gamma}_{t}} 
\cdot \int_{0}^{t}   ( \kappa (t-s) )^{- \frac{2 - \beta}{2}
}\| \mZ_{s} \|_{H^{\beta}} \ud s\\
& \lesssim \left( 1 +   \| u
\|_{\mC^{\gamma}_{t}} \right)^{2} \kappa^{-1} 
\int_{0}^{t}  (\kappa (t-s) )^{- \frac{2 - \beta}{2}
}  s^{- \frac{\beta}{2}} \ud s \\
& \lesssim_{\beta}  \left( 1 +   \| u
\|_{\mC^{\gamma}_{t}} \right)^{2} \kappa^{-1} \kappa^{- \frac{2- \beta}{2}}
\lesssim_{\beta, \ve} \left( 1 +   \| u \|_{\mC^{\gamma}_{t}} \right)^{2} \kappa^{-
\frac{3}{2} - \ve}\;. \qquad
\label{e:bd-frc-3}
\end{equs}
for any $ \ve > 0 $, provided $ \beta $ is sufficiently close to one.
Now we are ready to conclude the proof. Combining \eqref{e:bd-frc-3} with
\eqref{e:bd-frc-2} and \eqref{e:regu-ic} we obtain that
\begin{equs}
\| \mZ_{ \overline{\sigma}} \|_{H^{1}} \lesssim \kappa^{-
\frac{1}{2}} \overline{\sigma}^{- \frac{1}{ 2}} + \left( 1 +   \| u \|_{\mC^{\gamma}_{t}} \right)^{2} \kappa^{-
\frac{3}{2} - \ve} \;.
\end{equs}
Hence, upon taking expectations and through \eqref{e:mmt-ic} and
Lemma~\ref{lem:mmt-unif}, we obtain
\begin{equs}
\EE \left[ \| \pi_{ \overline{\sigma}} \|^{p}_{H^{1}} \right] & \lesssim
 \left( \kappa^{- \frac{1}{2}} M^{\frac{1}{2}} + \kappa^{- \frac{3}{2}
- \ve} \right)^{p} (\|
u_{0}\|_{\mC^{2}}+1)^{\frac{\r p}{2}} + M^{p} \\
& \lesssim   \kappa^{- \frac{p}{2}} M^{\frac{p}{2}}(\|
u_{0}\|_{\mC^{2}}+1)^{\frac{\r p}{2}} + M^{p}\\
& \lesssim   M^{p}(\| u_{0}\|_{\mC^{2}}+1)^{\frac{\r p}{2}} \;,
\end{equs}
where the last two steps follow under the assumption $ M \geqslant
\kappa^{-\q} $ for some $ \q > 2 $, by choosing $ \ve $ sufficiently small.
This concludes the proof.
\end{proof}
This concludes our first set of results concerning high-frequency regularity.
\subsection{Energy estimates for higher regularity}
We conclude this section with a deterministic upper bound on the regularity of
the projective process, which follows from an energy estimate. This estimate
can also be found in \cite[Section 2.3.2]{miles2018diffusion} in the case of
pure advection.

\begin{lemma}\label{lem:ub-h1-Lns}
Let $\varrho$ be a solution to \eqref{e:main} with $L[u] $ satisfying
\eqref{e:form-PSA} or \eqref{e:form-LNS}. Then for any $\gamma\in (2,3)$ and $0\leqslant s \leqslant t$, there exists a $C(\gamma)>1$ such that: for the projective process $\pi=\varrho/\|\varrho\|$ we have 
\[
\|\pi_{t} \|_{H^{1}}\leqslant C(\gamma) \exp\left(C(\gamma)\int_{s}^{t} \|u_{r}\|_{H^{\gamma}}\ud r \right)\left(\|\pi_{s}\|_{H^{1}}+ (t-s)
\sup_{s\leqslant r \leqslant t}\|u_{r}\|_{H^{\gamma +1}}\right)\;.
\]
Therefore, for any $\r \geqslant 1 $ there exists a $C(\r,t)>0$ such that 
\[
\E \left[ \| \pi_t\|_{H^{1}}^{\r} \right] \leqslant C(\r,t) \EE \left[\| \pi_{s}\|_{H^{1}}^{2\r} +1\right]^{\frac{1}{2}} V_{\r}^{\frac{1}{2}}(u_{0})\;,
\] 
where $V_{\r}$ is defined in \eqref{e:lyap-sns}. 
\end{lemma}
\begin{proof}
Since the result in the case of advection $ L[u] = u \cdot \nabla $ is simpler
than in the case of linearised SNS, we only consider the latter.
   We start by estimating $\varrho$ in $L^{2}$. By the divergence-free property of $u$, we have
   \[
   \frac{1}{2}\frac{\ud}{\ud t}\|\varrho_{t}\| + \kappa\|\nabla \varrho_{t}\|^{2}
= -\langle \varrho_{t},\Delta u_{t} \cdot \nabla^{-1}\varrho_{t}\rangle \;.
   \]
For any $\gamma\in(2,3)$, we choose $p>2$ so that $\gamma = 2+ \frac{p-2}{p}$ and let $p^{\prime} = 2p/(p-2)$. Then using Sobolev embedding, we can bound 
  \begin{equs}
      |\langle \varrho_{t},\Delta u_{t} \cdot \nabla^{-1}\varrho_{t}\rangle| &\lesssim \|\varrho_{t}\|\|\nabla^{-1}\varrho_{t}\|_{L^{p^{\prime}}}\|\Delta u_{t}\|_{L^{p}}\\
      &\lesssim \|\varrho_{t}\|^{2}\|u_{t}\|_{H^{\gamma}} \;,
  \end{equs} which gives us 
  \begin{equ}\label{e:L2est-rho-Lns}
     \frac{1}{2}\frac{\ud}{\ud t}\|\varrho_{t}\| \lesssim -\kappa\|\nabla \varrho_{t}\|^{2} + \|\varrho_{t}\|^{2}\|u_{t}\|_{H^{\gamma}}\;.
  \end{equ}
   Next we estimate $\varrho$ in $H^1$. Using the divergence-free property of $u$,
and integrating by parts leads to the estimate
  \begin{equs}
     \frac{1}{2}\frac{\ud}{\ud t}\|\nabla\varrho_{t}\|^{2} + \kappa\|\Delta\varrho_{t}\|^{2} &= -\langle \nabla \varrho_{t}, \nabla u_{t}\cdot \nabla \varrho_{t}\rangle - \langle \nabla \varrho_{t}, \nabla(\Delta u_{t} \cdot \nabla^{-1}\varrho_{t})\rangle\\
   &\lesssim \|\nabla \varrho_{t}\|^{2} \|\nabla u_{t}\|_{\infty} + |\langle \nabla
\varrho_{t}, \nabla(\Delta u_{t} \cdot \nabla^{-1}\varrho_{t})\rangle| \;.
  \end{equs}
   Note that we can bound the second term in a similar way to the $L^2$ setting above using Sobolev embedding
\begin{equs}
    |\langle \nabla \varrho_{t}, \nabla(\Delta u_{t} \cdot \nabla^{-1}\varrho)\rangle| &\lesssim \|\nabla \varrho_{t}\|(\|\varrho_{t}\|_{L^{p}}\|\Delta u_{t}\|_{L^{p^\prime}} + \|\nabla^{-1}\varrho_{t}\|_{L^{p}}\|\nabla\Delta u_{t}\|_{L^{p^\prime}})\\
      &\lesssim \|\nabla \varrho_{t}\|^{2}\|u_{t}\|_{H^{\gamma}} + \|\nabla
\varrho_{t}\|\|\varrho_{t}\|\|u_{t}\|_{H^{\gamma+1}} \;.
\end{equs}
Putting this together with the $L^2$ estimate in \eqref{e:L2est-rho-Lns}, we have
\begin{equs}
   \frac{1}{2}\frac{\ud}{\ud t }\frac{\|\nabla \varrho_{t}\|^2}{\|\varrho_{t}\|^{2}} + \kappa\left(\frac{\|\Delta \varrho_{t}\|^{2}}{\|\varrho_{t}\|^{2}} - \frac{\|\nabla \varrho_{t}\|^{4}}{\|\varrho_{t}\|^{4}}\right) &\lesssim \|u_{t}\|_{H^\gamma}\frac{\|\nabla \varrho_{t}\|^{2}}{\|\varrho_{t}\|^{2}} + \|u_{t}\|_{H^{\gamma+1}}\frac{\|\nabla \varrho_{t}\|}{\|\varrho_{t}\|}\;,
\end{equs} where we used that $\| \nabla u_{t} \|_{\infty} \lesssim \| u_{t} \|_{H^{\gamma}}$. Using the interpolation estimate $\|\nabla \varrho_{t}\| \leqslant \|\Delta \varrho_{t}\|^{1/2}\|\varrho_{t}\|^{1/2}$, and changing to the derivative of $\|\nabla \varrho_{t}\|/\|\varrho_{t}\|$, we have
\[
\frac{\ud}{\ud t}\frac{\|\nabla \varrho_{t}\|}{\|\varrho_{t}\|} \lesssim \|u_{t}\|_{H^\gamma}\frac{\|\nabla \varrho_{t}\|}{\|\varrho_{t}\|} + \|u_{t}\|_{H^{\gamma+1}}.
\]
Integrating the above estimate leads to the first of the desired estimates. Moreover,
Lemma~\ref{lem:lyap-sns} and Lemma~\ref{lem:mmt-unif} imply the second estimate. 
\end{proof}

\section{Analytic estimates}\label{sec:analytic_est}

In this section we collect some analytic estimates that are useful throughout
the paper. We start by introducing Besov spaces. Namely, for a smooth dyadic
partition of the unity $ \{ \chi_{j}
\}_{j \geqslant -1} $ (see \cite[Chapter 2]{BookChemin} , with the convention
that $ \chi_{-1} = \chi,
\chi_{j} = \varphi_{j} $ for $ j \geqslant 0$ in the notation of
\cite{BookChemin}) we define the Paley blocks
\begin{equ}[e:Paley]
\Delta_{j} \varphi = \mF^{-1} ( \chi_{j} \mF \varphi) \;,
\end{equ}
and the Besov space $\mB^{\beta}_{p,q}$ for $\beta\in\RR$ and $p,q\in[1,\infty]$,
through the norm
\begin{equ}[e:besov]
\| \varphi \|_{\mB^{\beta}_{p,q}} = \Bigl( \sum_{j\geqslant -1} 2^{j q \beta } \| \Delta_{j} \varphi \|_{L^{p}}^{q}\Bigr)^{\frac1q}\;.
\end{equ} In particular, we define the H\"older space $ \mC^{\beta} = \mB^{\beta }_{\infty,\infty}$ as 
\begin{equ}[e:holder]
\| \varphi \|_{\mC^{\beta}} = \sup_{j \geqslant -1} 2^{j \beta}\| \Delta_{j} \varphi
\|_{\infty} \;.
\end{equ}
In a similar way, one can also recover the  Bessel norm that we have
been using so far: for $H^\alpha = \mB^{\alpha}_{2,2}$ we define the norm by
\begin{equ}[e:bessel-equiv]
\| \varphi \|_{H^{\alpha}} \simeq  \Bigl( \sum_{j \geqslant -1} 2^{2 \alpha j}  \| \Delta_{j} \varphi
\|_{L^{2}}^{2}\Bigr)^{\frac{1}{2}} \;.
\end{equ}
Paley blocks are particularly useful to study products of distributions. In
particular, for any pair of distributions $ \varphi, \psi $ we can formally
(because the resonant product $ \reso $ is in general not defined) decompose
their product into
\begin{equ}[e:para-dec]
\varphi \cdot \psi = \varphi \para \psi + \varphi \reso \psi + \varphi \rpara
\psi \;,
\end{equ}
with 
\begin{equ}
\varphi \para \psi = \sum_{-1 \leqslant j \leqslant i-1} \Delta_{j} \varphi
\Delta_{i} \psi \;, \qquad \varphi \reso \psi = \sum_{| i - j | \leqslant 1}
\Delta_{i} \varphi \Delta_{j} \psi \;, \qquad \varphi \rpara \psi = \psi \para
\varphi \;.
\end{equ}
In particular, paraproducts are useful because one can estimate them as
follows.
\begin{lemma}[Theorems 2.82 and 2.85 \cite{BookChemin}]\label{lem:para-est}
  Fix $\alpha, \beta \in \RR$. Then uniformly over $\varphi, \psi \in
  \mS^{\prime}$ 
  \begin{align*}
    \| \varphi \para \psi \|_{H^{\alpha}} & \lesssim \| \varphi
    \|_{L^{2}} \| \psi \|_{\mC^{\alpha}} \;, \\
\| \varphi \rpara \psi \|_{H^{\alpha}} & \lesssim \| \varphi
    \|_{H^{\alpha}} \| \psi \|_{L^{\infty}} \;, & &  \\ 
    \| \varphi \para \psi \|_{H^{\alpha + \beta}} &\lesssim \| \varphi
    \|_{H^{\beta}} \| \psi \|_{\mC^{\alpha}} \;, \ \ & &  \text{if} \ \ \beta  <
    0 \;, \\
    \| \varphi \reso \psi \|_{H^{\alpha + \beta}} & \lesssim \| \varphi
    \|_{H^{\beta}} \| \psi \|_{\mC^{\alpha}}\;, \ \ & & \text{if} \ \ \alpha
    {+} \beta > 0 \;.
  \end{align*}
\end{lemma}
Note that as a consequence of Lemma~\ref{lem:para-est} one can bound products
of distributions for $ 0 < \alpha < \beta $ by
\begin{equ}[e:bd-paraproduct]
 \| \varphi \psi
\|_{H^{- \alpha}} \lesssim \| \varphi \|_{H^{- \alpha}} \| \psi
\|_{\mC^{\beta}} \;.
\end{equ}
With this we conclude our first set of analytic estimates. We will now use
paraproducts to obtain an improved bound on the linearisation of SNS.

\subsection{A product bound for linearised SNS}
The aim of this subsection is to establish an estimate on the operator $ R
$ defined in \eqref{e:R-nF} in the context of linearised SNS, where $ L[u] $ is of
the form \eqref{e:form-LNS}.
Note that we can write $L[u]\varrho$ in Fourier space as
\begin{equs}
   L[u]\varrho & = \sum_{\ell \in \ZZ^2_*}\sum_{j+k = \ell} c_{k,j}
\hat{w}(k)\hat{\varrho}(j) e_\ell \\
& = \sum_{\ell \in \ZZ^2_*}\sum_{j+k = \ell} \ell \cdot \mf{d}_{k,j}
\hat{w}(k)\hat{\varrho}(j) e_\ell = \div \left( R[u] \varrho \right) 
\end{equs}
where $w = \mathrm{curl}\, u$ is the vorticity and in the case of linearised SNS
$ \mf{d}, c $ are respectively the Fourier multipliers
\begin{equ}[e:defdkj]
\mf{d}_{k, j} =  k^\perp (|k|^{-2} - |j|^{-2}) \;, \qquad c_{k,j} =\langle
k^\perp, j\rangle (|k|^{-2} - |j|^{-2}) \;.
\end{equ}
Therefore we can rewrite $ R $ in terms of Fourier coefficients by
\begin{equ}[e:def-R]
R[u] \varrho =  \sum_{\ell \in \ZZ^2_*}\sum_{j+k = \ell} \mf{d}_{k,j}
\hat{w}(k)\hat{\varrho}(j) e_\ell \;.
\end{equ}
Now \eqref{e:defdkj} shows that there is a cancellation when $ | k | \sim |
j | $. This cancellation leads to the following estimate.

\begin{lemma}\label{lem:lin-sns-product}
Fix any $ \alpha, \beta > 2 $ and $ \alpha < \beta +1 $, and let $ R $ be as in
\eqref{e:def-R}. Then we can estimate uniformly over all $ \varphi, \psi $:
\begin{equ}
\| R[\varphi] \psi \|_{H^{- \alpha}} \lesssim \| \varphi
\|_{\mC^{\beta}} \| \psi \|_{H^{- \alpha}} \;.
\end{equ}
\end{lemma}

\begin{remark}\label{rem:why-good}
Note that in the case of pure transport we would have $ R[u] \varrho = u
\varrho $, so that the natural assumption for the previous lemma would be $
\alpha < \beta $. Instead, in linearised SNS we have a gain of regularity because of a cancellation
in the Fourier coefficients when similar frequencies of the velocity and the linearisation
interact (in the resonant product). The key point is that we can allow for $
\alpha < \beta +1 $.
\end{remark}

\begin{proof}
To prove the result we split $ R [ \varphi] \psi  $ into paraproducts. Namely,
we define
\begin{equ}
R [ \varphi] \para \psi = \sum_{\ell \in \ZZ^2_*}\sum_{j+k = \ell} \mf{d}_{k,j}
\hat{\zeta}(k)\hat{\psi}(j) \zeta^{\para}(k, j) e^{\iota \ell \cdot x} \;,
\end{equ}
were $ \zeta = \mathrm{curl} ( \varphi)$ and $ \zeta^{\para} (k, j) = \sum_{m \leqslant n-1} \chi_{m} (k)
\chi_{n} (j) $, where $ \{\chi_{m} \}$ is the dyadic partition of the unity
introduced in \eqref{e:Paley}. Similarly, we define $ R[\varphi] \rpara \psi $
and $ R[\varphi] \reso \psi $. Now we bound each paraproduct and the resonant
product separately.

For the first paraproduct $ R [\varphi] \para \psi $ we find from the
definition of $ \mf{d} $ that
\begin{equs}
R [\varphi] \para \psi = ( \nabla^{
\perp} (- \Delta)^{-1} \zeta) \para \psi - (\nabla^{\perp} \zeta) \para
(- \Delta)^{-1} \psi \;.
\end{equs}
From here we obtain for any $ \beta > 2 $, via Lemma~\ref{lem:para-est}
\begin{equ}
\| R [\varphi] \para \psi \|_{H^{- \alpha}} \lesssim \| \nabla^{ \perp}
(- \Delta)^{-1} \zeta \|_{\mC^{\beta}} \| \psi \|_{H^{- \alpha}} + \|
\nabla^{\perp} \zeta\|_{\mC^{\beta-2}} \| (- \Delta)^{-1} \psi \|_{H^{- \alpha}} \lesssim
\| \varphi \|_{\mC^{\beta}} \| \psi \|_{H^{- \alpha}} \;,
\end{equ}
where we used that $ \| \zeta \|_{\mC^{\beta -1}} \lesssim \| \varphi
\|_{\mC^{\beta}} $.
Similarly for the second paraproduct we can estimate, for any $ \beta >0 $:
\begin{equs}
\| R[\varphi] \rpara \psi  \|_{H^{-\alpha}} & = \|( \nabla^{
\perp} (- \Delta)^{-1} \zeta) \rpara \psi - (\nabla^{\perp} \zeta) \rpara
(- \Delta)^{-1} \psi  \|_{H^{- \alpha}} \\
& \lesssim \| \nabla^{
\perp} (- \Delta)^{-1} \zeta \|_{\mC^{\beta}} \| \psi  \|_{H^{- \alpha}} + \|
\nabla^{\perp} \zeta \|_{\mC^{\beta -2}} \| (- \Delta)^{-1} \psi
\|_{H^{- \alpha +2}} \\
& \lesssim \| \varphi \|_{\mC^{\beta}} \| \psi \|_{H^{- \alpha}} \;.
\end{equs}
This leaves us with the resonant product. Let us write
\begin{equs}
\| R[ \varphi ] \reso \psi \|^{2}_{H^{- \alpha}} & = \sum_{\ell} | \ell |^{-
2\alpha} \biggl\vert \sum_{k+j = \ell}
 \mf{d}_{j, k} \hat{ \zeta } (k)  \hat{ \psi} (j) \zeta^{\reso} (k, j)  \biggr\vert^{2} \\
& \lesssim \sum_{\ell} | \ell |^{- 2 \alpha + 2} \biggl\vert \sum_{k+j = \ell} | k |^{-2}
| \hat{\zeta} (k) | | \hat{\psi} (j) | | \zeta^{\reso} (k, j) | \biggr\vert^{2}
\;,
\end{equs}
where $ \zeta^{\reso} (k, j) = \sum_{| m-n | \leqslant 1} \chi_{m}(k)
\chi_{n}(j) $, and we have estimated
\begin{equs}[e:bd-reso]
\big\vert | k |^{-2} - | k- \ell |^{-2} \big\vert & \simeq |
k |^{- 4} \left\vert | k - \ell |^{2} - | k |^{2} \right\vert = | k |^{-4}
\left\vert | \ell |^{2} -2 \langle \ell, k \rangle \right\vert \\
& \lesssim \frac{| \ell |^{2}}{| k |^{4}} + \frac{| \ell |}{| k |^{3}} \lesssim |
\ell | | k |^{-3} \;.
\end{equs}
Here in the last step we have used that necessarily $ | \ell | \lesssim | k
| $ on the set $ | k | \sim | j | $. Now we estimate the inner sum as follows,
for fixed $ \ell \in \ZZ^{2} $, since $ | k | \sim | j | $ on the resonant set
\begin{equs}
 \sum_{k+j = \ell} | k |^{-2}
| \hat{\zeta} (k) | | \hat{\psi} (j) | | \zeta^{\reso} (k, j) | & \lesssim  \sum_{k} | k
|^{\beta-1}
| \hat{\zeta} (k) | | \ell-k |^{- \beta -1} | \hat{\psi} (\ell-k) | |
\zeta^{\reso} (k, j) | \\
& \lesssim \Bigl( \sum_{k} | k
|^{2(\beta-1)} | \hat{\zeta} (k) |^{2} \Bigr)^{\frac{1}{2}} \Bigl(
\sum_{k} |k|^{-2 (\beta +1)} | \hat{\psi} (k) |^{2}\Bigr)^{\frac{1}{2}} \\
& \lesssim \| \varphi \|_{H^{\beta}} \| \psi \|_{H^{-(\beta+1)}} \;.
\end{equs}
Therefore, if $ \alpha $ satisfies $ \alpha < \beta +1 $ and $ \alpha - 1 > d/2
= 1$ (in dimension $ d=2 $, so that the sum over $ \ell $ is finite), we obtain
the desired estimate $ \| R[\varphi] \psi \|_{H^{- \alpha}} \lesssim \| \varphi
\|_{\mC^{\beta}} \| \psi \|_{H^{- \alpha}} $.

\end{proof}

\subsection{Estimates on the velocity field}\label{sec:velocity-estimates}
Next we collect some results on the velocity field $ u $ solving
\eqref{e:sns}.
We first introduce a super-Lyapunov functional
$V_{\r}(u)$ associated to the
solution $u$ to the stochastic Navier--Stokes equation \eqref{e:sns} in
dimension two, for a parameter $ \r \geqslant 0 $ that captures the rate of
polynomial growth and a parameter $ \beta \in (\a/2-2, \a/2-1) $:
\begin{equ}\label{e:lyap-sns}
V_{\r, \beta}(u) = (1+ \| u\|^{2}_{H^{\beta}})^{\r}\exp\left( c_{\star} \|
u\|^{2}_{H^{1}} \right)\;, \qquad V_{\r} = V_{\r, \b} \;,
\end{equ}
where we recall that for simplicity we have fixed a particular choice of
regularity $ \b = (\a-3)/2$:
this choice is rather arbitrary and can replaced by any
$\a/2-2 <  \b < \a/2 -1$. In addition, the Lyapunov functional depends on a constant $
c_{\star} $ which is not important (and therefore we avoid writing explicitly
the dependence on it), but necessary in the bound of the exponential below. The
proof of the next result can be found in \cite[Lemma
3.7]{BBPS-AOP-22}.

\begin{lemma}\label{lem:lyap-sns}
There exists a $ c_{\star} > 0 $, as used in \eqref{e:lyap-sns}, such that
for all $ \beta \in (\a/2 -2 , \a/2-1 ), C_{0} >0, \zeta \in (0, 3)$ and
$T\geqslant 0 $
\begin{equ}
    \EE\left[\exp \biggl(C_{0}\int_{0}^{T} \|u_s \|_{H^{\zeta}}\ud s
\biggr) \right] \lesssim V_{r, \beta}(u_{0})\;,
 \end{equ} 
where the proportionality constant depends on all the parameters of the
problem. Moreover, for any $ \r > 1 $ there exists a constant $ \gamma>0$ such
that uniformly over all $t > 0 $,
 \begin{equ}
  \EE\left[ V_{\r, \beta}(u_{t})\right] \lesssim e^{- \gamma t}
V_{\r, \beta}(u_{0}) + 1 \;.
 \end{equ}
\end{lemma}

In the next result we consider the $\mC^{\beta}$ regularity for $\beta$ arbitrarily close to $\a/2$ for the velocity field. Note that in the case  $u_{0}\in H^{\beta}$, we need an intermediate time to regularise. For the next result recall that we write $ X $ for the Gaussian process $ X_{t} = P_{t} u_{0} +
\int_{0}^{t} P_{t-s} \mathbf{P} \xi \ud s $.

\begin{lemma}\label{lem:mmt-unif}
Fix any $t_{0} > 0$. There exists an $ \r >1 $ such that the following holds: for any $ p > 0, 2<  \beta < \a/2 $ and $t>t_{0} $, there exists a $C(t_{0})>0$ such that for any $u_{0}\in H^{\b}$ with $\b= (\a-3)/2$ 
\begin{equ}
 \EE \left[ \sup_{t_{0} \leqslant s \leqslant t}  \left( \| X_{s} \|_{\mC^{\beta}
}^{p} + \| u_{s} \|_{\mC^{\beta}
}^{p} \right) \right] \leqslant C(t_{0}) (\| u_{0} \|_{H^{\b}} + 1)^{p \r } \;.      
\end{equ} Moreover, there exists a $C>0$ such that for any $u_{0}\in \mC^{\beta} $ 
\begin{equ}
 \EE \left[ \sup_{0 \leqslant s \leqslant t}  \left( \| X_{s} \|_{\mC^{\beta}
}^{p} + \| u_{s} \|_{\mC^{\beta}
}^{p} \right) \right] \leqslant C (\| u_{0} \|_{\mC^{\beta}} + 1)^{p \r } \;.
\end{equ}
\end{lemma}

The proof of these estimates follow along standard lines, through an enstrophy
estimate and a bootstrap argument. See also \cite[Section 3]{MR4682783} for a
similar approach, in a more difficult setting. Note also that the restriction
$ \beta > 2 $ is artificial, and can be removed by following a slightly
modified proof.
\begin{proof}
From its definition, the process $X$ satisfies that for any $t\geqslant
t_{0}$, and for any $ \beta < \a/2 $, via Schauder estimates:
\begin{equ}
\| X_{t} \|_{\mC^{\beta}} \leqslant \| P_{t} u_{0} \|_{\mC^{\beta}} + \left\|
\int_{0}^{t} P_{t-s} \mathbf{P} \xi \ud s \right\|_{\mC^{\beta}} \lesssim t^{-
\frac{5}{4}}\| u_{0} \|_{\mC^{\beta-5/2}} +  \left\|
\int_{0}^{t} P_{t-s} \mathbf{P} \xi \ud s \right\|_{\mC^{\beta}}\;.
\end{equ}
By Besov embedding and since $ \b=(\a-3)/2 < \beta - 3/2$, we find $\| u_{0}
\|_{\mC^{\beta-5/2}} \lesssim \| u_{0} \|_{H^{\b}}$
and moreover the stochastic convolution convolution term lies in $ 
\mC^{\beta} $ for any $ \beta < \a/2 $ with all moments finite (by Fernique's
theorem for example). Therefore, we can estimate
\begin{equ}
\EE \left[ \sup_{t_{0} \leqslant s \leqslant t} \| X_{s}
\|_{\mC^{\beta}}^{p} \right] \leqslant C(t_{0})( \| u_{0} \|_{H^{\b}} + 1 )^{p} \;.
\end{equ} Now we move on to the estimate for $ u $. We start by writing $ u = \widetilde{X}  + v $, where $ v $ solves
\begin{equ}
\partial_{t} v + \mathbf{P} \div ( v + \widetilde{X})^{\otimes_{s} 2} = \Delta v \;, \qquad
v (0, \cdot) = u_{0} (\cdot) \;,
\end{equ}
and $ \widetilde{X} $ is the stochastic convolution $ \widetilde{X}_{t} =
\int_{0}^{t} P_{t-s} \mathbf{P} \xi  \ud s $. We establish an a-priori bound on $ v $ through an enstrophy estimate.
Define $ \omega = \mathrm{curl} ( v) $ and $ \chi = \mathrm{curl} (
\widetilde{X}) $. Then $ \omega $ solves
\begin{equ}
\partial_{t} \omega + K * (\omega + \chi) \cdot \nabla
(\omega + \chi)= \Delta \omega  \;, \qquad \omega (0, \cdot) = \mathrm{curl}(u_{0}) (\cdot) \;,
\end{equ}
where $ K $ is the Biot--Savart kernel. Then we find
\begin{equ}
\frac{1}{2} \partial_{t} \| \omega \|^{2} = - \| \nabla \omega \|^{2} - \langle
\omega, K * (\omega + \chi) \cdot \nabla (\omega + \chi)\rangle \;.
\end{equ}
Then, for the last term on the right hand-side we find by antisymmetry that
\begin{equ}
\langle \omega, K * (\omega + \chi) \cdot \nabla (\omega + \chi)\rangle =
\langle \omega, K * (\omega + \chi ) \cdot \nabla \chi \rangle \;.
\end{equ}
Now $ | \langle \omega, (K * \omega) \cdot \nabla \chi \rangle | \lesssim \|
\omega \|^{2} \| \nabla \chi \|_{ \infty} $, and $ | \langle \omega, (K * \chi)
\cdot \nabla \chi  \rangle | \lesssim \| \omega \| \| \chi
\|_{\infty} \| \nabla \chi \| $. Therefore for any $ \beta > 2 $, applying
Young's inequality $ \| \omega \| \| \chi
\|_{\infty} \| \nabla \chi \| \leqslant (1/2) \| \omega \|^{2} + C \| \chi
\|_{\mC^{\beta-1}}^{2}$, and absorbing the first term in the dissipation, we
obtain
\begin{equ}
\frac{1}{2} \partial_{t} \| \omega \|^{2} \lesssim \| \omega \|^{2}  \| \chi
\|_{\mC^{ \beta -1 }} + \| \chi \|_{\mC^{\beta -1}}^{2} \;.
\end{equ}
Integrating out, we have found that
\begin{equ}[e:bd-v]
\sup_{0 \leqslant s \leqslant t} \| v_{s} \|_{H^{1}}^{2} \lesssim e^{C \int_{0}^{t}
\| \widetilde{X}_{s} \|_{\mC^{\beta}} \ud s } \left( \| u_{0}
\|_{H^{1}}^{2}  +
\int_{0}^{t} \| \widetilde{X}_{s} \|_{\mC^{\beta}}^{2} \ud s \right) \;.
\end{equ}
In particular, choosing $ \beta \in (2, \a/2) $ and since $ \widetilde{X} $ is Gaussian and takes values in $
\mC^{\beta} $ for any $ \beta < \a/2 $, we deduce by Fernique's theorem that since $\b>1$
\begin{equ}
\EE \left[ \sup_{0 \leqslant s \leqslant t} \| v_{s} \|_{H^{1}}^{p} \right]
\lesssim (\| u_{0} \|_{H^{\b}} + 1)^{p} \;,
\end{equ}
for any $ p \geqslant 0 $. Now, for any $ \delta > 0
$ we find by Besov embedding that  $ \| v
\|_{\mC^{- \delta}} \lesssim \| v \|_{H^{1}} $, so that $ \| \div
(v^{\otimes_{s}2}) \|_{H^{- \delta}} \lesssim \| v
\|_{H^{1}}^{2} $. Therefore, if we set ${\bf M}_{t} = \sup_{0 \leqslant s \leqslant t} \| v_{s} \|_{H^{1}}$, we find via
Schauder estimates and through the Besov embeddings $H^{\b}\subseteq H^{2-\delta} \subseteq
\mC^{1 - \delta} $ that
\begin{equs}
\| v_{t} \|_{\mC^{1 - \delta}} & \leqslant \| u_{0} \|_{\mC^{1- \delta}} + \left\|
\int_{0}^{t}  P_{t-s} \div ( (v_{s} + \widetilde{X}_{s})^{\otimes_{s} 2})
 \ud s \right\|_{H^{2- \delta}}\\
& \lesssim_{\delta} \| u_{0} \|_{H^{\b}} + ({\bf M}_{t}+1)^{2} (\| \widetilde{X}
\|_{\mC^{\beta}_{t}} +1)^{2} \;,
\end{equs}
for any $ \beta > 1 $ and $ \delta > 0 $. Therefore, through \eqref{e:bd-v} we have proven that
\begin{equ}
\EE \left[ \sup_{0 \leqslant s \leqslant t} \| v_{s} \|_{\mC^{1 -
\delta}}^{p} \right] \lesssim (\| u_{0} \|_{H^{\b}} + 1)^{2 p} \;.
\end{equ} Now by considering the initial data $v_{t_{0}}$, one can iterate these estimates using Schauder estimates until one finds an $ \r > 2 $ such that for any $2<\beta <\a/2$  and $ p \geqslant 0 $
\begin{equ}
\EE \left[ \sup_{t_{0} \leqslant s \leqslant t} \| v_{s} \|_{\mC^{\beta}
}^{p} \right] \leqslant C(t_{0}) (\| u_{0} \|_{H^{\b}} +1)^{\r p} \;.
\end{equ}
This completes the proof of the first statement of the lemma. The second statement follows similarly in a simpler way. Indeed, one can directly bound for $\|X\|_{\mC^{\beta}}$
\begin{equ}
\| X_{t} \|_{\mC^{\beta}} \leqslant \| u_{0} \|_{\mC^{\beta}} + \left\|
\int_{0}^{t} P_{t-s} \mathbf{P} \xi \ud s \right\|_{\mC^{\beta}}\;.
\end{equ} Now for any $\beta\in(2,\a/2)$ there exists a $C>0$ such that 
\begin{equ}
\EE \left[ \sup_{0 \leqslant s \leqslant t} \| v_{s} \|_{\mC^{1 -
\delta}}^{p} \right] \leqslant C(\| u_{0} \|_{\mC^{\beta}} + 1)^{2 p} \;.
\end{equ} 
Then iterating the previous estimates leads to the desired bound.
\end{proof}

\appendix

\section{Fundamentals of Lyapunov exponents}\label{appendix:Lyapunov}
In this section we outline some basic properties of Lyapunov exponents. 
First, we check an integrability condition on the scalar
$ \varrho $, which is sufficient to guarantee the existence of Lyapunov
exponents. Here we denote with $ \mL (X) $ the set of bounded linear operators
on a Banach space $ X $.
\begin{lemma}\label{lem:integrability}
In the setting of Theorem~\ref{thm:main}, let $ \varrho $ be the solution to
\eqref{e:main} with $ L [u] $ given either by \eqref{e:form-PSA} or
\eqref{e:form-LNS}. Let us write $ S^{t}_{\omega, u_{0}} \varrho_{0} =
\varrho_{t} (\omega) $ for the linear solution map associated to $ \varrho $.
Then
\begin{equ}
\EE  \left[ \log{\| S^{1} \|_{\mL(L^{2}_{*}) }} \right] < \infty  \;.
\end{equ}
\end{lemma}

\begin{proof}
Let us treat only the case of linearised Navier--Stokes, since the pure
advection case is simpler: indeed in that case $ S $ is a contraction, i.e.\ $
\| S^{t}_{\omega,u_0} \|_{\mL (L^{2}_{0})} \leqslant 1 $. Instead for linearised
Navier--Stokes we find
\begin{equ}
\frac{1}{2} \partial_{t} \| \varrho_{t} \|^{2} \leqslant \langle
\varrho_{t},  \Delta u_{t} \cdot \nabla^{-1} \varrho_{t} \rangle \lesssim \|
\varrho_{t} \|^{2} \| u_{t} \|_{\mC^{\beta}} \;,
\end{equ}
for any $ \beta > 2 $. Therefore, for some $ C > 0 $, we find that $ \| \varrho_{t} \| \lesssim \|
\varrho_{0} \| \exp \left( C \int_{0}^{t} \| u_{s} \|_{\mC^{\beta}} \ud s \right) $. The
result follows then from Lemma~\ref{lem:mmt-unif}.
\end{proof}
Next we prove a result that guarantees that the top Lyapunov exponent, as
defined in \eqref{e:lambda-top}, is attainable.

\begin{lemma}\label{lem:exp-decay}
In the same setting as in Theorem~\ref{thm:main}, for every $\ep>0$, there exists a random constant $D_\ep(\omega,u_0)$ satisfying $D_\ep \in (0,1]$, $\P\times \mu$ almost surely, and a closed random subspace $F(\omega,u_0)\subset L^2_{*}$ with $\mathrm{codim}\,F(\omega,u_0) = m \in (0,\infty)$, such that for all $ \varrho_{0} \in L^{2}_{*}$ we have for all $t\geq 0$
\begin{equ}\label{e:exp-decay}
   \|\varrho_t\| \geq  D_\ep(\omega,u_0)\|\varrho_0\| \sin\angle(\varrho_0,F(\omega,u_0))e^{(\lambda^{\kappa}_1-\ep) t}\;,
\end{equ}
where
\[
\sin\angle(\varrho_0,F) = \inf_{\varphi\in
F(\omega,u_0)}\frac{\|\varrho_0-\varphi\|}{\|\varrho_0\|}\;.
\]
\end{lemma}
\begin{proof}
By the Multiplicative Ergodic Theorem (see e.g. \cite{MR4642633} for specific applications to the problems at hand), there exists a closed random subspace $F(\omega,u_0)\subset L^2_{*}$ with $\mathrm{codim}\,(F(\omega,u_0)) = m \in (0,\infty)$ such that for all $ \varrho_{0} \in L^{2}_{*}\backslash F(\omega, u_0)$ we have
\begin{equ}
\lambda_1^\kappa = \lim_{t\to \infty}\frac{1}{t}\log\|\varrho_t\|\;, \qquad
\P\times \mu\text{-almost surely}\;.
\end{equ}
This implies that for every $\ep>0$ there exists an almost surely finite random variable $T_\ep(\omega,u_0)>0$ such that for all $t\geq T_\ep(\omega,u_0)$ we have
\begin{equ}\label{e:large-t-decay}
\|\varrho_t\| \geq \|\pi^\perp_{F}\varrho_0\|e^{(\lambda_1^\kappa-\ep)t}\;,
\end{equ}
where $\pi^\perp_F$ is the orthogonal projection onto the orthogonal complement $F^\perp(\omega,u_0)$. Note that by the Hilbert projection theorem
\begin{equ}
\|\pi^\perp_{F}\varrho_0\| = \|\varrho_0\|\sin\angle(\varrho_0,F(\omega,u_0))
\;.
\end{equ}
We now define 
\begin{equ}
D_\ep(\omega,u_0) := 1\wedge \min_{t\in
[0,T_\ep(\omega,u_0)]}\inf_{\varrho_0\in L^2_*\backslash F(\omega,u_0)\atop
\|\varrho_0\|=1}\frac{\|\varrho_t\|e^{-(\lambda_1^\kappa-\ep)t}}{\sin\angle(\varrho_0,F(\omega,u_0))}\;.
\end{equ}
It is easy to see by backwards uniqueness for $\varrho_t$ that $D_\ep(\omega,u_0)>0$ almost surely and that \eqref{e:exp-decay} holds for all $t\geq 0$ by definition of $D_\ep(\omega,u_0)$ and \eqref{e:large-t-decay}.
\end{proof}

\endappendix

\bibliographystyle{Martin}

\bibliography{refs}

\end{document}